\documentclass{article}
\usepackage{algorithmic,algorithm}
\usepackage{appendix} 
\usepackage{amsmath}
\usepackage{amsthm}
\usepackage{amssymb}
\usepackage{amsfonts}
\usepackage{accents}
\usepackage{statex}
\usepackage{booktabs}
\usepackage{bbding}
\usepackage{bm}
\usepackage{bbm}
\usepackage{caption}
\usepackage{tikz}
\usepackage{dsfont}
\usepackage{float}
\usepackage{mathrsfs}
\usepackage{multirow}
\usepackage{multicol}
\usepackage[colorlinks]{hyperref}  
\usepackage{cleveref}
\usepackage{todonotes}

\usepackage{makecell}
\usepackage{tikz}
\usepackage{pgf}
\usepackage{subfigure}
\usepackage{setspace}
\usepackage{tabularx}
\usepackage{url}
\usetikzlibrary{arrows, decorations.pathmorphing, backgrounds, positioning, fit, petri, automata}
\usetikzlibrary{trees}
\usetikzlibrary{patterns}

\crefname{equation}{}{}
\crefname{algorithm}{Algorithm}{Algorithms}
\crefname{theorem}{Theorem}{Theorems}
\crefname{lemma}{Lemma}{Lemmas}
\crefname{remark}{Remark}{Remarks}
\crefname{figure}{Figure}{Figures}
\crefname{section}{Section}{Sections}
\crefname{subsection}{Subsection}{Subsections}
\usepackage[top=2.5cm, bottom=2.5cm, left=3cm, right=3cm]{geometry}

\numberwithin{equation}{section}

\newtheorem{theorem}{Theorem}[section]

\newtheorem{lemma}[theorem]{Lemma}

\newtheorem{remark}[theorem]{Remark}

\newtheorem{definition}{Definition}[section]

\newcommand{~}{\,\,\,}
\newcommand{\Diag}{\mathbf{Diag}}
\newcommand{\orth}{\mathbf{orth}}
\newcommand{\tr}{\mathbf{tr}}
\newcommand{\kvec}{\mathbf{vec}}

\makeatletter
\def\wideubar{\underaccent{{\cc@style\underline{\mskip10mu}}}}
\def\Wideubar{\underaccent{{\cc@style\underline{\mskip8mu}}}}
\makeatother

\allowdisplaybreaks
\definecolor{Gray}{rgb}{0.5,0.5,0.5}

\begin{document}

\title{Stochastic Gauss-Newton Algorithms for Online PCA}

\author{Siyun Zhou\thanks{School of Mathematical Sciences, University of Electronic Science and Technology of China, China, (\href{mailto:zhousiyun@std.uestc.edu.cn}{zhousiyun@std.uestc.edu.cn}).}
	\and Xin Liu\thanks{State Key Laboratory of Scientific and Engineering Computing, Academy of Mathematics and Systems Science, Chinese Academy of Sciences, and University of Chinese Academy of Sciences, China, (\href{mailto:liuxin@lsec.cc.ac.cn}{liuxin@lsec.cc.ac.cn}). The research was supported in part by the National Natural Science Foundation of China (No. 12125108, 11971466, 11991021, 12021001 and 11688101), Key Research Program of Frontier Sciences, Chinese Academy of Sciences (No. ZDBS-LY-7022), and the Youth Innovation Promotion Association, Chinese Academy of Sciences.}
	\and Liwei Xu\thanks{School of Mathematical Sciences, University of Electronic Science and Technology of China, China, (\href{mailto:xul@uestc.edu.cn}{xul@uestc.edu.cn}). The research was supported in part by the National Natural Science Foundation of China (No. 12071060).}}

\date{} 
\maketitle

\begin{abstract}
	In this paper, we propose a stochastic Gauss-Newton (SGN) algorithm to study the online principal component analysis (OPCA) problem, which is formulated by using
	the symmetric low-rank product (SLRP) model for dominant eigenspace calculation.
	Compared with existing OPCA solvers, SGN is of improved robustness 
	with respect to the varying input data and algorithm parameters.
	In addition, turning to an evaluation of data stream based on approximated objective functions, we develop a new adaptive stepsize strategy for SGN (AdaSGN) which requires no priori knowledge of the input data, and numerically illustrate its comparable performance with SGN adopting the manaully-tuned diminishing stepsize.
	Without assuming the eigengap to be positive, we also establish the global and optimal convergence rate of SGN with the specified stepsize using the diffusion approximation theory. 
	
\end{abstract}


\section{Introduction}\label{sec:intro}
Principal component analysis (PCA) \cite{pearson1901liii} is an orthogonal linear transformation technique primarily used for feature extraction and dimension reduction (see \cite{Jolliffe1986Principal, jolliffe2016principal}), which finds applications in diverse fields (e.g., \cite{raychaudhuri1999principal,subasi2010eeg, turk1991eigenfaces}).
Given a random vector $a\in\mathbb{R}^{n}$ with mean $\mathbb{E}[a] = 0\in\mathbb{R}^n$ and covariance $\mathbb{E}[aa^\top] = \Sigma\in\mathbb{R}^{n\times n}$,
let $\lambda_{i}~(i=1, 2, \cdots, n)$ be the $i$-th largest eigenvalue of $\Sigma$ and 
$u_{i}~(i=1, 2, \cdots, n)$ be the corresponding unit eigenvector.
Suppose the target dimension is $p~(p\ll n)$, the PCA aims to recover a $p$-dimensional subspace spanned by the top $p$ eigenvectors of $\Sigma$, i.e., $u_{1}, u_{2}, \cdots, u_{p}$, and these $p$ eigenvectors are called principal components (PCs).
In a practical context, the sample average approximation (SAA) picks $m$ samples to construct the empirical covariance~$\Sigma_{m}=\sum_{i=1}^{m}a^{(i)}a^{(i)\top}/m$ as an estimate of $\Sigma$. Denote the sample matrix by $A_{m} = [a^{(1)}, a^{(2)}, \cdots, a^{(m)}]\in\mathbb{R}^{n\times m}$.
Then, the key of traditional PCA lies in the top-$p$ eigenvalue decomposition (EVD) of $\Sigma_{m}$ or equivalently the top-$p$ singular value decomposition (SVD) on $A_{m}$.
For the theoretical analysis, we make the following statistical assumption on $A_{m}$:

[A1] (i.i.d.)~\emph{$a^{(1)}, a^{(2)},\cdots ,a^{(m)}$ are independently and identically distributed realizations of the given random vector $a\in\mathbb{R}^{n}$}.

However, in many real-world scenarios, the input data does not arrive
simultaneously, instead, 
it comes in a stream. The dynamic nature of the data demands real-time updates of the estimated PCs.
This inspires PCA in an online setting, which proceeds under the following additional assumption:

[A2] (batchsize)~\emph{the samples are received sequentially, and no more than $h~(h\ll m)$ passes over the data flow in within each round of update.}\\
The estimated PCs need to be updated before the new data enters.
This variant of PCA is known as OPCA (OPCA), or streaming PCA (see \cite{cardot2018online}).

There have been many works in solving the OPCA problem, and most of them are developed from the well-known Oja's iteration~\cite{Oja1982Simplified}.
The direct extension of Oja's iteration, originally designed for $h=1,~p=1$ case, is updated by
\begin{align}
\label{eq:oja}X^{(k+1)}=\orth\left\{X^{(k)}+\frac{\alpha^{(k)}}
{\vert \mathcal{I}^{(k)}\vert}\sum_{i\in \mathcal{I}^{(k)}}a^{(i)}a^{(i)\top}X^{(k)}\right\},
\end{align}
where $\alpha^{(k)}>0$ is the stepsize,
and $\mathcal{I}^{(k)}$ is the index set of the $k$-th data block with the cardinality being smaller than $h$. One active line of related research is adapting Oja's iteration to meet specific requirements, e.g., embedding iterative soft thresholding for sparse PCs \cite{wang2016online}, downsampling for time-dependent data \cite{chen2018dimensionality},
studying modified variant for missing data \cite{balzano2018streaming}, and to name a few.

The research on asymptotic or non-asymptotic analysis of convergence properties for \cref{eq:oja} has been attracting increasing attention in recent years \cite{Allen2017First, balcan2016improved,hardt2014noisy, huang2021streaming, Li2015Rivalry}, and much more progress \cite{Balsubramani2013The,de2015global,feng2018semigroups,jain2016streaming, Li2016Near,li2017diffusion,shamir2016convergence, zhou2020convergence} has been made for the simpler $p=1$ case. 
Among these works, a three-phase analysis of Oja's iteration based on the diffusion approximation is proposed in \cite{li2017diffusion}. The analysis not only provides theoretical guarantees, but also offers a more intuitive streamline along which we could attain a better understanding on the global and local behaviours of Oja's iteration. However, their results are restricted to the circumstance where $p=1$ and $\lambda_{1}>\lambda_{2}$.
The forementioned theoretical studies indicate that in each step, Oja's update admits a constant variance upper bound resulting from the random samples, which leads to an optimal sublinear rate of $\mathcal{O}(1/k)$.
To speed up the convergence, VR-PCA \cite{shamir2015stochastic, shamir2016fast}
applies the variance reduction (VR) technique \cite{Johnson2013Accelerating} to Oja's iteration,
and obtains a faster linear rate which is however limited to the full-sampled case.
Tang \cite{tang2019exponentially} proved the local linear rate of the less-studied Krasulina's method  \cite{krasulina1969method}, whose update formula is closely related to Oja's iteration, but the result is valid only for data of low-rankness.

As being applied to practical computations, Oja-type algorithms 
require much more cares during the determination of stepsize.
The Robbins-Monro conditions \cite{Robbins1951A} suggest that, the stepsizes of Oja's method have to satisfy
$\sum_{k}\alpha^{(k)}=\infty$~and $ \sum_{k}\left(\alpha^{(k)}\right)^2<\infty$ in order to guarantee its convergence to the optimum.
Due to this reason, the commonly-used stepsize takes the form $\alpha^{(k)} = \gamma/(k+1)$.
In this diminishing regime, Oja's method behaves with high sensitivity to $\gamma$, 
and the optimal $\gamma$ has to be scaled with  
some quantities,
which needs to be known in advance and are usually inaccessible in practical implementations, e.g., the eigengap $(\lambda_p-\lambda_{p+1})$.
Some work \cite{lv2006global} attempts to design a new stepsize scheme, and however still involves hyperparameter tuning.
To be completely parameter-free, AdaOja \cite{henriksen2019adaoja} applies the simplified AdaGrad algorithm \cite{ward2019adagrad} designed for SGD to Oja's iteration in a column-wise fashion, taking the update in the form
\begin{align}
\notag b_{i}^{(k)}&=\sqrt{\left(b_{i}^{(k-1)}\right)^2+\left\Vert G^{(k)}_{(:,~i)}\right\Vert_{2}^{2}},\quad i =1,2,\cdots, p,\\
\label{eq:adaoja}X^{(k+1)}&=\orth\left\{X^{(k)}+
G^{(k)}\Diag^{-1}(b^{(k)}_{1},\cdots, b^{(k)}_{p})\right\},
\end{align}
where $G^{(k)}$ is Oja's direction in the $k$-th step, $b_{i}^{(k)}$~and~$G^{(k)}_{(:,~i)}$ denote the $i$-th element of $b^{(k)}$ and $i$-th column of $G^{(k)}$, respectively.
The AdaOja empirically demonstrates the ability of self-adjustment on the stepsize compared with the state-of-art OPCA algorithms.
But the numerical tests in \cite{henriksen2019adaoja} only address the explained variance, which is defined by the percentage of variance recovered from the original dataset, without presenting any results of the error measurement that is based on the canonical subspace angle.

Algorithms for offline PCA have been extensively studied in the past decades.
The conventional eigensolvers are built on Krylov subspace (e.g., \cite{saad1980rates, sleijpen2000jacobi, sorensen1997implicitly, stathopoulos1994davidson}), 
and approximate the eigenvectors in an incremental fashion without fulfilling the demands for parallel scalability and high efficiency at moderate accuracy.
These limitations could be fixed by developing block algorithms, and such algorithms are mostly derived from the following Rayleigh-Ritz trace minimization model over the Stiefel manifold
\begin{equation}
\label{eq:p1}\min_{X\in\mathbb{R}^{n\times p}}~~-\tr\left(X^\top \Sigma_{m} X\right),~~~~~\text{subject~to}~~X^\top X=I_{p}.
\end{equation}
The eigensolvers based on solving \cref{eq:p1} include SSI (see \cite{Golub1996Matrix,rutishauser1970simultaneous}), LOBPCG \cite{knyazev2001toward}, LMSVD \cite{liu2013limited} and ARRABIT \cite{wen2017accelerating}.
Some other works focus on constructing equivalent models 
to \cref{eq:p1} in generating top/bottom eigenspace, e.g., EigPen \cite{wen2016trace} for the trace-penalty minimization model and SLRPGN \cite{liu2015efficient} for the symmetric low-rank product (SLRP) model.
We refer to \cite{gao2019parallelizable, wang2021multipliers, xiao2020class} for recent progress in solving general problems over the Stiefel manifold.

In this paper, taking the advantage of the nonlinear least squares form, 
we focus on the SLRP model which formulates the PCA problem into
\begin{align}
\label{eq:p2}\min_{X\in\mathbb{R}^{n\times p}}~~f(X)=\frac{1}{2}\left\Vert XX^\top-\Sigma_{m}\right\Vert_{\text{F}}^2,
\end{align}
with the corresponding expectational counterpart
\begin{align}
\label{eq:p2E}\min_{X\in\mathbb{R}^{n\times p}}~~f_{\text{E}}(X)=\frac{1}{2}\left\Vert XX^\top-\Sigma\right\Vert_{\text{F}}^2.
\end{align}
To solve the SLRP model, Liu et al. \cite{liu2015efficient} proposed a Gauss-Newton (GN) method of minimum weighted-norm named as SLRPGN, which enjoys a simple explicit update rule and a local linear rate with the constant stepsize $\alpha^{(k)}= 1~(k=0, 1,\cdots)$.

\subsection{Contributions}\label{sec:intro-contri}
For PCA in the online fashion, we propose a new stochastic optimization algorithm named as the stochastic Gauss-Newton (SGN) method.
Numerical experiments on both the simulated and the real data demonstrate that SGN  
is stably effective with different choices of initial points 
when the constant stepsizes are used.
Furthermore, SGN is empirically robust with respect to (w.r.t.) the change of stepsize parameters as well as a variety of input data when the frequently-used diminishing stepsizes are adopted.
In addition, we develop an adaptive stepsize version of SGN named as AdaSGN,
which is based on the consistency of successive batches of online data reflected by the approximate objective functions of SLRP.
The AdaSGN shows much better numerical performance than the state-of-the-art adaptive OPCA algorithm AdaOja \cref{eq:adaoja}, and is comparable to SGN using the manually-tuned diminishing stepsizes.

Regarding to the theoretical aspect, based on the diffusion approximation,
we establish the weak convergence~(or convergence in distribution) of the SGN/properly-rescaled-SGN sequence in the infinitesimal stepsize regime to their corresponding continuous time process limits, which can be modeled by the solution of certain ordinary/stochastic differential equation (ODE/SDE).
Using these differential equation approximations, we obtain the global convergence rate of SGN with both the constant and diminishing stepsizes
for any $p\geq 1$ under the assumption $\lambda_{p}>\lambda_{n}$, which is weaker than the commonly-used positive eigengap assumption $\lambda_{p}>\lambda_{p+1}$.
The error bound for the diminishing stepisze case is optimal in the sense that it exactly matches the minimax lower bound given in \cite{vu2013minimax},
and this result is new among the existing related works.
However, due to an extra $\log$ factor, the bound for the constant stepsize case is only proved to be nearly optimal.

\subsection{Notations}
For given matrix $A\in\mathbb{R}^{n_{1}\times n_{2}}$, $\sigma_{\max}(A)$ and $\sigma_{\min}(A)$ stand for the largest and smallest singular values of $A$, respectively. 
The $(n_{1}n_{2})$-dimensional vector $\kvec(A)$ denotes the vectorization of $A$ which piles columns of $A$ on top of one another. 
The submatrices $A_{(i_{1}:i_{2}, j_{1}:j_{2})}$ and $A_{(:,j)}$ are composed of $i_{1}$-th to $i_{2}$-th rows and $j_{1}$-th to $j_{2}$-th columns of $A$ and the $j$-th column of $A$, respectively. 
Given an event $\mathcal{A}$, we use $\mathbbm{1}_{\{\mathcal{A}\}}$ to denote the indicator function with the value of one if $\mathcal{A}$ occurs, and otherwise the value of zero. 
We denote the $n$-dimensional indentity matrix by $I_{n}$ and the unit vector with the $i$-th coordinate being one and all others being zero by $e_{i}$. We denote $0_{n_{1}\times n_{2}}$ as the $n_{1}\times n_{2}$ zero matrix.
Given two functions $g(\alpha)$ and $h(\alpha)$, the relationships $g \asymp h$ and  $g\lesssim h$ imply $\limsup\limits_{\alpha\to 0}g(\alpha)/h(\alpha)=1$ and  $\limsup\limits_{\alpha\to 0}g(\alpha)/h(\alpha)\leq 1$, respectively.

\subsection{Organization}\label{sec:intro-org}
The rest of this paper is organized as follows.
We introduce the derivation of our proposed GN algorithms~(SGN and AdaSGN) in \cref{sec:alg} and  prove the convergence properties for SGN in the infinitesimal stepsize regime based on It$\hat{\text{o}}$ diffusion theory in \cref{sec:theory}.
The numerical results in \cref{sec:numeri} demonstrate the feasibility and advantages of the proposed GN algorithms.
Finally, a conclusion is made in \cref{sec:conclu}.

\section{Algorithm}\label{sec:alg}
In this section, we first introduce the SLRP model in the online setting, which has been proved to be equivalent to
PCA in the sense of eigen-space computation. We then present the derivation of SGN algorithm for solving the OPCA problem, followed by its adaptive-stepsize version.

\subsection{Online SLRP}\label{sec:alg-oslrp}
In view of the online setting, we first define a batch-$h$ approximation of $\Sigma_{m}$ in the $k$-th step: 
\begin{equation}\label{eq:batch-h}
\Sigma^{(k)}_{h}= \frac{1}{h}A^{(k+1)}_{h}
A^{(k+1)\top}_{h} = \frac{1}{h}\sum_{i=kh+1}^{kh+h}a^{(i)}a^{(i)\top},
\end{equation}
where the dynamic data block is given by
\begin{align}
\label{eq:samplebatch} A^{(k+1)}_{h}=[a^{(kh+1)}\quad a^{(kh+2)} ~\cdots~a^{(kh+h)}].
\end{align} 

\noindent Then \cref{eq:batch-h} yields the corresponding batch-$h$ approximation of \cref{eq:p2} in the $k$-th step:
\begin{align}
\label{eq:p2k}
\min_{X\in\mathbb{R}^{n\times p}}~~ \hat{f}^{(k)}(X)= \frac{1}{2}\left\Vert XX^\top-\Sigma^{(k)}_{h}\right\Vert_{\text{F}}^{2}.
\end{align}
The gradient of $\hat{f}^{(k)}(X)$ is given by
$$
\nabla \hat{f}^{(k)}(X) = 2\left(X(X^\top X)-\Sigma_{h}^{(k)}X\right).
$$

\subsection{SGN for OPCA}\label{sec:alg-sgn}
For simplicity, let $R^{(k)}(X)=(XX^\top-\Sigma^{(k)}_{h})$ be the residual in the $k$-th step.
The GN direction for the nonlinear least squares problem \cref{eq:p2k}, denoted as $S^{(k)}(X)$, could be obtained by solving the normal equation:
\begin{align}
\label{eq:normal}J(X)^\top J(X)(S)=-J(X)^\top R^{(k)}(X),
\end{align}
where $J(X): \mathbb{R}^{n\times p}\rightarrow \mathbb{R}^{n\times n}$ is the Jacobian operator of $R^{(k)}(X)$ at $X$, and $J(X)^\top: \mathbb{R}^{n\times n} \to \mathbb{R}^{n\times p}$ is the adjoint operator of $J(X)$ at $X$.
The rank deficiency of $J(X)$ makes the linear system formed by \cref{eq:normal} admit an infinite number of solutions, but fortunately, the solutions could be expressed explicitly (see Proposition 3.1 of \cite{liu2015efficient}).
In the same way as \cite{liu2015efficient}, we choose the one of minimum weighted-norm from the solution set of \cref{eq:normal} as our SGN direction:
\begin{equation}
\label{eq:direction}
S^{(k)}(X)
=\Sigma^{(k)}_{h}X(X^\top X)^{-1}-\frac{1}{2}X-\frac{1}{2}X(X^\top X)^{-1}X^\top \Sigma^{(k)}_{h}X(X^\top X)^{-1},
\end{equation}
which could be formulated into a three-step procedure: for $X = X^{(k)},~A = A^{(k+1)}_h$,
\begin{equation}
\label{eq:sgndir}
P = X(X^\top X)^{-1},~Q = \left. A^{\top}P\middle /\sqrt{h}\right.,~S^{(k)}(X) = \left. AQ\middle /\sqrt{h}\right.-\left. X(I_p+Q^\top Q)\middle/2\right..
\end{equation}

The SGN approach for OPCA is summarized in \cref{alg:sgn}.
\begin{algorithm}
	\caption{A SGN method for top-$p$ OPCA}
	\label{alg:sgn}
	\begin{algorithmic}[1]
		\REQUIRE{Samples $a^{(1)}, a^{(2)}, \cdots$; target dimension $p$; stepsizes $\alpha^{(k)}$; batchsize $h$.}
		\ENSURE{Choose rank-$p$ $\hat{X}^{(0)}\in\mathbb{R}^{n\times p}$ randomly and set $X^{(0)}=\orth\{\hat{X}^{(0)}\}$.}
		\FOR{$k=0, 1, 2, \cdots$}
		\STATE{Update the sample batch $A_{h}^{(k+1)}$ by \cref{eq:samplebatch}.}\label{alg1:step2}
		\STATE{Compute $S^{(k)}(X^{(k)})$ by \cref{eq:sgndir}.}\label{alg1:step3}
		\STATE{Iterate $X^{(k+1)}= X^{(k)}+\alpha^{(k)} S^{(k)}(X^{(k)})$.}\label{alg1:step4}
		\ENDFOR
	\end{algorithmic}
	{\bf Output:} $\orth\{X^{(k)}\}$.
\end{algorithm}

\subsection{AdaSGN for OPCA}\label{sec:alg-adasgn}
As being seen from \cref{eq:adaoja}, the AdaOja uses the stochastic gradients obtained from a single batch to adjust stepsizes, which is a widely used strategy in online algorithms.
Taking into account the limitation of incomplete knowledge from a single batch, our adaptive stepsize scheme focuses on the consistency between successive sample batches, which is characterized by the approximated objective function $\hat{f}^{(k)}(X)$ given by \cref{eq:p2k}.
Since $X^{(k)}$ is derived from the subproblem of minimizing~$\hat{f}^{(k-1)}(X)$, it is clear that
\begin{align} 
\label{eq:consis}
\hat{f}^{(k-1)}(X^{(k)})\leq \hat{f}^{(k-1)}(X^{(k-1)}),
\end{align} 
always holds.
If the new sample batch $A_{h}^{(k+1)}$ maintains consistency with history samples in the sense of \cref{eq:consis}, i.e., satisfying $\hat{f}^{(k)}(X^{(k)})\leq \hat{f}^{(k)}(X^{(k-1)})$, we make the stepsize larger or unchanged; otherwise we decrease the stepsize. 	

To describe the extent of consistency, we define a parameter-free indicator $r^{(k)}$ as
\begin{equation}
\label{eq:adaindicator}
r^{(k)}= 
\begin{cases}
\left. \hat{f}^{(k)}(X^{(k-1)}) \middle/\hat{f}^{(k)}(X^{(k)})\right.	& \hat{f}^{(k)}(X^{(k)})> \hat{f}^{(k)}(X^{(k-1)})\\
\quad\quad\quad \quad\quad\quad 0& \text{otherwise}
\end{cases}
\end{equation}
for $k>0$ and $r^{(0)} = 1$. 
Then, we set the stepsize by
\begin{equation}
\label{eq:ada}
\alpha^{(k)}= 
\begin{cases}
\left. r^{(k)} \middle/\sum^{k}_{i=0} r^{(i)}\right.\quad\quad& \hat{f}^{(k)}(X^{(k)})> \hat{f}^{(k)}(X^{(k-1)})\\
\quad\left.1 \middle/\sum^{k}_{i=0} r^{(i)}\right.\quad\quad& \text{otherwise}
\end{cases}
\quad\in[0, 1).
\end{equation}
For the poorly-consistent sample case where $\hat{f}^{(k)}(X^{(k)})> \hat{f}^{(k)}(X^{(k-1)})$, smaller $r^{(k)}$ implies less consistent sample, thus we provide a smaller $\alpha^{(k)}$ according to~$a_{1}/(c+a_{1})>a_{2}/(c+a_{2})$~for $a_{1}>a_{2}>0,~c>0$. 
For the well-consistent sample case where $\hat{f}^{(k)}(X^{(k)})\leq \hat{f}^{(k)}(X^{(k-1)})$, our scheme makes AdaSGN yield no monotonic decrease in stepsize but an overall decreasing trend. This  strategy has advantages in such circumstance as many extremely bad samples are received in the early stage.

The SGN using the stepsize scheme stated above is generalized in \cref{alg:adasgn}.

\begin{algorithm}[h]
	\caption{AdaSGN for OPCA} 
	\label{alg:adasgn}
	\begin{algorithmic}[1]
		\REQUIRE{
			Samples $a^{(1)}, a^{(2)}, \cdots$; 
			target dimension $p$; batchsize $h$.}
		\ENSURE{Choose rank-$p$ $\hat{X}^{(0)}\!\in\!\mathbb{R}^{n\times p}$ randomly and set $X^{(0)}\!\!=\!\orth\{\hat{X}^{(0)}\!\}$, $r^{(0)}\!=\!1.\!$}
		\FOR{$k=0, 1, 2, \cdots$}
		\STATE{Update $A_{h}^{(k+1)}$ according to \cref{eq:samplebatch}.}
		\IF{$k>0$}
		\STATE{Compute $\hat{f}^{(k)}(X^{(k-1)}),~ \hat{f}^{(k)}(X^{(k)})$ by \cref{eq:p2k}, and $r^{(k)}$  by \cref{eq:adaindicator}.}
		\ENDIF
		\STATE{Update the stepsize $\alpha^{(k)}$ by \cref{eq:ada}.}
		\STATE{Perform Step \ref{alg1:step3} and Step \ref{alg1:step4} in \cref{alg:sgn}.}
		\ENDFOR
		\STATE{{\bf Output:} $\orth\{X^{(k)}\}$.}
	\end{algorithmic}
\end{algorithm}

\section{Theory of algorithm}\label{sec:theory}
In this section, 
since the sequence $\{X^{(k)}\}$ generated by SGN algorithm constitutes a discrete-time Markov process,
we use ODE/SDEs to characterize the SGN sequence with infinitesimal stepsize in the weak sense,
and prove its global convergence rate based on these differential equations.
We begin with some basic concepts and analytical techniques in \cref{sec:theory-pre}. In  \cref{sec:theory-main}, we present the convergence results of SGN in the constant stepsize regime, along with detailed proofs.
The last subsection is devoted to investigating the diminishing stepsize case.

\subsection{Preliminaries}\label{sec:theory-pre}
We firstly revisit some basic definitions.
\begin{definition}[Sub-Gaussian random vector]
	A random vector $v\in\mathbb{R}^n$ is said to follow a sub-Gaussian distribution if there exists a constant $\sigma^2>0$~(called variance proxy) such that for any $s\in\mathbb{R}$ and unit vector $c\in\mathbb{R}^n$, the following inequality holds
	\begin{equation*}
	\mathbb{E}\exp\left\{s(c^\top v-c^\top\mathbb{E}v)\right\}\leq\exp\left\{\sigma^2 s^2/2\right\}.
	\end{equation*}
\end{definition}
In order to carry out the analysis, 
we assume that 

[A3] (sub-Gaussian distribution)~\emph{$a\in\mathbb{R}^n$ is a sub-Gaussian random vector.}

\noindent It is known that a sub-Gaussian vector has finite moments of any order. This property will be used to prove the expectational boundedness and full-rankness of SGN iterates in \cref{lemma:bound2} and \cref{lemma:fullrank}.

We assume for distinguishing the top PCs from the remaining ones that

[A4]~\emph{$\lambda_{p}>\lambda_{n}$.}

\begin{definition}\label{def:GTD}
	Under the assumption [A4],
	we define $p'$ as the smallest index satisfying $\lambda_{p'+1}<\lambda_{p}$. Namely,
	$$p'=\min\{i~\vert~p\leq i\leq n-1,~\lambda_{i+1}<\lambda_{p}\}.$$
	We denote $\nu = \lambda_{p}-\lambda_{p'+1}$.
\end{definition}

\noindent 
Note that the assumption [A4] covers a frequently-used requirement $\lambda_{p}>\lambda_{p+1}$ for the analysis of OPCA algorithms when $p'$ coincides with
$p$.
Given $X\in\mathbb{R}^{n\times p}$, let $\bar{X}= X_{(1:p',~:)}\in\mathbb{R}^{p'\times p}$ and $\wideubar{X} = X_{((p'+1):n,~:)}\in\mathbb{R}^{(n-p')\times p}$.

\begin{definition}[Canonical angles between subspaces] 
	Let $X_1\in\mathbb{R}^{n\times p_{1}}$ and $X_2\in\mathbb{R}^{n\times p_{2}}~(p_1\leq p_2\leq n)$ be two matrices with orthonormal columns. The canonical or principal angles between two subspaces spanned by the columns of $X_1$ and $X_2$ are 
	\begin{align*}
	\theta_{i} = \arccos \sigma_{i}\in[0, \pi/2],\quad i = 1, 2, \cdots, p_1,
	\end{align*}
	where $\sigma_{i}$~are singular values of $X_1^\top X_2$. 
	Denote $\Theta(X_1, X_2) = \Diag(\theta_{1}, \cdots, \theta_{p_1})$ and $\sin\Theta(X_1, X_2) = \Diag(\sin\theta_{1}, \cdots, \sin\theta_{p_1})$.
\end{definition}
The error of the $k$-th iterate $X^{(k)}$ is measured by
\begin{align}\label{eq:measure} \Vert\sin\Theta(X^{(k)}, U_{p'})\Vert_{\text{F}}^{2}=p-\Vert U_{p'}^\top \orth\{X^{(k)}\}\Vert_{\text{F}}^{2},
\end{align}
where $U_{p'} = [u_{1}, u_{2}, \cdots, u_{p'}]$.
Without loss of generality, the covariance is assumed to be of a diagonal form $\Sigma = \Diag(\lambda_{1},\lambda_{2},\cdots,\lambda_n)$
hereinafter, and the corresponding error measurement becomes 
\begin{equation}\label{eq:error2}
\Vert\sin\Theta(X^{(k)}, U_{p'})\Vert_{\text{F}}^{2}=p-\tr\left(\bar{X}^{(k)\top}\bar{X}^{(k)}(X^{(k)\top} X^{(k)})^{-1}\right).
\end{equation}

As mentioned earlier, the SGN sequence $\{X^{(k)}\}$ could be viewed as a discrete-time Markov process.
The basis of the theory of SGN is the well-known diffusion approximation technique, 
of which the fundamental idea is to characterize the given analytically intractable stochastic process by some certain diffusion processes (modeled by the solution to some SDEs) with sufficiently good and useful properties.

Considering the case of $p=1$~and~$\alpha^{(k)}= \alpha~(k=0, 1,\cdots)$, where the SGN update reads
\begin{equation}\label{eq:sgn_sde}
X^{(k+1)} = X^{(k)}+\alpha S_{\text{E}}(X^{(k)})
+\sqrt{\alpha} \left[\sqrt{\alpha}\left(S^{(k)}(X^{(k)})-S_{\text{E}}(X^{(k)})\right)\right],
\end{equation}
with $S_{\text{E}}(X^{(k)}) = \mathbb{E}[S^{(k)}(X^{(k)})]$. 
By taking $\alpha = \Delta t$ and comparing formula \cref{eq:sgn_sde} with the Euler discretization
\begin{align*}
X^{(k+1)} = X^{(k)}+\Delta t b(X^{(k)})+\sqrt{\Delta t}\sigma(X^{(k)})\xi^{(k)},\quad \xi^{(k)}\sim\mathcal{N}(0, I_n),
\end{align*}
of some SDE in the form
\begin{align}\label{eq:ito}
dX(t)=b(X(t))dt+\sigma(X(t))dB(t),
\end{align}
we could then naturally expect \cref{eq:ito} with certain drift coefficient $b(X)$ and diffusion coefficient $\sigma(X)$ to be a distributional approximation of the discrete SGN sequence.

As a formalization of the above intuitive statement, Corollary 4.2 in Section 7.4 of \cite{ethier1986markov} suggests that the main issues in the convergence analysis (of the constant stepsize case) consist of
(i)~constructing a continuous-time process
\begin{equation*}
X_{\alpha}(t)= X^{(\lfloor t\alpha^{-1}\rfloor)},
\end{equation*} 
from the discrete-time SGN sequence $\{X^{(k)}\}$ using the constant stepsize $\alpha^{(k)} = \alpha\in(0, 1]$ by piecewise constant interpolation;
(ii)~finding the limiting SDEs \cref{eq:ito} of vectorized $X_{\alpha}(t)$ in distribution as $\alpha\to 0$
via calculating the infinitesimal mean $b$ and variance $a=\sigma\sigma^\top$ defined by
\begin{align}\label{eq:infinitemean}
b(X)&=\frac{d}{dt}\mathbb{E}[\kvec(X(t))] =\lim_{\alpha\to 0} \frac{1}{\alpha}\mathbb{E}[\kvec(\Delta X_{\alpha}(t))\vert X_{\alpha}(t)=X],\\
\label{eq:infinitevariance} 
a(X)&=\frac{d}{dt}\text{Var}[\kvec(X(t))]=\!\lim_{\alpha\to 0} \frac{1}{\alpha} \mathbb{E}[\kvec(\Delta X_{\alpha}(t))\kvec(\Delta X_{\alpha}(t))^\top\vert X_{\alpha}(t)=X],
\end{align}
where $\kvec(\Delta X_\alpha(t)) = \kvec(X_\alpha(t+\alpha)-X_\alpha(t))$;
(iii)~using the derived SDE approximation to learn the properties of SGN iterates.

\subsection{Main results on constant stepsizes}\label{sec:theory-main}
We first provide the main results of the constant stepsize case in \cref{th:odeconvergence} and \cref{th:rate}, and delay their proofs until later in this subsection.

The following theorem states the global convergence of SGN with the constant stepsize.
\begin{theorem}\label{th:odeconvergence}
	Under the assumptions [A1]-[A4],
	the process $\{X^{(\lfloor t\alpha^{-1}\rfloor)}\}$ generated by \cref{alg:sgn} with the stepsize $\alpha^{(k)}= \alpha~(k=0, 1,\cdots)$ converges in distribution to the solution of
	\begin{align}\label{eq:ode}	\frac{dX}{dt}=\Sigma X(X^\top X)^{-1}-\frac{1}{2}X-\frac{1}{2}X(X^\top X)^{-1}X^\top\Sigma X(X^\top X)^{-1},
	\end{align}
	as $\alpha\to 0$ with an initial value $X(0)=X^{(0)}$. 
	Furthermore, the solution $X(t)$ to \cref{eq:ode} satisfies
	\begin{align*}
	\lim\limits_{t\to\infty}\nabla f_{\text{E}}(X(t)) = 0.
	\end{align*}
\end{theorem}

\vspace{-0.7cm}
\begin{figure}[H]
	\centering
	\mbox{
		\hspace{-0.5cm}
		\subfigure[$\lambda_{1} = 1.1$]{
			\includegraphics[width=4cm]{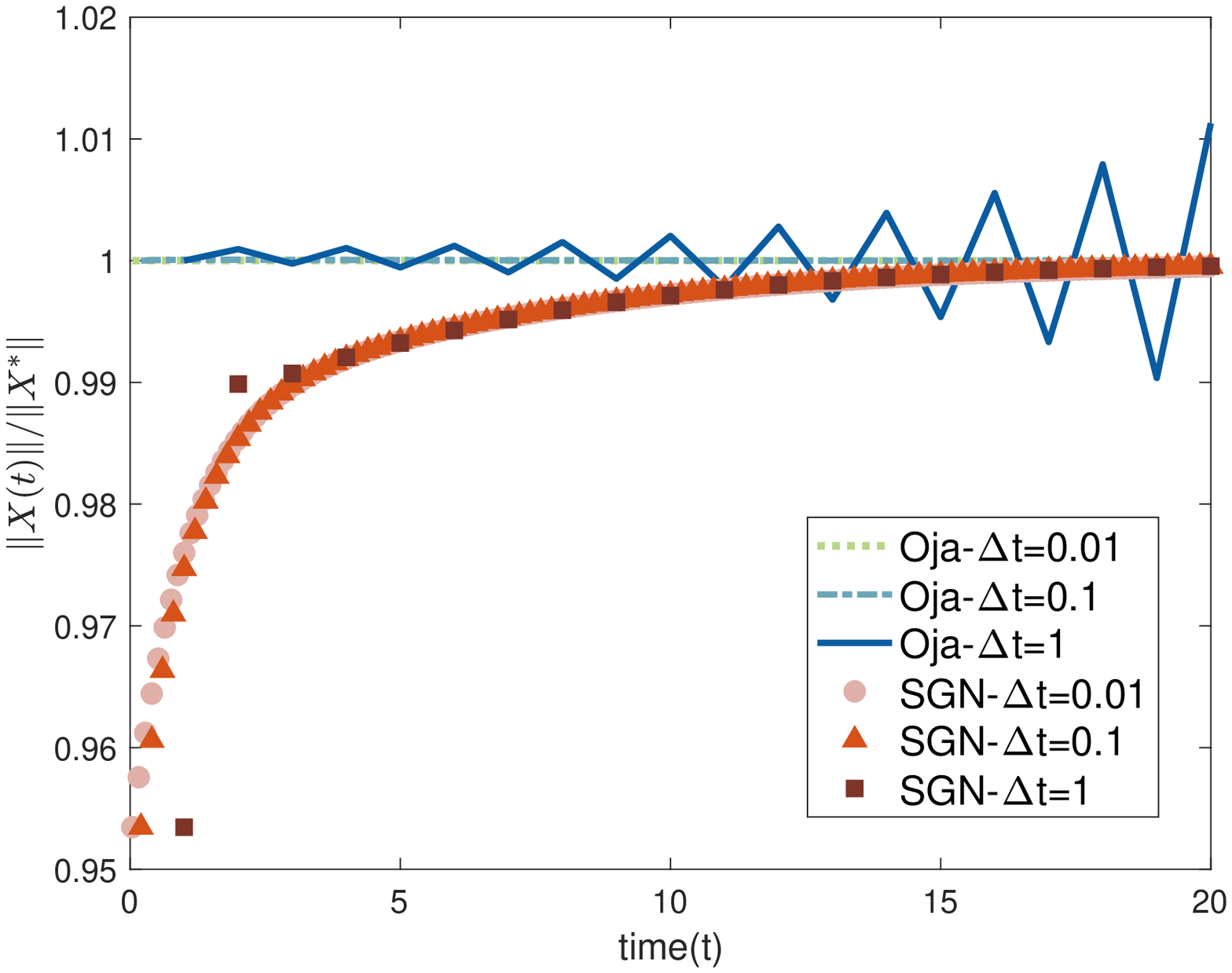}}
		\hspace{-0.5cm}
		\subfigure[$\lambda_{1} = 11$]{
			\includegraphics[width=4cm]{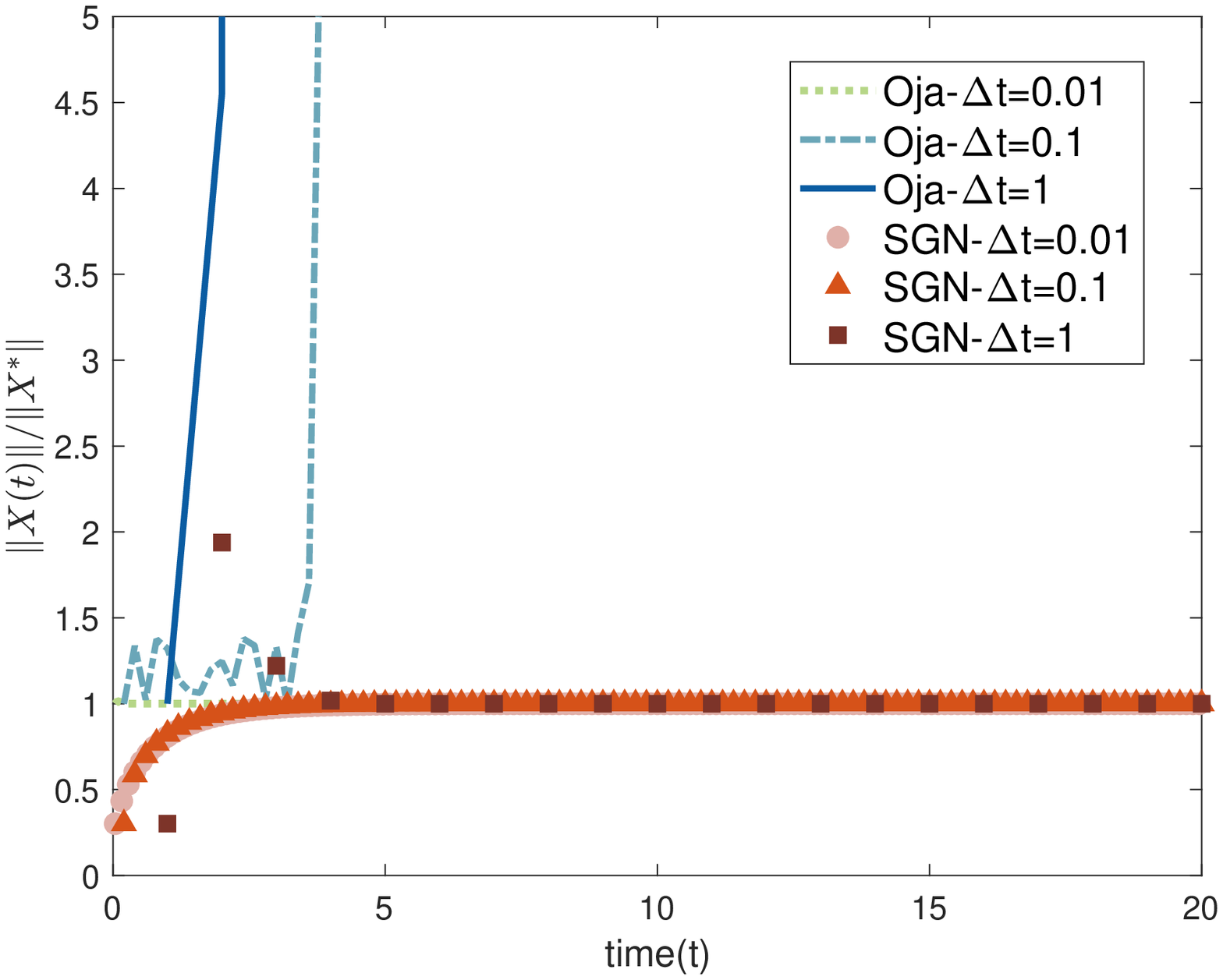}}
		\hspace{-0.5cm}
		\subfigure[$\lambda_{1} = 101$]{
			\includegraphics[width=4cm]{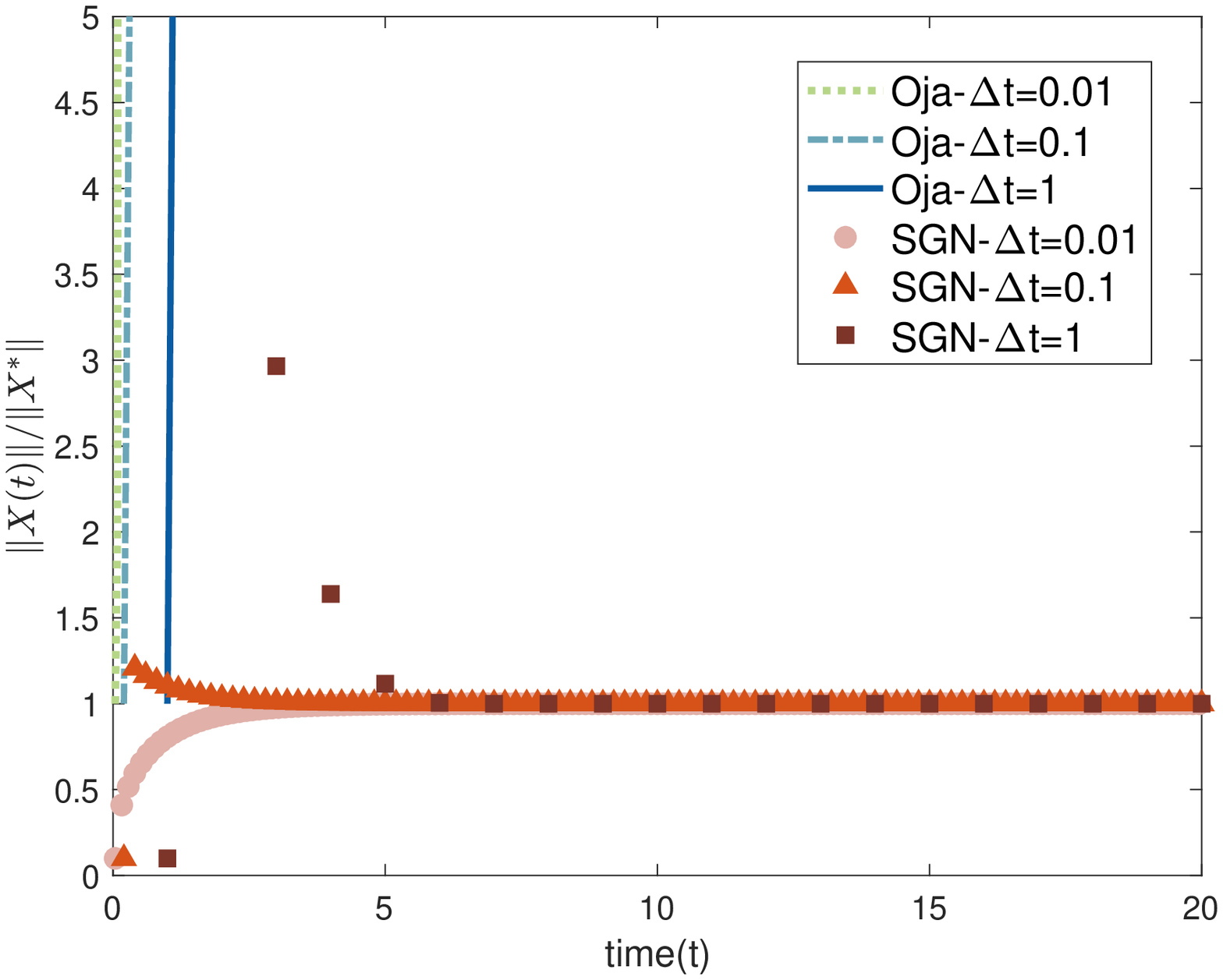}}
		\hspace{-0.5cm}
		\subfigure[$\lambda_{1} = 1001$]{
			\includegraphics[width=4cm]{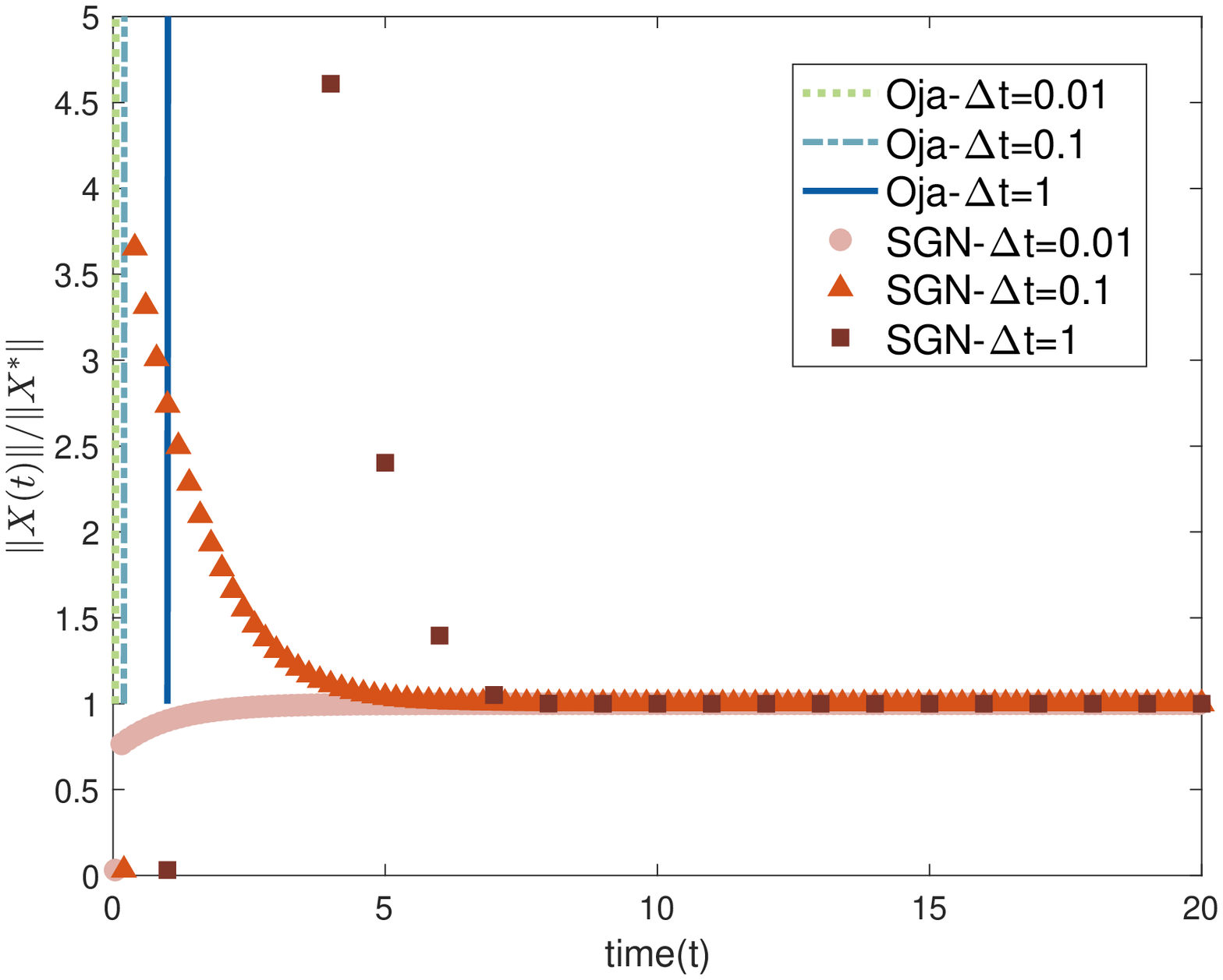}}
		\hspace{-0.5cm}
	}
	\label{fig:ode}
	\vspace{-0.3cm}
	\caption{Finite difference discretization of the limiting ODE for Oja's iteration (Theorem 3.1 in \cite{li2017diffusion}) and the one for SGN \cref{eq:ode} with various stepsize $\Delta t$~($n=500,~p=1,~ \lambda_{2}=\cdots = \lambda_{500}=1$).}
\end{figure}
Using the classic forward difference approximations, we provide a simple comparison between the discretization of the limiting ODE for Oja's iteration (Theorem 3.1 in \cite{li2017diffusion}) and the one of \cref{eq:ode} for SGN in \cref{fig:ode}, where the numerical instability of Oja's method could be apparently observed as the top eigenvalue $\lambda_{1}$ increases.

Then, we present the convergence rate of SGN in terms of the expectational error.
\begin{theorem}\label{th:rate}
	Suppose that the assumptions [A1]-[A4] hold.
	Given the total sample size $m>0$, and the stepsize
	\begin{equation}\label{eq:fixedstepsize}
	\alpha^{(k)}= \alpha=\frac{\lambda_{p}}{\nu}K^{-1}\log K,\quad~K = \left\lceil\frac{m}{h}\right\rceil,\quad~k=0, 1,\cdots, K,
	\end{equation}
	the expectational error of the final output $X^{(K)}$ of \cref{alg:sgn} is given by
	\begin{align}\label{eq:rate}
	\mathbb{E}\Vert\sin\Theta(X^{(K)}, U_{p'})\Vert_{\text{\rm F}}^{2}&\lesssim \frac{C}{h}\frac{\lambda_{p}\log K}{\nu K}\sum_{l=p'+1}^{n}\sum_{r=1}^{p}\frac{\lambda_{l}}{2(\lambda_{r}-\lambda_{l})}\\
	\notag&\leq \frac{C}{2h}p(n-p')\frac{\lambda_{p}\lambda_{p'+1}}{\nu^2}K^{-1}\log K,
	\end{align}
	where $C$ is some positive constant.
\end{theorem}

\begin{remark}[Comparison with previous work]\label{remark: comparison}
	When taking $p=p'=1$, our error bound \cref{eq:rate} is coincident with the bound of \cite{li2017diffusion} for Oja's iteration, where the stepsize is chosen to be $\alpha^{(k)}=\alpha = \nu^{-1}K^{-1}\log K$. This suggests that when $\lambda_{1}$ is large, our SGN allows for a larger stepsize to achieve a fixed order of accuracy.
\end{remark}

\begin{remark}[Constant variance upper bound]\label{remark:variance}
	As the update goes to the optimum, for $X=X^{(k)},~\Sigma_c=\Sigma^{(k)}_{h}-\Sigma$, the variance statistics \cite{tropp2015introduction} of one-step Oja's iteration and SGN in the Frobenius norm sense in the $k$-th update are
	\begin{align}
	\label{eq:var_oja}v\left(X^{(k+1)}_{\text{Oja}}\right) &=(\alpha^{(k)})^2\left\Vert X^\top \mathbb{E}[\Sigma_c^2] X\right\Vert_{\text{F}}\to (\alpha^{(k)})^2\left\Vert U_{p}^\top \mathbb{E}\Sigma_c^2 U_{p}\right\Vert_{\text{F}},\\
	v\left(X^{(k+1)}_{\text{SGN}}\right)\notag&= (\alpha^{(k)})^2\left\Vert \mathbb{E}\left[P^\top \Sigma_c\left(I-3PX^\top/4\right)\Sigma_cP\right]\right\Vert_{\text{F}}\\
	\label{eq:var_sgn}&\leq (\alpha^{(k)})^2\left\Vert P^\top \mathbb{E}[\Sigma_c^2] P\right\Vert_{\text{F}}\to(\alpha^{(k)})^2\left\Vert \bar{\Lambda}^{-\frac{1}{2}}U_{p}^\top \mathbb{E}[\Sigma_c^2] U_{p}\bar{\Lambda}^{-\frac{1}{2}}\right\Vert_{\text{F}},
	\end{align}
	where $\bar{\Lambda}=\Diag(\lambda_{1},\cdots, \lambda_{p}),~U_{p}=[u_{1},\cdots, u_{p}]$ and $P$ is given by \cref{eq:sgndir}. 
	It could be easily verified from \cref{eq:var_oja} and \cref{eq:var_sgn} that both of SGN and Oja's iteration admit a constant variance upper bound unless the stepsize $\alpha^{(k)}$ vanishes, 
	and therefore neither of them could achieve a linear rate. However, unlike the Oja's iteration, the bound of SGN involves an inverse top-eigenvalue matrix $\bar{\Lambda}$ as the iterate approaches the optimum,
	which we could take as an
	advantage in the adaptation to different datasets.
\end{remark}

%

The followings are two technical lemmas that are necessary for the proofs of \cref{th:odeconvergence} and \cref{th:rate}.
\begin{lemma}\label{lemma:kT1} 
	Under the assumptions [A1]-[A4],
	the process $X_{\alpha}(t)= X^{(\lfloor t\alpha^{-1}\rfloor)}$ generated by \cref{alg:sgn} with the stepsize $\alpha^{(k)}= \alpha$~$(k=0,1,\cdots)$ converges in distribution to the solution of \cref{eq:ode} as $\alpha\to 0$
	with $X(0)=X^{(0)}$.
\end{lemma}

\begin{remark}
	Since the diffusion coefficient of $X(t)$ is $\sigma(X) = \mathcal{O}(\alpha^{\frac{1}{2}})$, we obtain an ODE \cref{eq:ode} for the global dynamics of SGN updates omitting the stochasticity. 
	However, the drift term of \cref{eq:ode} vanishes and its diffusion term dominates instead
	when the iterate gets close to some stationary point $\tilde{X}\in\mathbb{R}^{n\times p}$ of \cref{eq:p2E}. For purposes of capturing local dynamics near $\tilde{X}$, we introduce a new Markov process via rescaling the original SGN updates by a factor of $\alpha^{-\frac{1}{2}}$ and get a new SDE in \cref{lemma:kT2}.
\end{remark}

It is known from the first order optimality condition of \cref{eq:p2E}
\begin{equation}\label{eq:firstorder}
\Sigma X = XX^\top X
\end{equation}
that a rank-$p$ stationary point $\tilde{X}$ spans an invariant subspace of $\Sigma$, which could also be spanned by $p$ eigenvectors of $\Sigma$.
Based on this, we equip each $\tilde{X}$ with an eigen-index set $\mathcal{I}_{p}=\{i_{1},\cdots, i_{p}\}\subset \{1, \cdots, n\}$, which indicates that $\tilde{X}$ has $p$ singular values $\lambda_{i}^{\frac{1}{2}}~(i\in \mathcal{I}_{p})$.
Given that the multiplicity of $\lambda_{i}~(i\in \mathcal{I}_{p})$ may not be one, we further define an extended eigen-index set $\hat{\mathcal{I}}_p$ that contains all indexes sharing the same eigenvalue as each one in the original eigen-index set $\mathcal{I}_{p}$, i.e., satisfying
(i)~$\mathcal{I}_{p}\subseteq\hat{\mathcal{I}}_p$;
(ii)~$\lambda_{i} = \lambda_{j}$ for any $i\in\hat{\mathcal{I}}_{p}$ and some $j\in \mathcal{I}_{p}$;
(iii)~$\lambda_{i}\neq \lambda_{j}$ for any $i\in\hat{\mathcal{I}}_{p},~j\notin\hat{\mathcal{I}}_{p}$.

Then, we have the following result.
\begin{lemma}\label{lemma:kT2}
	Suppose that the assumptions [A1]-[A4] hold and the initial point is around some rank-$p$ stationary point of \cref{eq:p2E}, which is associated with an eigen-index set~$\mathcal{I}_{p}=\{i_{1}, \cdots, i_{p}\}$
	and an orthogonal matrix~$V\in\mathbb{R}^{p\times p}$. We define a new process
	\begin{equation}\label{eq:z}
	Y_\alpha(t)=\alpha^{-\frac{1}{2}}X_\alpha(t)V= \alpha^{-\frac{1}{2}}X^{(\lfloor t\alpha^{-1}\rfloor)}V,
	\end{equation}
	generated by \cref{alg:sgn} with the stepisze $\alpha^{(k)}= \alpha~(k=0,1,\cdots)$.
	If $Y_{\alpha}(0)$ converges in distribution to some constant $Y^{0}\in\mathbb{R}^{n\times p}$ as $\alpha\to 0$ and the index $l$ is not in the extended eigen-index set $\hat{\mathcal{I}}_{p}$, then the $(l, r)$-th entry of $Y_{\alpha}(t)$ converges in distribution to the following SDE with $Y_{(l, r)}(0) = Y^{0}_{(l, r)}$
	\begin{equation}\label{eq:ksde}
	dY_{(l, r)}=\left(\Sigma Y\Lambda_{p}^{-1}-Y\right)_{(l, r)}dt+h^{-\frac{1}{2}}\lambda_{l}^{\frac{1}{2}}dB(t),\quad \quad 1\leq r\leq p,
	\end{equation}
	where $\Lambda_{p}=\Diag(\lambda_{i_{1}}, \lambda_{i_{2}}, \cdots, \lambda_{i_{p}})\in\mathbb{R}^{p\times p}$ and $B(t)$ is the Brownian motion.
\end{lemma}

\noindent The proofs of the above two lemmas are provided in the appendix. 

We are now ready to prove \cref{th:odeconvergence} and \cref{th:rate} as follows.

\begin{proof}[Proof of \cref{th:odeconvergence}]
	We first decompose the right-side of \cref{eq:ode} into two orthogonal components $P_{1}(X), P_{2}(X)$ as
	\begin{equation*}	
	\frac{dX}{dt} = P_{1}(X)+P_{2}(X),\quad P_{1}(X)=\Sigma P-XP^\top\Sigma P,\quad
	P_{2}(X)=\frac{1}{2}XP^\top\Sigma P-\frac{1}{2}X,
	\end{equation*}
	where $P=X(X^\top X)^{-1}$. Thus, the gradient of $f_{\text{E}}(X)$ could be expressed by
	\begin{align*}
	\nabla f_{\text{E}}(X)=2(XX^\top X-\Sigma X)=-2(P_{1}+2P_{2})(X^\top X).
	\end{align*}
	We then obtain a new ODE
	\begin{align*}
	\frac{df_{\text{E}}}{dt}=\tr\left((\nabla f_{\text{E}}(X))^\top\frac{dX}{dt}\right)&=-2\tr\left[(X^\top X)(P_{1}+2P_{2})^\top(P_{1}+P_{2})\right]\\
	&=-2\left(\Vert  P_{1}X^\top\Vert_{\text{F}}^{2}+2\Vert  P_{2}X^\top\Vert_{\text{F}}^{2}\right)\leq 0,
	\end{align*}
	which demonstrates the non-increasing property of the residual function $f_{\text{E}}(X)$. When both $\Vert P_{1}X^\top\Vert_{\text{F}}^{2}=0$ and $\Vert  P_{2}X^\top\Vert_{\text{F}}^{2}=0$ hold, we obtain 
	\cref{eq:firstorder},
	which indicates that the asymptotic solution of \cref{eq:ode} is the stationary point of \cref{eq:p2E}. Finally, applying \cref{lemma:kT1} immediately completes the proof of this theorem.
\end{proof}

\begin{remark}
	If $p'=p$ and $\bar{X}(t)$ is invertible, we could further conclude that the solution $X(t)$ to \cref{eq:ode} converges to the global optimum of \cref{eq:p2E}, that is
	\begin{equation*}
	\lim_{t\to\infty}X(t)=\left[\bar{\Lambda}^{\frac{1}{2}}\quad 0_{p\times(n-p)}\right]^\top V^\top,
	\quad\lim_{t\to\infty}\bar{X}(t)=\bar{\Lambda}^{\frac{1}{2}}V^\top,
	\end{equation*}
	where $V\in\mathbb{R}^{p\times p}$ is an orthogonal matrix. This indicates that for any $\varepsilon>0$, there exists a constant $T>0$ such that for all $t>T$
	\begin{equation}
	\label{eq:limit}
	\left\Vert X(t)-\left[\bar{\Lambda}^{\frac{1}{2}}\quad 0_{p\times(n-p)}\right]^\top V^\top\right\Vert_{\text{F}}<\varepsilon,\quad 
	\Vert \bar{X}(t)-\bar{\Lambda}^{\frac{1}{2}}V^\top\Vert_{\text{F}}<\varepsilon.
	\end{equation}
	Let $r_{i} = e_{i}^\top X(\bar{X}^\top \bar{X})^{-1}X^\top e_{i},~(p<i\leq n)$, we can obtain from  \cref{eq:ode} that
	\begin{equation*}
	\begin{aligned}
	\frac{dr_i}{dt}&=2e_{i}^\top\Sigma X(X^\top X)^{-1}(\bar{X}^\top\bar{X})^{-1}X^\top e_{i}-2e_{i}^\top X\bar{X}^{-1}\bar{\Lambda}\bar{X}(X^\top X)^{-1}(\bar{X}^\top\bar{X})^{-1}X^\top e_{i}\\
	&\leq\left\{
	\begin{array}{rcl}
	&2(\lambda_{i}-\lambda_{p})\lambda_{1}^{-1}r_{i}, & {\quad p=1,}\\
	&2(\lambda_{i}-\lambda_{p})\lambda_{1}^{-1}r_{i}+\mathcal{O}(\varepsilon),& {\quad p>1,~t>T.}
	\end{array} \right.\\
	\end{aligned}
	\end{equation*}
	The inequality above implies that the $i$-th row of $X~(p<i\leq n)$ declines to zero with a speed which is at least exponential when $p=1$, or when  $t>T$.
\end{remark}

Recall from \cref{eq:error2} that the error is measured by $p-\tr\left(\bar{X}^\top\bar{X}^\top(X^\top X)^{-1}\right)\in[0, ~p]$ for $X=X^{(k)}$, we define a zero-error optimal set for the top-$p$ OPCA problem as
$$\mathcal{X}^*=\left\{X\in\mathbb{R}^{n\times p}~\text{of full rank satisfying}~\cref{eq:firstorder}~\text{and}~ \mathbb{E}\tr\left(\bar{X}^\top\bar{X}^\top(X^\top X)^{-1}\right)=p\right\},$$ which includes the global minimum of \cref{eq:p2E}. As could be seen from \cref{fig:threephase}, we divide the whole trajectory of SGN iterates into three phases according to the expectational error. Based on this, we could then prove the convergence rate of SGN.

	\begin{figure}[H]
	\centering
	\begin{tikzpicture}[->,>=stealth',shorten >=1pt,auto,node distance=2.8cm,
	semithick]
	\tikzstyle{rectangle2}=[fill=none,draw=none,text=black]
	\tikzstyle{arrow} = [thick,->,>=stealth, draw=black, fill=black]
	\node[rectangle2]     (x1) at (0, 0){$\mathbb{E}\tr\left(\bar{X}^\top\bar{X}^\top(X^\top X)^{-1}\right)=0$};
	\node[rectangle2]     (x2) at (4.1, 0){$p\delta$};
	\node[rectangle2]     (x3) at (6.6, 0){$p(1-\delta)$};
	\node[rectangle2]     (x4) at (9.5, 0)[align=center]{certain\\accuracy};
	\node[rectangle2]     (ph1) at (3, 0.3){\emph{Phase 1}};
	\node[rectangle2]     (ph2) at (5.1, 0.3){\emph{Phase 2}};
	\node[rectangle2]     (ph3) at (8, 0.3){\emph{Phase 3}};
	\draw[arrow](x1) -- (x2);
	\draw[arrow](x2) -- (x3);
	\draw[arrow](x3) -- (x4);
	\end{tikzpicture}
	\label{fig:threephase}
	\caption{Three-phase analysis~($X = X^{(k)},~\delta\in(0, 1/2)$). }
\end{figure}

\begin{proof}[Proof of \cref{th:rate}]
	The proof consists of four parts. The first three are for analyzing the behaviours of SGN in different phases as being shown in \cref{fig:threephase} using the ODE/SDE approximations given in \cref{lemma:kT1} and \cref{lemma:kT2}. The last part is dedicated to bringing the first three parts together to obtain the total iteration required for the convergence to a neighbourhood of the optimal set with certain order of accuracy.
	
	\noindent \emph{[Phase 1]}.\quad We investigate into the worst case where the initial point is at or around the lowest saddle point with $ \mathcal{I}_{p}=\{i_{1}, \cdots, i_{p}\}\subset \{p'+1, \cdots, n\}$.
	In this case, \cref{lemma:kT2} suggests that the $(l, r)$-th element of $\bar{Y}$ satisfies the following SDE
	\begin{align}
	\label{ph1}
	dy_{lr}(t)=\left(\frac{\lambda_{l}}{\lambda_{i_{r}}}-1\right)y_{lr}dt+h^{-\frac{1}{2}}\lambda_{l}^{\frac{1}{2}}dB(t),\quad\quad\quad 1\leq l\leq p',\quad 1\leq r\leq p,
	\end{align}
	which is an Ornstein-Uhlenbeck process \cite{uhlenbeck1930theory} with formal solution
	\begin{align*}
	y_{lr}(t)&=y_{lr}(0)\exp\left\{\left(\frac{\lambda_{l}}{\lambda_{i_{r}}}-1\right)t\right\}+h^{-\frac{1}{2}}\lambda_{l}^{\frac{1}{2}}\int_{0}^{t}\exp\left\{\left(\frac{\lambda_{l}}{\lambda_{i_{r}}}-1\right)(t-s)\right\}dB(s)\\
	&=h^{-\frac{1}{2}}\lambda_{l}^{\frac{1}{2}}\int_{0}^{t}\exp\left\{\left(\frac{\lambda_{l}}{\lambda_{i_{r}}}-1\right)(t-s)\right\}dB(s).
	\end{align*}
	Calculating the expectational square and placing the timescale $t$ back to $k\alpha$, we obtain
	\begin{equation*}
	\mathbb{E}y_{lr}^2(t)\!=\! \frac{\lambda_{l}}{h}\!\int_{0}^{t}\!\!\exp\left\{2\left(\frac{\lambda_{l}}{\lambda_{i_{r}}}\!-\!1\right)(t\!-\!s)\right\}ds\!=\!\frac{ \lambda_{l}\lambda_{i_{r}}}{2h(\lambda_{i_{r}}\!-\!\lambda_{l})}
	\left(1\!-\!
	\exp\!\left\{
	\frac{2(\lambda_{l}\!-\!\lambda_{i_{r}})}{\lambda_{i_{r}}}k\alpha\right\}
	\right).
	\end{equation*}
	Recall that $Y_{\alpha} (t)= \alpha^{-\frac{1}{2}}X_{\alpha}(t)V$, we have for $1\leq l\leq p'$
	\begin{align*}
	\mathbb{E}e_{l}^\top XX^\top e_{l}=\alpha\mathbb{E}e_{l}^\top YY^\top e_{l}
	&=\sum_{r=1}^{p}
	\frac{\alpha \lambda_{l}\lambda_{i_{r}}}{2h(\lambda_{l}-\lambda_{i_{r}})}\left(\exp\left\{\frac{2(\lambda_{l}-\lambda_{i_{r}})}{\lambda_{i_{r}}}k\alpha\right\}-1\right)\\
	&\geq\frac{\alpha }{2h}\frac{\lambda_{1}\lambda_{n}}{\lambda_{1}-\lambda_{n}}\sum_{r=1}^{p}
	\left(\exp\left\{\frac{2(\lambda_{l}-\lambda_{i_{r}})}{\lambda_{i_{r}}}k\alpha\right\}-1\right)\\
	&>\frac{\alpha p}{2h}\frac{\lambda_{1}\lambda_{n}}{\lambda_{1}-\lambda_{n}}
	\left(\exp\left\{\frac{2\nu}{\lambda_{p}}k\alpha\right\}-1\right),
	\end{align*} 
	\noindent The boundedness of the solution $X(t)$ to ODE \cref{eq:ode}, which is given in \cref{lemma:bound}, indicates that $\sigma^{2}_{\max}\left(X\right)<\infty$. Then, we have
	\begin{equation}
	\label{eq:ph1goal}\mathbb{E}e_{l}^\top X(X^\top X)^{-1}X^\top e_{l}
	>\frac{1}{\sigma^{2}_{\max}(X)}\frac{\alpha p}{2h}\frac{\lambda_{1}\lambda_{n}}{\lambda_{1}-\lambda_{n}}
	\left(\exp\left\{\frac{2\nu}{\lambda_{p}}k\alpha\right\}-1\right).
	\end{equation}
	To arrive at the target of this stage, we take the right-hand side of \cref{eq:ph1goal} to be $\delta $ and obtain
	\begin{align}	\label{eq:ph1goal2}
	\frac{\alpha p}{2h}\frac{\lambda_{1}\lambda_{n}}{\lambda_{1}-\lambda_{n}}
	\left(\exp\left\{\frac{2\nu}{\lambda_{p}}k\alpha\right\}-1\right)=\delta \sigma^{2}_{\max}(X),
	\end{align}
	which thus yields the estimated passage time of the first phase:
	\begin{align}
	\label{eq:N1}
	k_{\alpha, 1}\asymp \frac{1}{2}\alpha^{-1}\lambda_{p}\nu^{-1}\log(\alpha^{-1}).
	\end{align}
	Note from \cref{eq:ph1goal} and \cref{eq:ph1goal2} that in this phase, we demand additionally that the square length of each row of $\bar{X}^\top(X^\top X)^{-\frac{1}{2}}$ should move from zero towards $\delta$ so that $\bar{X}^{(k_{\alpha, 1})}$	is of full rank. Thus, in the next phase, the iterate would not converge to some stationary point out of $\mathcal{X}^*$.

	\noindent \emph{[Phase 2]}.\quad Let~$R=\tr(X^\top\Sigma X(X^\top X)^{-1})$. By \cref{lemma:kT1}, we attain
	\begin{align*}
	\frac{dR}{dt}=&\tr\left[\left(\frac{dX^\top}{dt}\Sigma+X^\top\Sigma\frac{dX}{dt}\right)(X^\top X)^{-1}-X^\top\Sigma X(X^\top X)^{-1}\frac{d X^\top X}{dt}(X^\top X)^{-1}\right]\\
	=&2\tr\left[(X^\top X)^{-2}X^\top \Sigma^2 X-(X^\top X)^{-2}X^\top \Sigma X(X^\top X)^{-1} X^\top \Sigma X\right]\\
	=&\frac{1}{2}\left\Vert (I_{n}-X(X^\top X)^{-1}X^\top)\nabla f_{\text{E}}(X)(X^\top X)^{-1}\right\Vert_{\text{F}}^2\geq 0,
	\end{align*}
	which implies that $R(t)$ will keep increasing until $\nabla f_{\text{E}}(X) = 0$. Then, we have
	\begin{align*}
	R(X^*)-R(X)&=\sum_{i=1}^{p}\lambda_{i}-\sum_{j=1}^{n}\lambda_{j}\tr\left((X^\top X)^{-1}X_{(j, :)}^\top X_{(j, :)}\right)\\
	&=\sum_{i=1}^{p}\lambda_{i}\left(1-X_{(i, :)}(X^\top X)^{-1}X_{(i, :)}^\top \right)-\sum_{j=p+1}^{n}\lambda_{j}X_{(j, :)}(X^\top X)^{-1}X_{(j, :)}^\top\\
	&\geq\lambda_{p}\sum_{i=1}^{p}\left(1-X_{(i, :)}(X^\top X)^{-1}X_{(i, :)}^\top \right)-\sum_{j=p+1}^{n}\lambda_{j}X_{(j, :)}(X^\top X)^{-1}X_{(j, :)}^\top\\
	&=\lambda_{p}\left(p-\sum_{i=1}^{p'}X_{(i, :)}(X^\top X)^{-1}X_{(i, :)}^\top \right)-\tau\sum_{j=p'+1}^{n}X_{(j, :)}(X^\top X)^{-1}X_{(j, :)}^\top\\
	&=(\lambda_{p}-\tau)\left(p-\tr(\bar{X}^\top\bar{X}(X^\top X)^{-1}\right)>0,
	\end{align*}
	where $X^*\in\mathcal{X}^*$ and $\lambda_{n}\leq \tau< \lambda_{p}$ satisfying $$\sum_{j=p'+1}^{n}\lambda_{j}X_{(j, :)}(X^\top X)^{-1}X_{(j, :)}^\top=\tau\sum_{j=p'+1}^{n}X_{(j, :)}(X^\top X)^{-1}X_{(j, :)}^\top.$$
	That is to say, we obtain a lower bound of $\tr\left(\bar{X}^\top\bar{X}^\top(X^\top X)^{-1}\right)$ using $R(X)$.
	Then, by setting $R(X^*)-R(X)\leq (\lambda_{p}-\tau) p\delta$, $t=k\alpha$, we have
	\begin{equation}\label{eq:N2}
	k_{\alpha, 2}\asymp \alpha^{-1}R^{-1}\left(\sum_{i=1}^{p}\lambda_{i}-(\lambda_{p}-\tau) p\delta\right).
	\end{equation}

	\noindent \emph{[Phase 3]}.\quad
	This stage shows oscillations around some $X^*\in\mathcal{X}^*$, in which case $X_{\alpha}(t)$ could be written as 
	\begin{equation}
	\label{eq:y}
	X_{\alpha}(t)=E_{p'}\Lambda_{p'}^{\frac{1}{2}}WV^\top+\mathcal{O}(\alpha^{\frac{1}{2}}),
	\end{equation}
	where $E_{p'}=[e_{1}, e_{2}, \cdots, e_{p'}]\in\mathbb{R}^{n\times p'},  \Lambda_{p'}=\Diag(\lambda_{1}, \lambda_{2}, \cdots, \lambda_{p'})\in\mathbb{R}^{p'\times p'}$,
	\begin{equation*} 
	W=
	\left[         
	\begin{array}{cc}   
	I_{p-1}&0_{(p-1)\times 1}\\
	0_{(p'-p+1)\times(p-1)}&w
	\end{array}
	\right]\in\mathbb{R}^{p'\times p},\quad w^\top w=1,\quad w\in\mathbb{R}^{p'-p+1}.
	\end{equation*}
	In this situation, \cref{lemma:kT2} indicates that the $(l, r)$-th entry of $\wideubar{Y}$ takes the form
	\begin{align*}
	y_{lr}(t)=y_{lr}(0)\exp\left\{\frac{(\lambda_{l}-\lambda_{r})}{\lambda_{r}}t\right\}+h^{-\frac{1}{2}}\lambda_{l}^\frac{1}{2}\int_{0}^{t}\exp\left\{\frac{(\lambda_{l}-\lambda_{r})}{\lambda_{r}}(t-s)\right\}dB(s).
	\end{align*}
	Then we have the expectational square
	\begin{align*}
	\mathbb{E}y_{lr}^{2}(t)&=y_{lr}^{2}(0)\exp\left\{\frac{2(\lambda_{l}-\lambda_{r})}{\lambda_{r}}t\right\}+h^{-1}\lambda_{l}\int_{0}^{t}\exp\left\{\frac{2(\lambda_{l}-\lambda_{r})}{\lambda_{r}}(t-s)\right\}ds\\
	&=\exp\left\{\frac{2(\lambda_{l}-\lambda_{r})}{\lambda_{r}}t\right\}\left(y^{2}_{lr}(0)-\frac{\lambda_{l}\lambda_{r}}{2h(\lambda_{r}-\lambda_{l})}\right)+\frac{\lambda_{l}\lambda_{r}}{2h(\lambda_{r}-\lambda_{l})}.
	\end{align*}
	According to the matrix perturbation theory \cite{stewart1990matrix}, we have
	\begin{align*}
	(X^\top X)^{-1} = \left(V\Lambda_{p'}^{\frac{1}{2}}W^\top E_{p'}^\top E_{p'} \Lambda_{p'}^{\frac{1}{2}} WV^\top+\mathcal{O}\left(\alpha^\frac{1}{2}\right)\right)^{-1}=V\Lambda_{p}^{-1}V^\top+\mathcal{O}\left(\alpha^{\frac{1}{2}}\right).
	\end{align*}
	Then, using the relation $Y_{\alpha}(t) = \alpha^{-
		\frac{1}{2}}X_{\alpha}(t)V$, $t = k\alpha$ and the formula \cref{eq:y}, we get
	\begin{align*}
	&~\mathbb{E}\tr\left(\wideubar{X}^\top\wideubar{X}(X^\top X)^{-1}\right)=\alpha\mathbb{E}\tr\left(V\wideubar{Y}^\top\wideubar{Y}V^\top V\Lambda_{p}^{-1}V^\top\right)+\mathcal{O}\left(\alpha^{\frac{3}{2}}\right)\\
	=&\sum_{l=p'+1}^{n}\sum_{r=1}^{p}\exp\left\{\frac{2(\lambda_{l}-\lambda_{r})}{\lambda_{r}}k\alpha\right\} \frac{\alpha y^{2}_{lr}(0)}{\lambda_{r}}+C\sum_{l=p'+1}^{n}\sum_{r=1}^{p}\frac{\lambda_{l}}{2h(\lambda_{r}-\lambda_{l})}\alpha+\mathcal{O}\left(\alpha^{\frac{3}{2}}\right)\\
	\lesssim&~ p\delta\cdot\exp\left\{-\frac{2\nu}{\lambda_{p}}k\alpha\right\}+C\sum_{l=p'+1}^{n}\sum_{r=1}^{p}\frac{\lambda_{l}}{2h(\lambda_{r}-\lambda_{l})}\alpha+\mathcal{O}\left(\alpha^{\frac{3}{2}}\right),
	\end{align*}
	where $C$ is some positive constant.
	Then, we set
	$
	p\delta\cdot\exp\left\{-2\nu\lambda_{p}^{-1}k\alpha\right\}=o(\alpha),
	$
	which thus gives the traversing time of the last phase
	\begin{align}
	\label{eq:N3}k_{\alpha, 3}=\frac{1}{2}\log(p\alpha^{-1}\delta)\lambda_{p}\nu^{-1}\alpha^{-1}\asymp \frac{1}{2}\log(\alpha^{-1})\lambda_{p}\nu^{-1}\alpha^{-1}.
	\end{align}	
	
	\noindent\emph{[Combination of three phases]}.\quad
	Based on the Markov property of SGN update, summing up \cref{eq:N1}, \cref{eq:N2} and \cref{eq:N3}  gives an estimate on the global convergence time of SGN asymptotically:
	\begin{equation*}
	k_{\alpha}=	k_{\alpha, 1}+k_{\alpha, 2}+k_{\alpha, 3}\asymp\lambda_{p}\nu^{-1} \log(\alpha^{-1})\alpha^{-1}.
	\end{equation*}
	The expectational estimation error of the approximation $X^{(k_\alpha)}$ is
	\begin{equation}
	\label{eq:error}\mathbb{E}\Vert\sin\Theta( X^{(k_{\alpha})}, U_{p'})\Vert_{\text{\rm F}}^{2}\lesssim \frac{C}{h}\sum_{l=p'+1}^{n}\sum_{r=1}^{p}\frac{\lambda_{l}}{2(\lambda_{r}-\lambda_{l})}\alpha+o(\alpha).
	\end{equation}
	Since the sample size $m$ is known, the total number of iterations could be expressed as $K = \left\lceil m/h\right\rceil$. 	
	Let		
	\begin{align}
	\label{eq:alpha} \widetilde{\alpha}(K)=\frac{\lambda_{p}\log K}{\nu K},
	\end{align}
	we have
	\begin{align*}
	k_{\widetilde{\alpha}}\asymp\left(\lambda_{p}\nu^{-1}\widetilde{\alpha}^{-1}\log\widetilde{\alpha}^{-1}\right)\asymp K.
	\end{align*}
	Substituting \cref{eq:alpha} into \cref{eq:error} yields the bound for SGN.
\end{proof}

\subsection{Main results on diminishing stepsizes}\label{sec:theory-dimi}
The forementioned diffusion approximation theory can also be applied to the diminishing stepsize case
\begin{equation}
\label{eq:dimi}
\alpha^{(k)} = \frac{\gamma}{c_1(k+c_2)^{\beta}},\quad 0<\beta<1,\quad c_1,~c_2>0,
\end{equation}
with infinitesimal $\gamma$.
The corresponding differential equation approximations could be obtained by making a slight modification on the construction of the continuous-time extension of SGN iterates $\{X^{(k)}\}$ using the stepsize \cref{eq:dimi}, that is,
$$X_{\gamma}(t)= X^{(\lfloor t\gamma^{\frac{1}{\beta-1}}\rfloor)}.$$
In a way similar to the constant stepsize case, we compute the infinitesimal characteristics \cref{eq:infinitemean},~\cref{eq:infinitevariance}, and have the results in \cref{lemma:kT1-dimi} and \cref{lemma:kT2-dimi}. There is no essential difference of the proofs of them from those of \cref{lemma:kT1} and \cref{lemma:kT2}, and thus are omitted here.

\begin{lemma}\label{lemma:kT1-dimi} 
	Under the assumptions [A1]-[A4],
	the Markov process $X_{\gamma}(t)= X^{(\lfloor t\gamma^{\frac{1}{\beta-1}}\rfloor)}$ generated by \cref{alg:sgn} with the stepsize \cref{eq:dimi} and  $c_2 = \gamma^{\frac{1}{\beta-1}}$
	converges in distribution to the solution of 	\begin{equation}\label{eq:ode-dimi}	\frac{dX}{dt}=\frac{1}{c_1(t+1)^\beta}\left(\Sigma X(X^\top X)^{-1}-\frac{1}{2}X-\frac{1}{2}X(X^\top X)^{-1}X^\top\Sigma X(X^\top X)^{-1}\right),
	\end{equation}
	as $\gamma\to 0$ with $X(0)=X^{(0)}$.
\end{lemma}

\begin{remark}
	If the parameter $c_2$ in \cref{eq:dimi} is set to a constant independent of $\gamma$, we would have the following limiting differential equation as $\alpha\to 0$
	\begin{equation}\label{eq:ode-dimi2}
	\frac{dX}{dt}=\frac{1}{c_1 t^\beta}\left(\Sigma X(X^\top X)^{-1}-\frac{1}{2}X-\frac{1}{2}X(X^\top X)^{-1}X^\top\Sigma X(X^\top X)^{-1}\right),
	\end{equation}
	which are almost the same as \cref{eq:ode-dimi} except for the multiplicative factor. In the subsequent discussions, we will make full use of \cref{eq:ode-dimi} because it avoids the singularity of \cref{eq:ode-dimi2} at $t = 0$.
\end{remark}

\begin{lemma}\label{lemma:kT2-dimi}
	Suppose that the assumptions [A1]-[A4] hold and the initial point is
	around some rank-$p$ stationary point of \cref{eq:p2E}, which is associated with an eigen-index set~$\mathcal{I}_{p}=\{i_{1}, \cdots, i_{p}\}$
	and an orthogonal matrix~$V\in\mathbb{R}^{p\times p}$.
	We define a new process
	\begin{equation}\label{eq:z-dimi}
	Y_\gamma(t)=\gamma^{-\frac{1}{2(1-\beta)}}X_\gamma(t)V,
	\end{equation}
	generated by \cref{alg:sgn} with the stepsize \cref{eq:dimi} and $c_2 = \gamma^{\frac{1}{\beta-1}}$.
	If $Y_{\gamma}(0)$ converges in distribution to some constant $Y^{0}\in\mathbb{R}^{n\times p}$ as $\gamma\to 0$ and the index $l$ is not in the extended eigen-index set $\hat{\mathcal{I}}_{p}$, then the $(l, r)$-th entry of $Y_{\gamma}(t)$ converges in distribution to the following SDE with $Y_{(l, r)}(0) = Y^{0}_{(l, r)}$
	\begin{equation}\label{eq:ksde-dimi}
	dY_{(l, r)}=\frac{1}{c_1(t+1)^\beta}\left(\Sigma Y\Lambda_{p}^{-1}-Y\right)_{(l, r)}dt+\frac{\lambda_{l}^{\frac{1}{2}}}{c_1h^{\frac{1}{2}}(t+1)^\beta}dB(t),\quad \quad 1\leq r\leq p,
	\end{equation}
	where $\Lambda_{p}=\Diag(\lambda_{i_{1}}, \lambda_{i_{2}}, \cdots, \lambda_{i_{p}})\in\mathbb{R}^{p\times p}$ and $B(t)$ is the Brownian motion.
\end{lemma}

The additional multiplicative factor $c_{1}^{-1}(t+1)^{-\beta}$ makes the limiting differential equations \cref{eq:ode-dimi} and \cref{eq:ksde-dimi} time-inhomogeneous, which calls for a more sophisticated analysis than in the constant stepsize case in order to obtain the convergence rate.

The following theorem states the convergence result for the diminishing stepsize case.

\begin{theorem}\label{th:rate-dimi}
	Suppose that the assumptions [A1]-[A4] hold.
	Given the total sample size $m>0$, and the stepsize \cref{eq:dimi} with
	\begin{equation*}
	c_1 = \frac{\nu}{\lambda_p},\quad c_2 = \gamma^{\frac{1}{\beta-1}},\quad\beta = 1-\frac{1}{\log K},\quad \gamma = \frac{(1-\beta)\log K}{K^{1-\beta}}
	,\quad K = \left\lceil\frac{m}{h}\right\rceil,
	\end{equation*}
	where $k=0, 1,\cdots, K$,
	the expectational error of the final output $X^{(K)}$ of \cref{alg:sgn} is given by
	\begin{equation}\label{eq:rate-dimi2}
	\mathbb{E}\Vert\sin\Theta(X^{(K)}, U_{p'})\Vert_{\text{\rm F}}^{2}\lesssim \frac{C}{2h}p(n-p')\frac{\lambda_{p}\lambda_{p'+1}}{\nu^2}K^{-1},
	\end{equation}
	where $C$ is some positive constant.
\end{theorem}

\begin{proof}
	We adopt the framework same as that in the proof of \cref{th:rate}, where the sequence of SGN iterates with the diminishing stepsize \cref{eq:dimi} is described by a three-phase process.
	
	\noindent \emph{[Phase 1]}.\quad 
	We take the initial condition $\mathbb{E}\tr(\bar{X}^\top\bar{X}(X^\top X)^{-1})=0$ and consider the worst case for which the initial point is at or around the lowest saddle point with $ \mathcal{I}_{p}=\{i_{1}, \cdots, i_{p}\}\subset \{p'+1, \cdots, n\}$.
	As \cref{lemma:kT2-dimi} suggests in this case, the $(l, r)$-th element of $\bar{Y}$ is the solution to the following SDE
	\begin{equation*}
	dy_{lr}(t)=\frac{1}{c_1(t+1)^\beta}\left(\frac{\lambda_{l}}{\lambda_{i_{r}}}-1\right)y_{lr}dt+\frac{\lambda_{l}^{\frac{1}{2}}}{c_1h^{\frac{1}{2}}(t+1)^\beta}dB(t),\quad 1\leq l\leq p',\quad 1\leq r\leq p,
	\end{equation*}
	whose solution has an explicit form
	\begin{equation*}
	y_{lr}(t)=\Phi(t)\left(y_{lr}(0)+\frac{\lambda^\frac{1}{2}}{c_1h^\frac{1}{2}}\int_0^t\Phi^{-1}(s)(s+1)^{-\beta}dB(s)\right),
	\end{equation*}
	where
	\begin{equation*}
	\Phi(t) = \exp\left\{\frac{\lambda_l-\lambda_{i_r}}{c_1\lambda_{i_r}}\frac{(t+1)^{1-\beta}-1}{1-\beta}\right\}.
	\end{equation*}
	The expectational square of $y_{lr}(t)$ then satisfies
	\begin{align}
	\notag\mathbb{E}y_{lr}^2(t)&=\frac{\lambda_l}{c_1^2 h}\Phi^2(t)\int_{0}^t\Phi(s)^{-2}(s+1)^{-2\beta}dB(s)\\
	\label{eq:ph1-dimi-ineq}&\geq C_1\left(\exp\left\{\frac{2(\lambda_l-\lambda_{i_r})}{c_1\lambda_{i_r}}\frac{1}{1-\beta}\left((t+1)^{1-\beta}-1\right)\right\}-1\right),
	\end{align}
	where
	\begin{equation*}
	C_1=\frac{\lambda_l}{c_1^2h(1-\beta)}\left(\frac{(1-\beta)c_1\lambda_{i_r}}{2(\lambda_l-\lambda_{i_r})}\right)^{\frac{1-2\beta}{1-\beta}}\frac{2(\lambda_l-\lambda_{i_r})}{\beta c_1\lambda_{i_r}+2(\lambda_l-\lambda_{i_r})}.
	\end{equation*}
	The inequality \cref{eq:ph1-dimi-ineq} could be attained by taking the variable transformation $$z = \frac{2(\lambda_l-\lambda_{i_r})}{c_1\lambda_{i_r}(1-\beta)}(s+1)^{1-\beta},$$ which leads to the difference of two incomplete Gamma functions.

	Using the relation $Y_{\gamma} (t)= \gamma^{-\frac{1}{2(1-\beta)}}X_{\gamma}(t)V$ and placing the timescale $t$ back to $k\gamma^{\frac{1}{1-\beta}}$, we have for $1\leq l\leq p'$
	\begin{align}
	\notag&\mathbb{E}e_{l}^\top X(X^\top X)^{-1}X^\top e_{l}\\
	\label{eq:ph1goal-dimi}	>&\sum_{r=1}^{p}\frac{C_1\gamma^{\frac{1}{1-\beta}}}{\sigma^{2}_{\max}(X)}\left(\exp\left\{\frac{2(\lambda_l-\lambda_{i_r})}{c_1\lambda_{i_r}}\frac{1}{1-\beta}\left((k\gamma^{\frac{1}{1-\beta}}+1)^{1-\beta}-1\right)\right\}-1\right).
	\end{align}
	Taking the right side of \cref{eq:ph1goal-dimi} to be $\delta $ gives the passage time of the first phase:
	\begin{align}
	\label{eq:N1-dimi}
	k_{\gamma, 1}^{1-\beta}\asymp \frac{c_{1}\lambda_p}{2\nu}\gamma^{-1}\log(\gamma^{-1}).
	\end{align}
	
	\noindent \emph{[Phase 2]}.\quad 
	Following the same steps as in the proof of \cref{th:rate} directly yields the estimated traversing time of the second phase:
	\begin{align}\label{eq:N2-dimi}
	k_{\gamma, 2}^{1-\beta}=\mathcal{O}(\gamma^{-1}).
	\end{align}
	
	\noindent \emph{[Phase 3]}.\quad
	This phase depicts the behaviour of SGN oscillating around the optimal set $\mathcal{X}^*$, and in the neighbourhood of $\mathcal{X}^*$, $X_{\gamma}(t)$ could be formulated like \cref{eq:y} as
	\begin{equation}
	\label{eq:y-dimi}
	X_{\gamma}(t)=E_{p'}\Lambda_{p'}^{\frac{1}{2}}WV^\top+\mathcal{O}(\gamma^{\frac{1}{2(1-\beta)}}).
	\end{equation}
	In this case, \cref{lemma:kT2-dimi} indicates that the $(l, r)$-th entry of $\wideubar{Y}$ takes the form
	\begin{equation*}
	y_{lr}(t)=\Phi(t)\left(y_{lr}(0)+\frac{\lambda^\frac{1}{2}_l}{c_1h^\frac{1}{2}}\int_0^t\Phi^{-1}(s)(s+1)^{-\beta}dB(s)\right),
	\end{equation*}
	where
	\begin{equation*}
	\Phi(t) = \exp\left\{\frac{\lambda_l-\lambda_{r}}{c_1\lambda_{r}}\frac{(t+1)^{1-\beta}-1}{1-\beta}\right\}.
	\end{equation*}
	As before, we have the expectational square
	\begin{align*}
	\mathbb{E}y_{lr}^{2}(t)=\Phi^2(t)\left(y_{lr}^2(0)+\frac{\lambda_l}{c_1^2 h}\int_{0}^{t}\Phi^{-2}(s)(s+1)^{-2\beta}ds\right).
	\end{align*}
	According to $Y_{\gamma}(t) = \gamma^{-
		\frac{1}{2(1-\beta)}}X_{\gamma}(t)V$ and the expression \cref{eq:y-dimi}, we have
	\begin{align*}
	&\mathbb{E}\tr\left(\wideubar{X}^\top\wideubar{X}(X^\top X)^{-1}\right)\\
	=&\gamma^{\frac{1}{1-\beta}}\mathbb{E}\tr\left(V\wideubar{Y}^\top\wideubar{Y}V^\top V\Lambda_{p}^{-1}V^\top\right)+\mathcal{O}\left(\gamma^{\frac{3}{2(1-\beta)}}\right)\\
	\asymp&\gamma^{\frac{1}{1-\beta}}\sum_{l=p'+1}^{n}\sum_{r=1}^{p}\exp\left\{\frac{2(\lambda_{l}-\lambda_{r})}{c_1\lambda_{r}}\frac{(t+1)^{1-\beta}-1}{1-\beta}\right\}\frac{y_{lr}^2(0)}{\lambda_r}\\
	\!+&\!\!\sum_{l=p'\!+\!1}^{n}\!\sum_{r=1}^{p}\frac{\gamma^{\frac{1}{1\!-\!\beta}}\lambda_l}{c_1^2h\lambda_r}\exp\!\left\{\frac{2(\lambda_l\!-\!\lambda_r)}{c_1\lambda_r}\frac{(t\!+\!1)^{1\!-\!\beta}}{1\!-\!\beta}\right\}\int_{0}^{t}\!\!\exp\!\left\{\frac{2(\lambda_r\!-\!\lambda_l)}{c_{1}\lambda_r}\frac{(s\!+\!1)^{1\!-\!\beta}}{1\!-\!\beta}\right\}\!(s\!+\!1)^{-\!2\beta}ds\!\\
	\leq&~p\delta\exp\left\{\frac{2\nu}{c_1\lambda_p}\frac{(t+1)^{1-\beta}-1}{1-\beta}\right\}\\
	\!+&\!\!\sum_{l=p'\!+\!1}^{n}\!\sum_{r=1}^{p}\frac{\gamma^{\frac{1}{1\!-\!\beta}}\lambda_l}{c_1^2h\lambda_r}\exp\!\left\{\frac{2(\lambda_l\!-\!\lambda_r)}{c_1\lambda_r}\frac{(t\!+\!1)^{1\!-\!\beta}}{1\!-\!\beta}\right\}\int_{0}^{t}\!\!\exp\!\left\{\frac{2(\lambda_r\!-\!\lambda_l)}{c_{1}\lambda_r}\frac{(s\!+\!1)^{1\!-\!\beta}}{1\!-\!\beta}\right\}\!(s\!+\!1)^{-\beta}ds,\!
	\end{align*}
	Then, using $t = k\gamma^{\frac{1}{1-\beta}}$ and setting
	$$
	p\delta\exp\left\{\frac{2\nu}{c_1\lambda_p}\frac{(t+1)^{1-\beta}-1}{1-\beta}\right\}=o(\gamma^{\frac{1}{1-\beta}}),
	$$
	give the traversing time of the last phase:
	\begin{equation}
	\label{eq:N3-dimi}k_{\gamma, 3}^{1-\beta}=\frac{c_{1}\lambda_p}{2\nu}\gamma^{-1}\log(\gamma^{-1}).
	\end{equation}	
	
	\noindent\emph{[Combination of three phases]}.\quad
	Adding up \cref{eq:N1}, \cref{eq:N2} and \cref{eq:N3}, we obtain an estimate on the global convergence time of SGN asymptotically as
	\begin{equation*}
	k_{\gamma}^{1-\beta}=	k_{\gamma, 1}^{1-\beta}+k_{\gamma, 2}^{1-\beta}+k_{\gamma, 3}^{1-\beta}\asymp\frac{c_{1}\lambda_p}{\nu}\gamma^{-1}\log(\gamma^{-1}).
	\end{equation*}
	The expectational estimation error of the approximation $X^{(k_\gamma)}$ is
	\begin{align}
	\notag&\mathbb{E}\Vert\sin\Theta( X^{(k_{\gamma})}, U_{p'})\Vert_{\text{\rm F}}^{2}\\
	\notag\lesssim&~\frac{\gamma^{\frac{1}{1-\beta}}}{c_1h}\sum_{l=p'+1}^{n}\sum_{r=1}^{p}\frac{\lambda_l}{2(\lambda_{r}-\lambda_{l})}\left(1-\exp\left\{\frac{2(\lambda_r-\lambda_l)}{c_{1}(1-\beta)\lambda_r}\left(1-(t+1)^{1-\beta}\right)\right\}\right)\\
	\label{eq:error-dimi} \leq&~\frac{\lambda_{p'+1}p(n-p')}{c_1\nu h}\gamma^{\frac{1}{1-\beta}}.
	\end{align}
	Given the sample size $m$, we let
	\begin{equation}
	\label{eq:alpha-dimi} \widetilde{\gamma}(K)=\frac{(1-\beta)\log K}{K^{1-\beta}},\quad c_1 = \frac{\nu}{\lambda_p},
	\end{equation}
	where $K = \left\lceil m/h\right\rceil$ is the total number of iterations, and
	we have
	\begin{equation*}
	k_{\widetilde{\gamma}}^{1-\beta}\asymp\left(c_1\lambda_{p}\nu^{-1}\widetilde{\gamma}^{-1}\log\widetilde{\gamma}^{-1}\right)\asymp K^{1-\beta}.
	\end{equation*}
	Substituting \cref{eq:alpha-dimi} into \cref{eq:error-dimi} and further setting $\beta = 1-1/\log K$ give the error bound \cref{eq:rate-dimi2} for SGN with the diminishing stepsize \cref{eq:dimi}.
\end{proof}

\begin{remark}	
	Comparing the bound \cref{eq:rate} for the constant stepsize case with the one \cref{eq:rate-dimi2} for the diminishing stepsize case, we see that the latter bound removes the $\log K$ factor, and thus yields the optimality of SGN method.
\end{remark}

\section{Numerical experiments}\label{sec:numeri}
This section presents the numerical performance of our proposed stochastic GN algorithms for solving OPCA on both synthetic Gaussian data and frequently-used real data in machine learning. All experiments are implemented on a MacBook Pro 3.1GHz laptop with 8GB of RAM using Matlab 2018a. 

\subsection{Testing examples}\label{sec:numeri-problem}
We test each competing algorithm on both simulated data of nearly low-rank structure and three of the most popular real datasets in machine learning.
Given the integers $p\ll n$ and $p\leq p'\ll n$, the simulated data are generated from an $n$-dimensional Gaussian distribution $\mathcal{N}_{n}(0, \Sigma)$ with
\begin{equation}
\Sigma = Q D Q^\top + \rho^2 I_{n},
\end{equation}
where $Q\in\mathbb{R}^{n\times p'}$ is a randomly generated orthogonal matrix, $D\in\mathbb{R}^{p'\times p'}$ is a diagonal matrix with $D_{ii} = \mu_{i}\in[\underline{\mu},~\bar{\mu}]$.
We categorize the data into three groups as shown in \cref{table:simudata} for specific purpose.
\begin{table}[H]\label{table:simudata}
	\centering
	\begin{tabular}{c|c|c}
		\toprule
		{\bf Top Eigenvalues}&{\bf Parameters}&{\bf Remark}\\
		\midrule
		\multicolumn{3}{c}{\texttt{Gau-gap-1}~($p'=p$)}\\
		\midrule
		\makecell[c]{
			$\mu_{i}\sim \mathcal{U}(\underline{\mu}, \bar{\mu})$\\sorted in a descending order}
		&\makecell[c]{$\underline{\mu}=0.01; \bar{\mu} = 1, 10, 100;$\\$p=1, \cdots, 30$}
		&\makecell[c]{various gaps~\&\\uniformity}\\
		\midrule
		\multicolumn{3}{c}{\texttt{Gau-gap-2}~($p'=p$)}\\
		\midrule
		\makecell[c]{$\mu_{1}=\cdots=\mu_{p_1}=\bar{\mu}$;\\$\mu_{p_1+1}=\cdots=\mu_{p}=\underline{\mu}$}
		&\makecell[c]{$\underline{\mu}=1;~ \bar{\mu}=100$;\\$p=30$;~$p_{1} = 0, 5, 15, 25$}
		&\makecell[c]{~~fixed gap~\&\\nonuniformity}\\
		\midrule
		\multicolumn{3}{c}{\texttt{Gau-ngap}~($p'>p$)}\\
		\midrule	
		\makecell[c]{$\mu_{j}\sim \mathcal{U}(\underline{\mu}, \bar{\mu}),~1\leq j\leq p$\\sorted in a descending order; \\ $\mu_{p+1}=\cdots=\mu_{p'}=\mu_{p}$}
		&\makecell[c]{$\underline{\mu}=0.01;~ \bar{\mu}=100$;\\$p=30$;~$p' = 32, 45, 60$}
		&no gap\\
		\bottomrule
	\end{tabular}
	\caption{Descriptions of simulated data.}
\end{table}
\vspace{-0.6cm}
\noindent The other parameters are fixed to 
$
n=500,~ m=10000,~\rho = 0.1.
$

The real data examples are described in \cref{table:realdata}.
\begin{table}[H]\label{table:realdata}
	\centering
	\begin{tabular}{c|c|c|c}
		\toprule
		{\bf Name}& $\mathbf m$& $\mathbf n$&{\bf Remark}\\
		\midrule
		\texttt{MNIST}~\cite{lecun1998gradient}&10000&784&test set, hand-written digits in gray-scale\\
		\texttt{Fashion-MNIST}~\cite{xiao2017fashion}&10000&784&test set, fashion product in gray-scale\\
		\texttt{CIFAR-10}~\cite{krizhevsky2009learning}&10000&3072&test set, color images of natural objects\\
		\bottomrule
	\end{tabular}
	\caption{Descriptions of real data.}
\end{table}
\vspace{-0.6cm}

\subsection{Default settings of algorithms}
The algorithms for comparison are Oja's iteration \cref{eq:oja} with constant stepsize $\alpha^{(k)}= \alpha$, diminishing stepsize $\alpha^{(k)}= \gamma/(k+1)$, its adaptive variant AdaOja \cref{eq:adaoja} and the proposed SGN using constant and diminishing stepsizes along with the adaptive-stepsize version AdaSGN.
For the test in \cref{sec:numeri-sgn-constant}, the stepsize parameter $\alpha$ is set according to \cref{eq:fixedstepsize}.
For the test in \cref{sec:numeri-sgn}, the stepsize parameter $\gamma$ goes through the set $\{2^{-5}, 2^{-4},\cdots, 2^{4}, 2^{5}\}$ and 
for the test in \cref{sec:numeri-adasgn}, $\gamma$ is optimally selected from this set, used as a benchmark result for the parameter-free AdaSGN.
Both of the single-pass $h=1$ and mini-batch $h=10, 100$ models are considered in the tests.
The accuracy of algorithms is measured by $\Vert \sin\Theta(X^{(k)}, U_{p'})\Vert_{\text{F}}^2/p\in[0, 1]$. Due to the inaccessible population covariance $\Sigma$ in the real-data case, we use the full-sampled empirical one $\Sigma_{m}$ to compute $U_{p'}$ as a compromise. 
All results shown in this paper are the average of 100 runs from random initial points.

\subsection{Comparison on constant stepsizes}\label{sec:numeri-sgn-constant}

The experiments in this subsection are to (i)~show the three-phase behaviour of SGN as introduced in \cref{sec:theory} and verify the statement in \cref{remark: comparison}; (ii)~prove the stability of SGN w.r.t. the random initialization; (iii)~validate that SGN produces smaller volatility in the third phase than Oja's iteration when $\lambda_{1}$ is large as analyzed in \cref{remark:variance}, when adopting the constant stepsize.

To provide a fair comparison with the three-phase result of Oja's iteration in \cite{li2017diffusion}, we set $h=1$ and $p=p'=1$ in this test. Under the case of fixed $\alpha$ but $100$~random initializations and the case of fixed initial point $X^{(0)}=e_{n}$ but different $\alpha$ and $\mu_{1}$, the averaged results of SGN and Oja's iteration with variance shaded are presented in \cref{fig:constSGN_simu} and \cref{fig:constSGN_simu_2}, respectively, where the ``Z"-shaped trajectories are consistent with the three-phase analysis in the theory.
As could be observed in \cref{fig:constSGN_simu}, the initialization has obvious impact on the stability of Oja's itertiaon, and however has little effect on SGN. Fixing the initial value to the saddle point $X^{(0)}=e_{n}$, \cref{fig:constSGN_simu_2} implies that as the top eigenvalue gets larger, the Oja's iteration fluctuates more significantly and yields results of much lower accuaray than SGN, and thus requires smaller stepsizes for the purposes of obtaining both the smaller oscillation and the higher accuracy, which agrees with the statements in \cref{remark: comparison} and \cref{remark:variance}.

\begin{figure}[h]
	\centering
	\subfigure[$\alpha=4\times 10^{-3}$]{
		\begin{minipage}[t]{4.5cm}
			\centering
			\includegraphics[width=4.5cm]{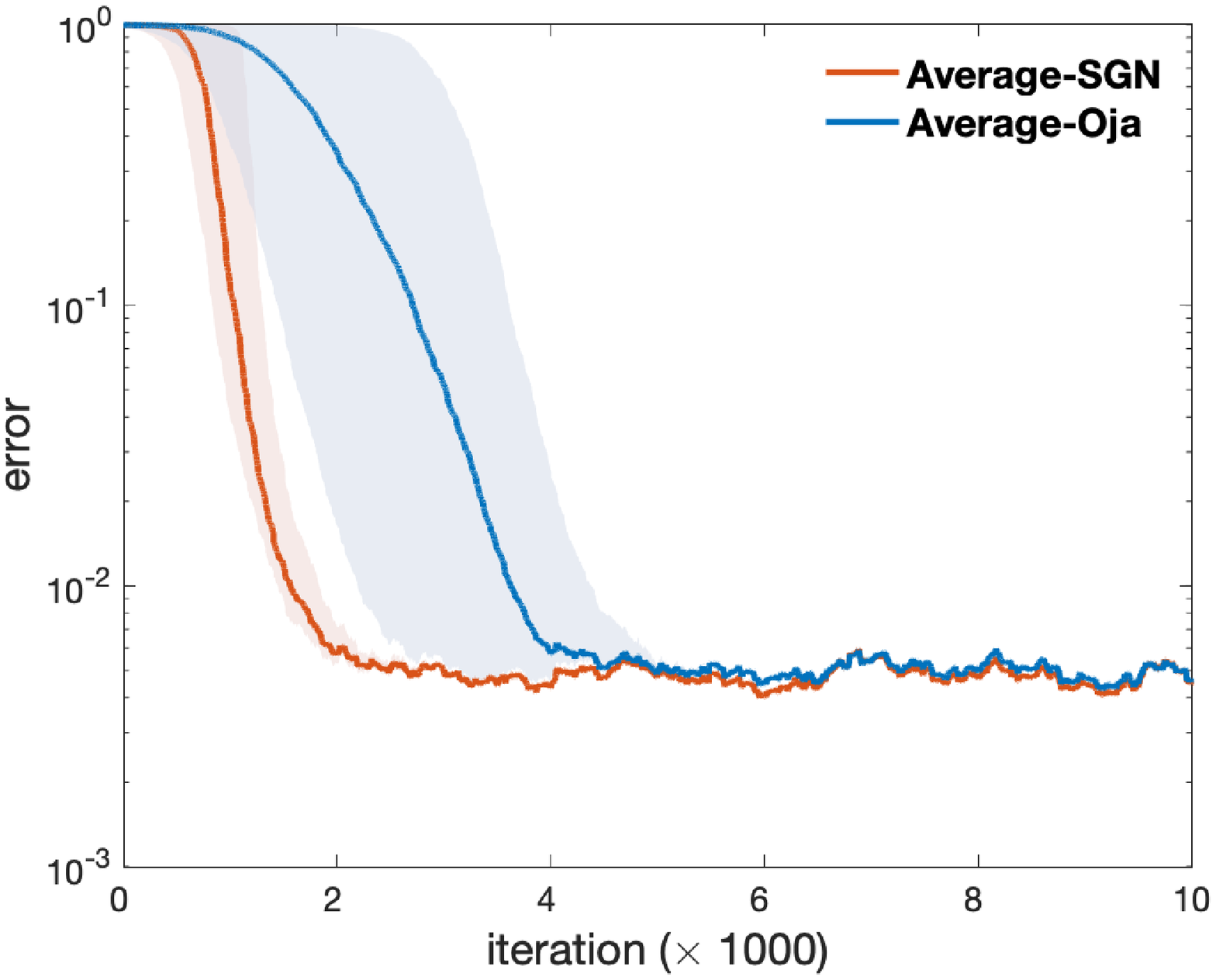}
		\end{minipage}
	}
	\subfigure[$\alpha=2\times 10^{-3}$]{
		\begin{minipage}[t]{4.5cm}
			\centering
			\includegraphics[width=4.5cm]{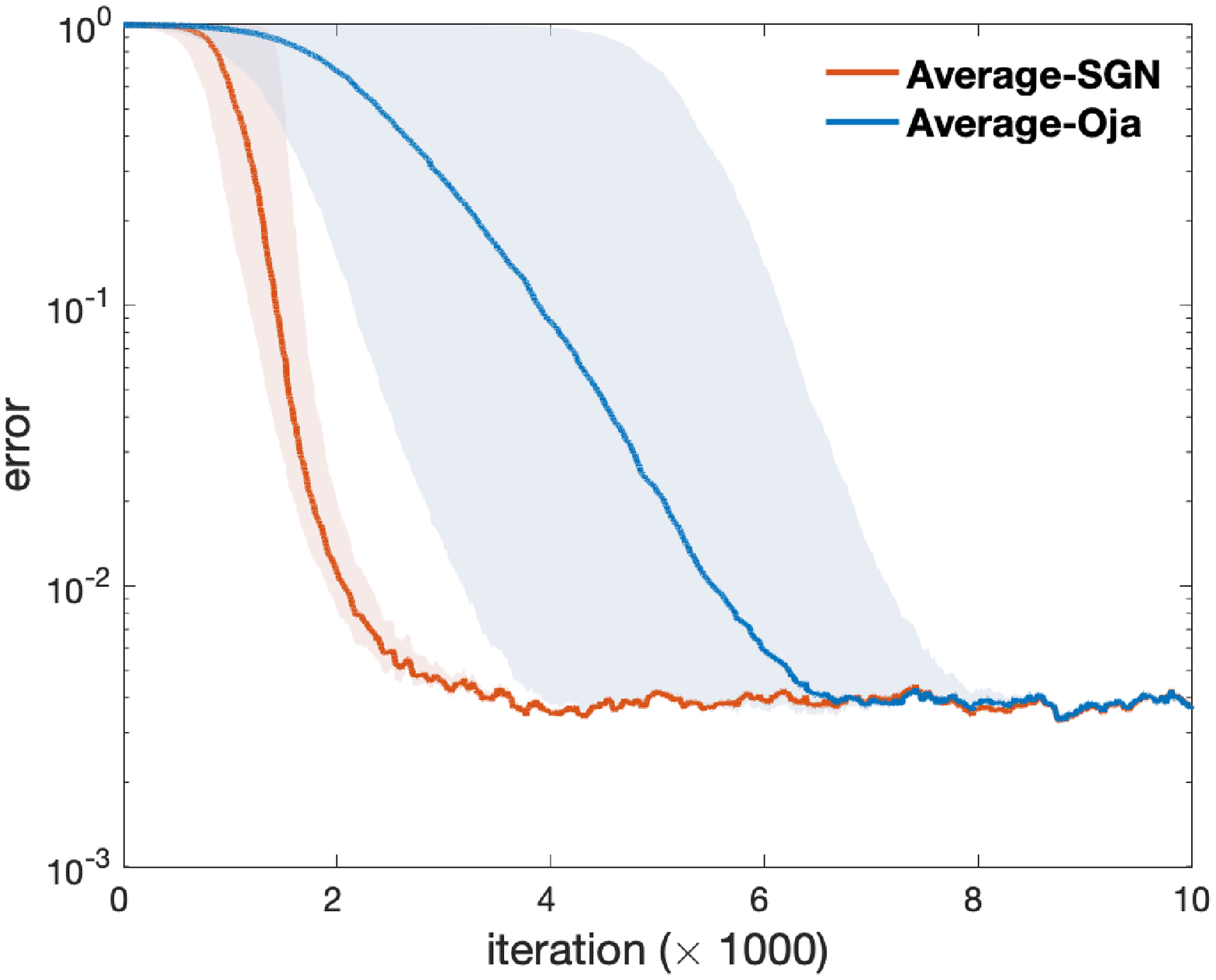}
		\end{minipage}
	}
	\subfigure[$\alpha=1\times 10^{-3}$]{
		\begin{minipage}[t]{4.5cm}
			\centering
			\includegraphics[width=4.5cm]{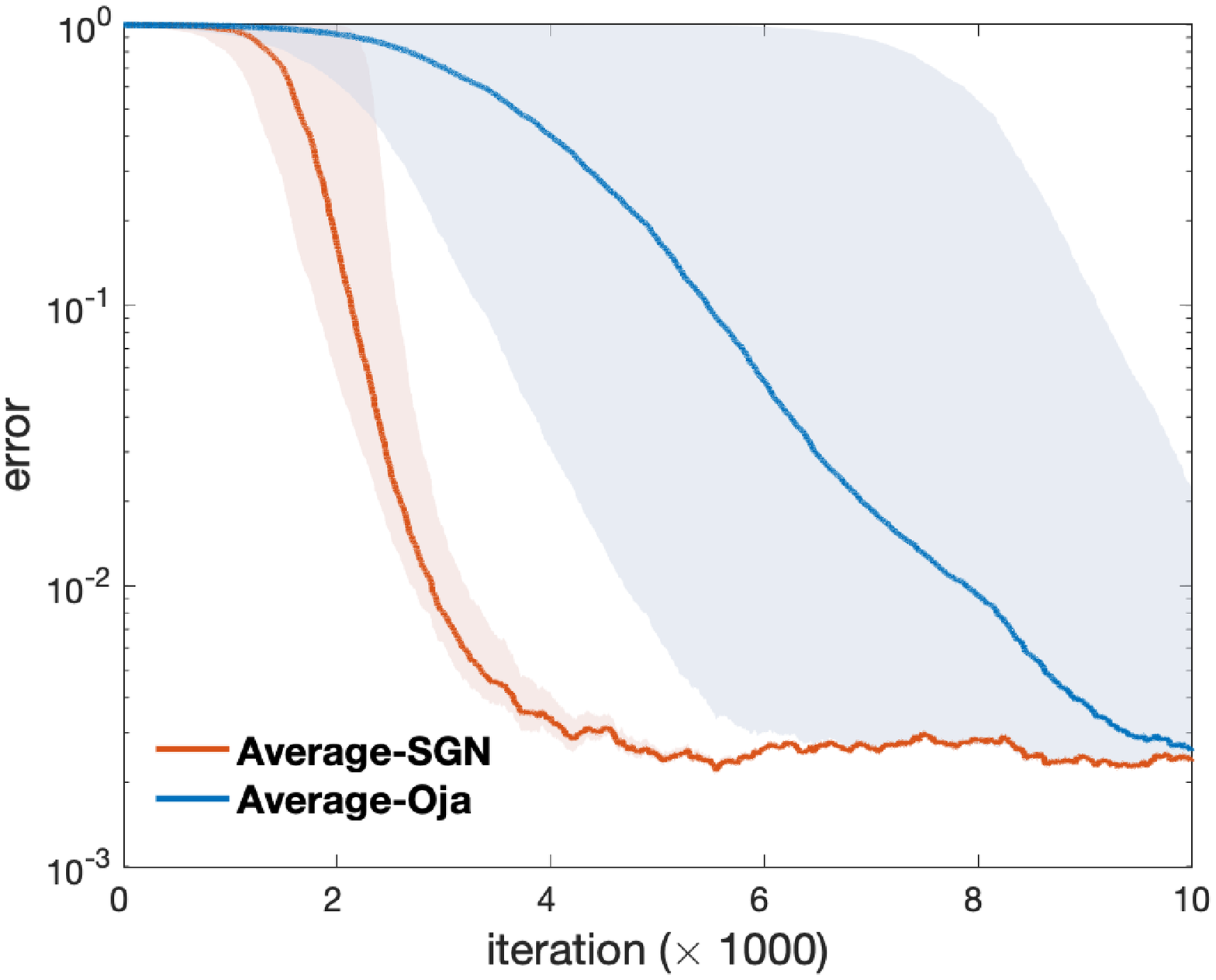}
		\end{minipage}
	}
	\caption{The estimation error of Oja's iteration and
		SGN using constant stepsizes $\alpha^{(k)}=\alpha$ on simulated data with $\mu_{1}=1$ from $100$ randomly generated initial points.}
	\label{fig:constSGN_simu}
\end{figure}

\begin{figure}[h]
	\centering
	\subfigure[$\mu_{1}=1$,~$\alpha \in I_{test}$]{
		\begin{minipage}[t]{4.5cm}
			\centering
			\includegraphics[width=4.5cm]{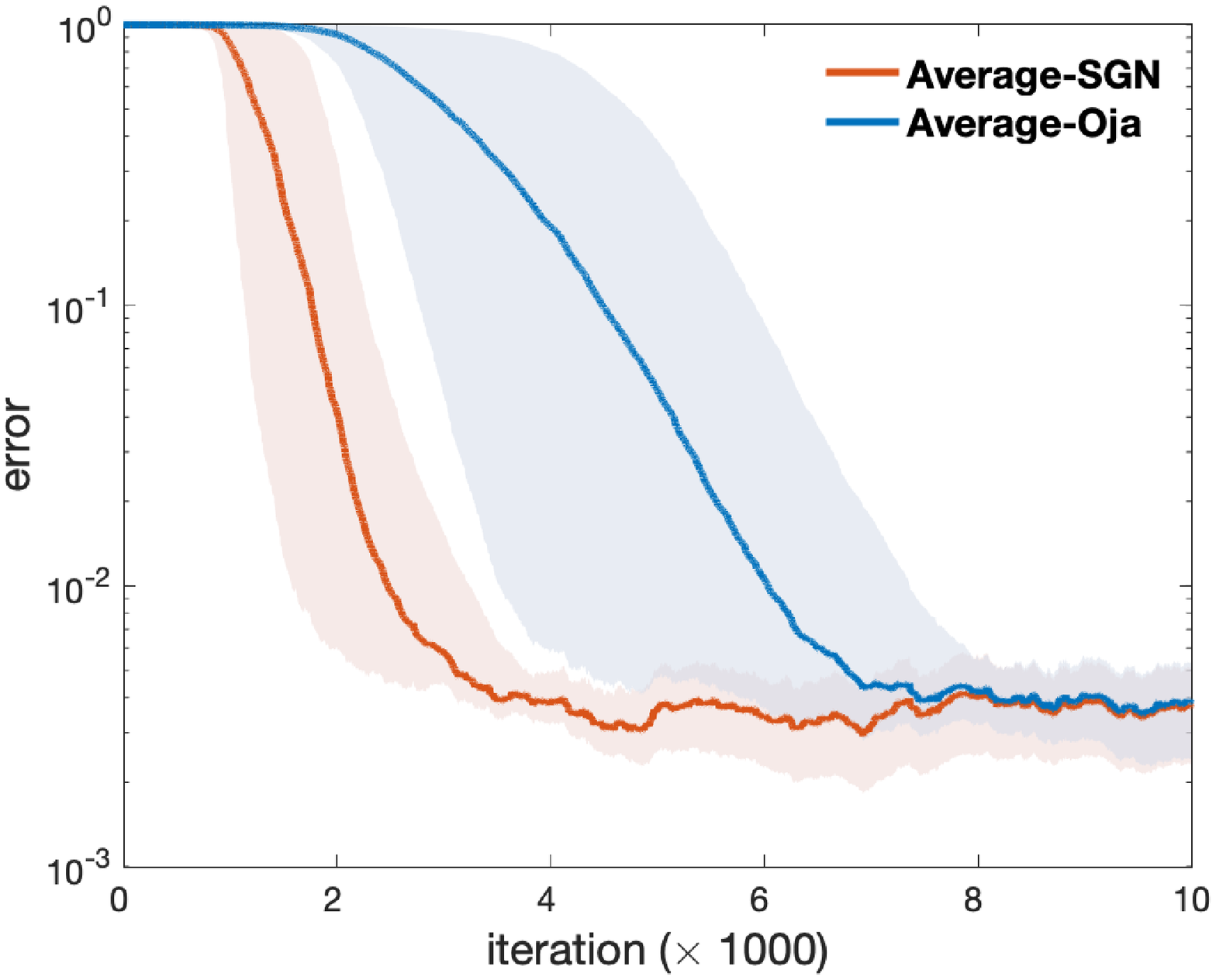}
		\end{minipage}
	}
	\subfigure[$\mu_{1}=10$,~$\alpha \in I_{test}$]{
		\begin{minipage}[t]{4.5cm}
			\centering
			\includegraphics[width=4.5cm]{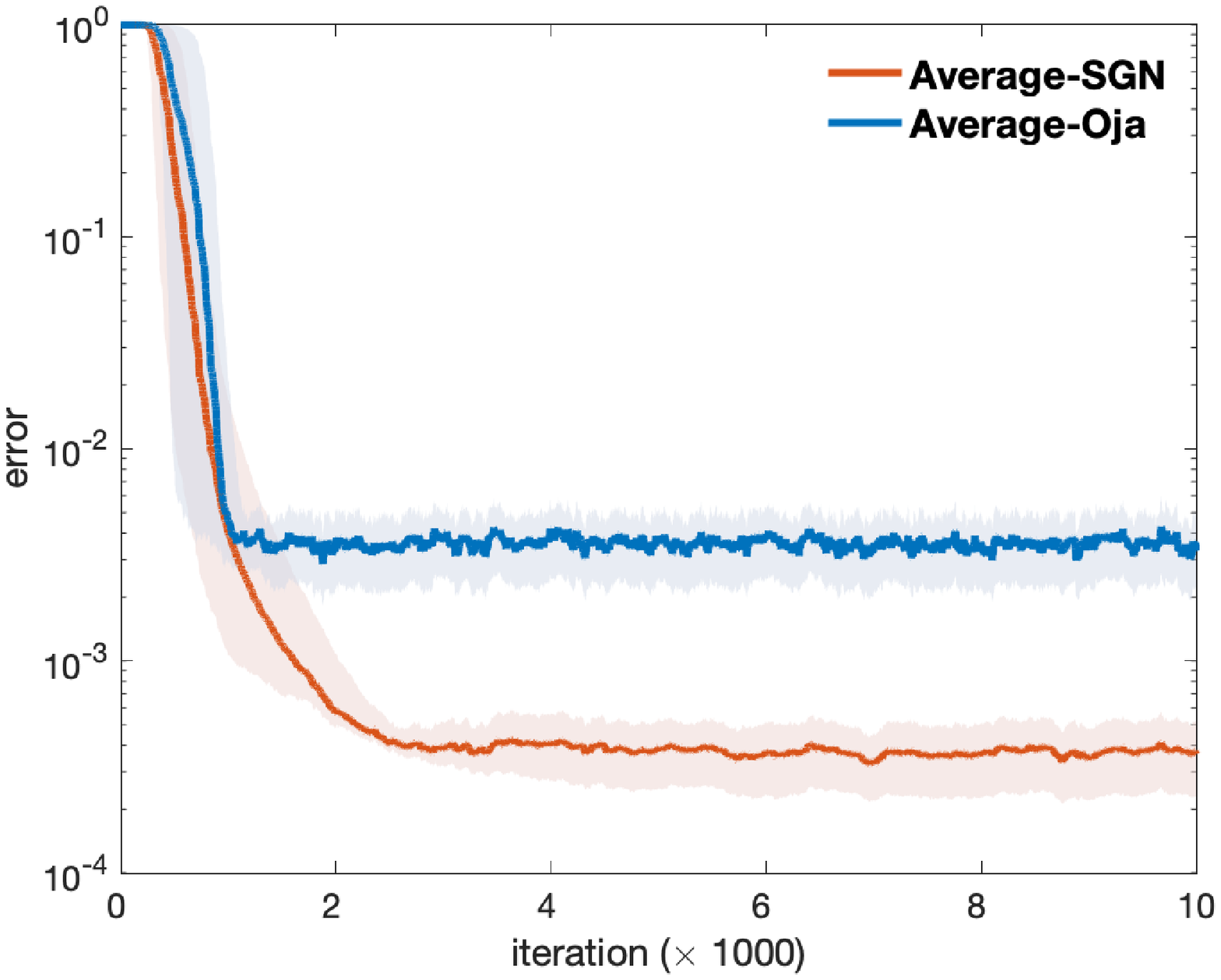}
		\end{minipage}
	}
	\subfigure[$\mu_{1}=100$,~$\alpha \in I_{test}$]{
		\begin{minipage}[t]{4.5cm}
			\centering
			\includegraphics[width=4.5cm]{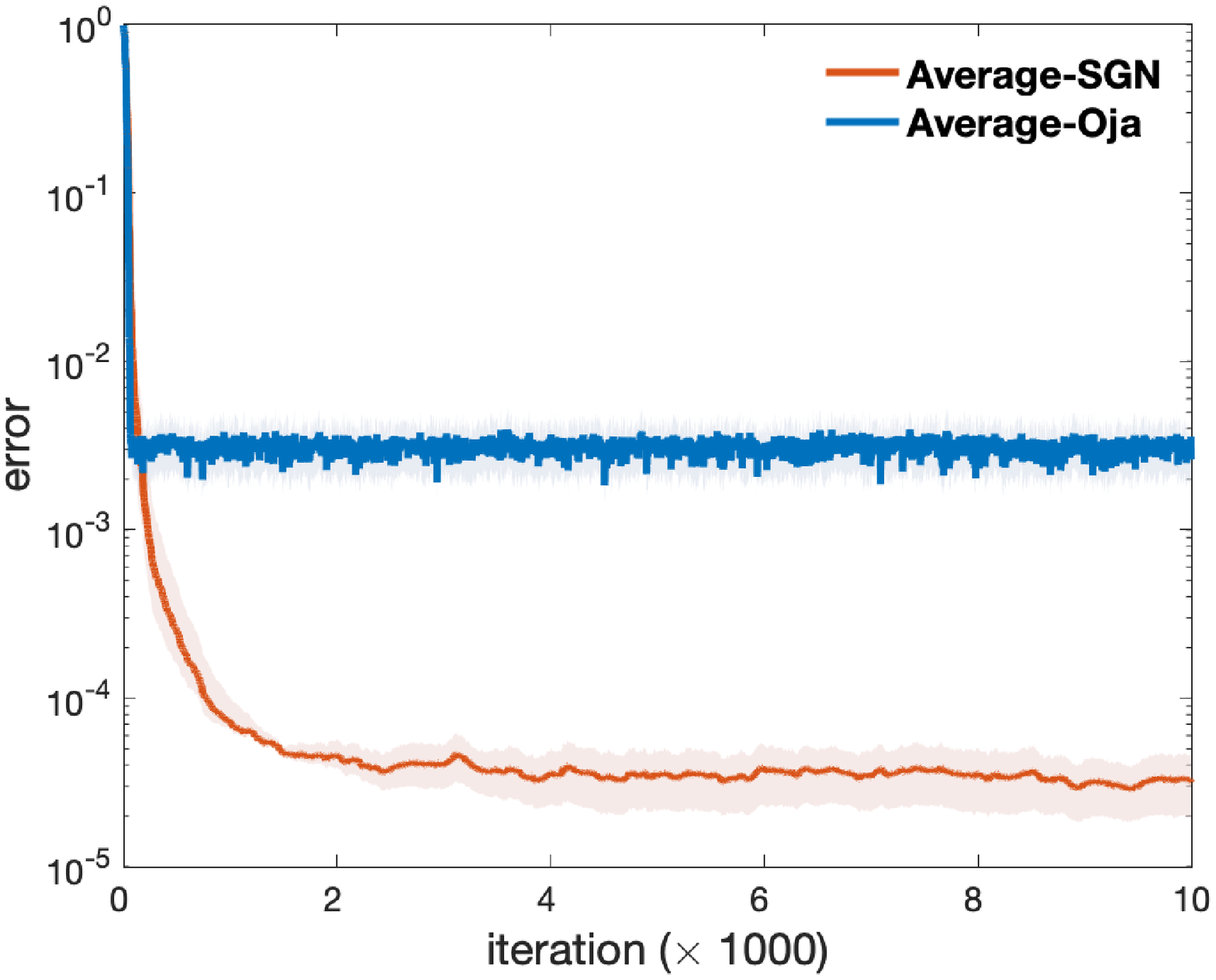}
		\end{minipage}
	}

	\subfigure[$\mu_{1}=10^4$,~$\alpha \in I_{test}$]{
		\begin{minipage}[t]{4.5cm}
			\centering
			\includegraphics[width=4.5cm]{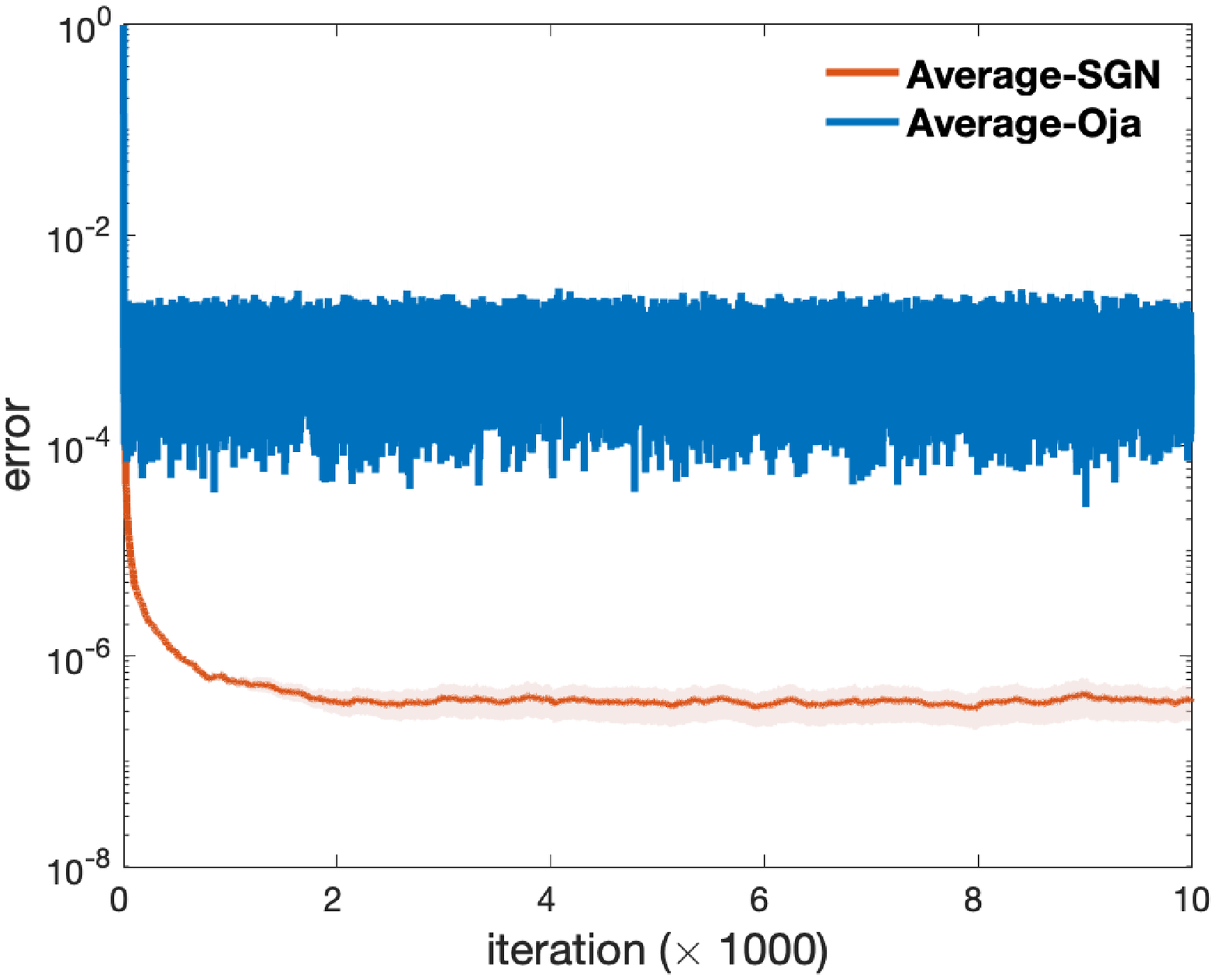}
		\end{minipage}
	}
	\subfigure[$\mu_{1}=10^4$,~$\alpha \in I_{test}/10$]{
		\begin{minipage}[t]{4.5cm}
			\centering
			\includegraphics[width=4.5cm]{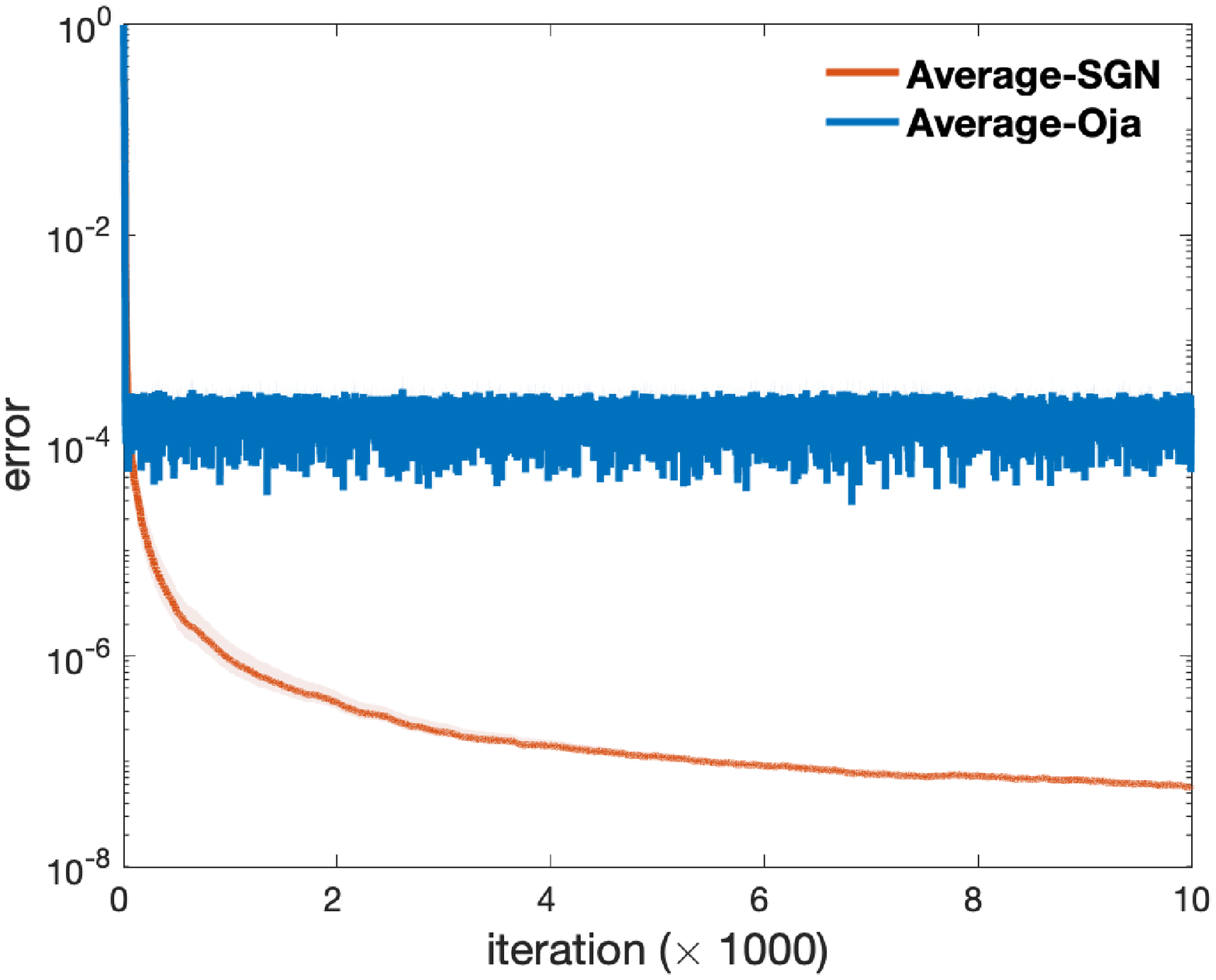}
		\end{minipage}
	}
	\subfigure[$\mu_{1}=10^4$,~$\alpha \in I_{test}/100$]{
		\begin{minipage}[t]{4.5cm}
			\centering
			\includegraphics[width=4.5cm]{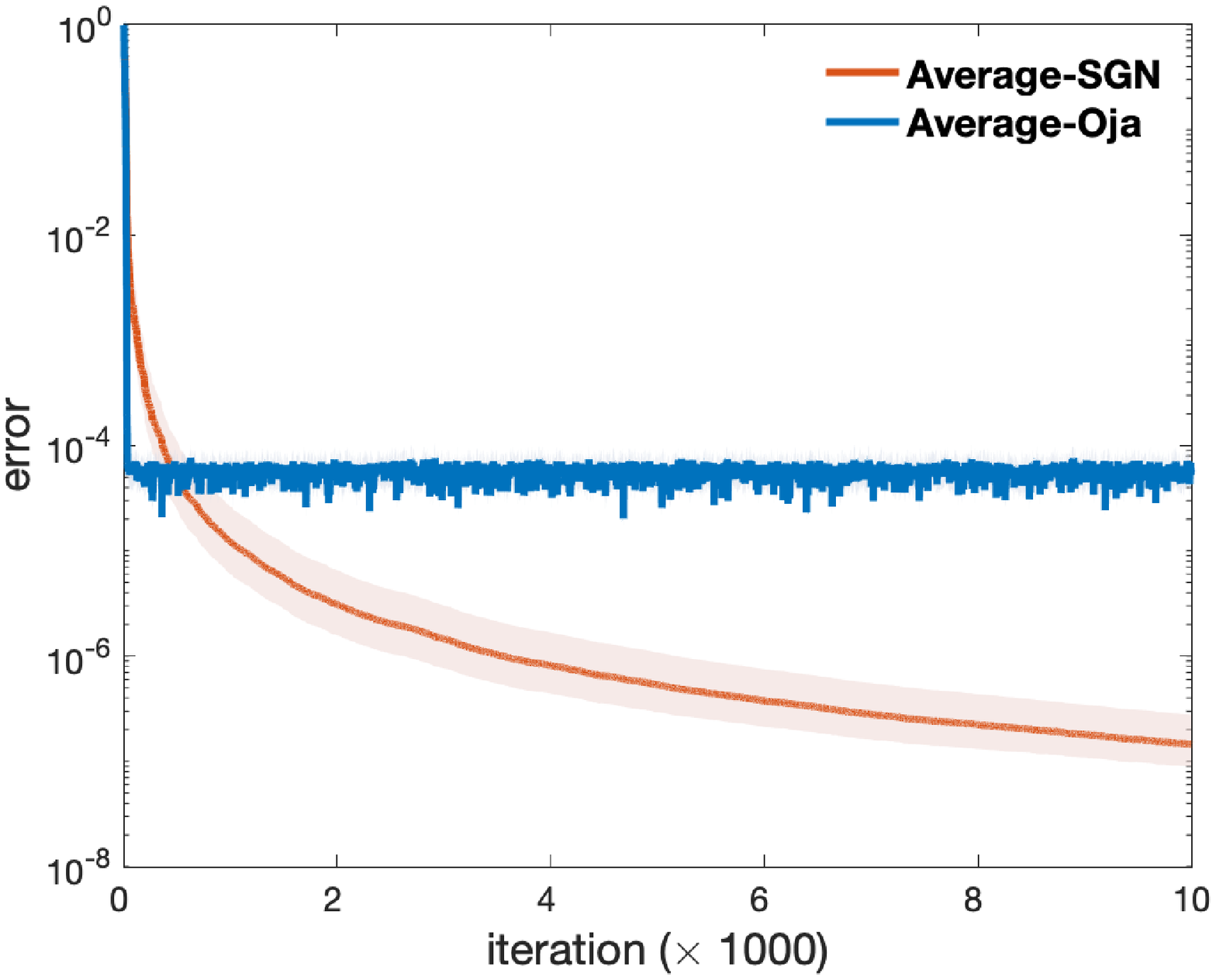}
		\end{minipage}
	}
	\caption{The estimation error of Oja's iteration and
		SGN using  $\alpha^{(k)}=\alpha$ with $100$ evenly spaced $\alpha \in I_{test}=[9\times 10^{-4}, 2\times 10^{-3}], I_{test}/10, I_{test}/100$~starting at one saddle point.}
	\label{fig:constSGN_simu_2}
\end{figure}

\subsection{Comparison on diminishing stepsizes}\label{sec:numeri-sgn}
In this subsection, we are going to demonstrate the robustness of SGN over Oja's iteration when adopting the diminishing stepsize $\alpha^{(k)}= \gamma/(k+1)$ in terms of the performance as $\gamma$ varies and the optimally-tuned $\gamma$ from $\{2^{-5}, 2^{-4},\cdots, 2^{4}, 2^{5}\}$.

In \cref{fig:T1gau1,fig:T1gau2,fig:T1cifar}, the gradation in color indicates the change of parameter $\gamma$ and the deeper-colored dotted line represents the larger value of parameter $\gamma$ used.
The performance of the proposed SGN is shown in the warm tone while the one of Oja's iteration is shown in the cold tone.
Here, we focus on the overall performance of two algorithms w.r.t. the variation in $\gamma$.
In addition, the optimal performer is highlighted in sold line with corresponding $\gamma$ being bolded in the legend.
\cref{fig:T1gau1} exhibits the results of processing \texttt{Gau-gap-1} with $\bar{\mu}=10$. 
When the number of PCs (i.e., $p$) increases, on one side, SGN exhibits a better performance than Oja's iteration in overall performance. On the other side, the optimal $\gamma$ for Oja's iteration goes from $2^{-2}$ to $2^{5}$ in the single-pass model while the one for SGN adimts a slight change from $1$ to $2$ in the single-pass model and stays in the value of one in the mini-batch model.
Similar results could be observed in \cref{fig:T1cifar} for \texttt{CIFAR-$10$} dataset.
Moreover, in \texttt{Gau-gap-2}, we fix the gap and $p$, and change the distribution of top $p$ eigenvalues. As being seen in \cref{fig:T1gau2}, Oja's iteration becomes dramatically worse as $p_{1}$ increases, while SGN could manage it within the given parameter set. To save space, the results of other datasets listed in \cref{sec:numeri-problem} are not provided here, which are in a good agreement with the above observations.

Further, we present the best parameter $\gamma\in\{2^{-5}, 2^{-4},\cdots, 2^{4}, 2^{5}\}$ in handling various datasets in \cref{fig:bestsimu} and \cref{fig:bestreal}. 
The color of the area $\{\gamma\leq 1\}$ is filled with gray.
It is apparent that the value of the best $\gamma$ for SGN varies smoothly in both different datasets and number of PCs, and however the one for Oja's iteration changes greatly. And interestingly, a larger $p$ always calls for a larger $\gamma$.

We have also performed tests on the diminishing stepsize \cref{eq:dimi} with different choices of $\beta$, which produce the results similar to the case of $\beta=1$ as shown in \cref{fig:T1gau1}-\ref{fig:bestreal}.
For brevity,, these results are not presented in this paper.

\begin{figure}[h]
	\centering
	\subfigure[$h=1$,~$p=1$]{
		\begin{minipage}[t]{0.47\linewidth}
			\includegraphics[width=6cm,trim=0.75cm 3cm 1.8cm 2.5cm,clip]{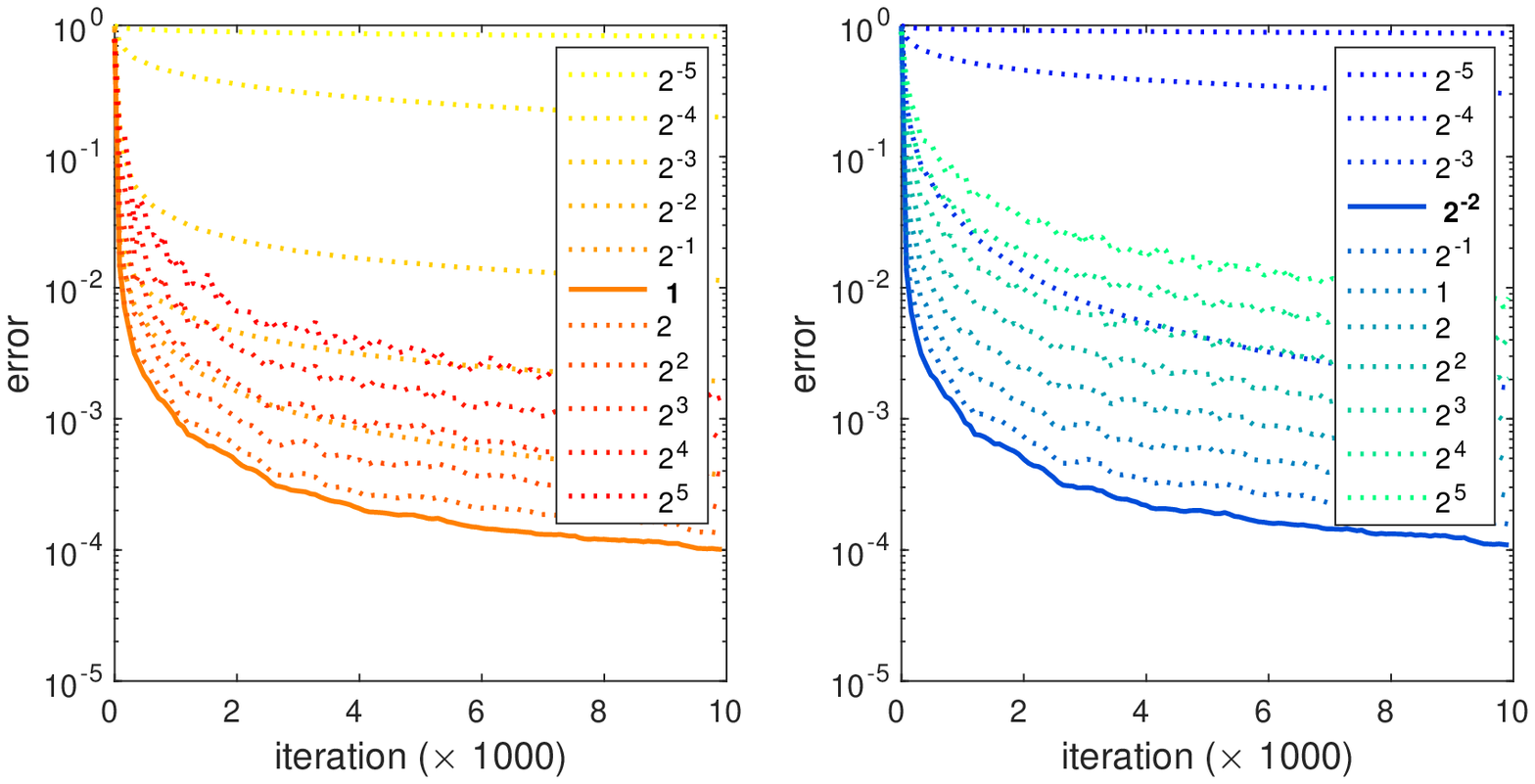}
		\end{minipage}
	}
	\subfigure[$h=10$,~$p=1$]{
		\begin{minipage}[t]{0.47\linewidth}
			\includegraphics[width=6cm,trim=0.75cm 3cm 1.8cm 2.5cm,clip]{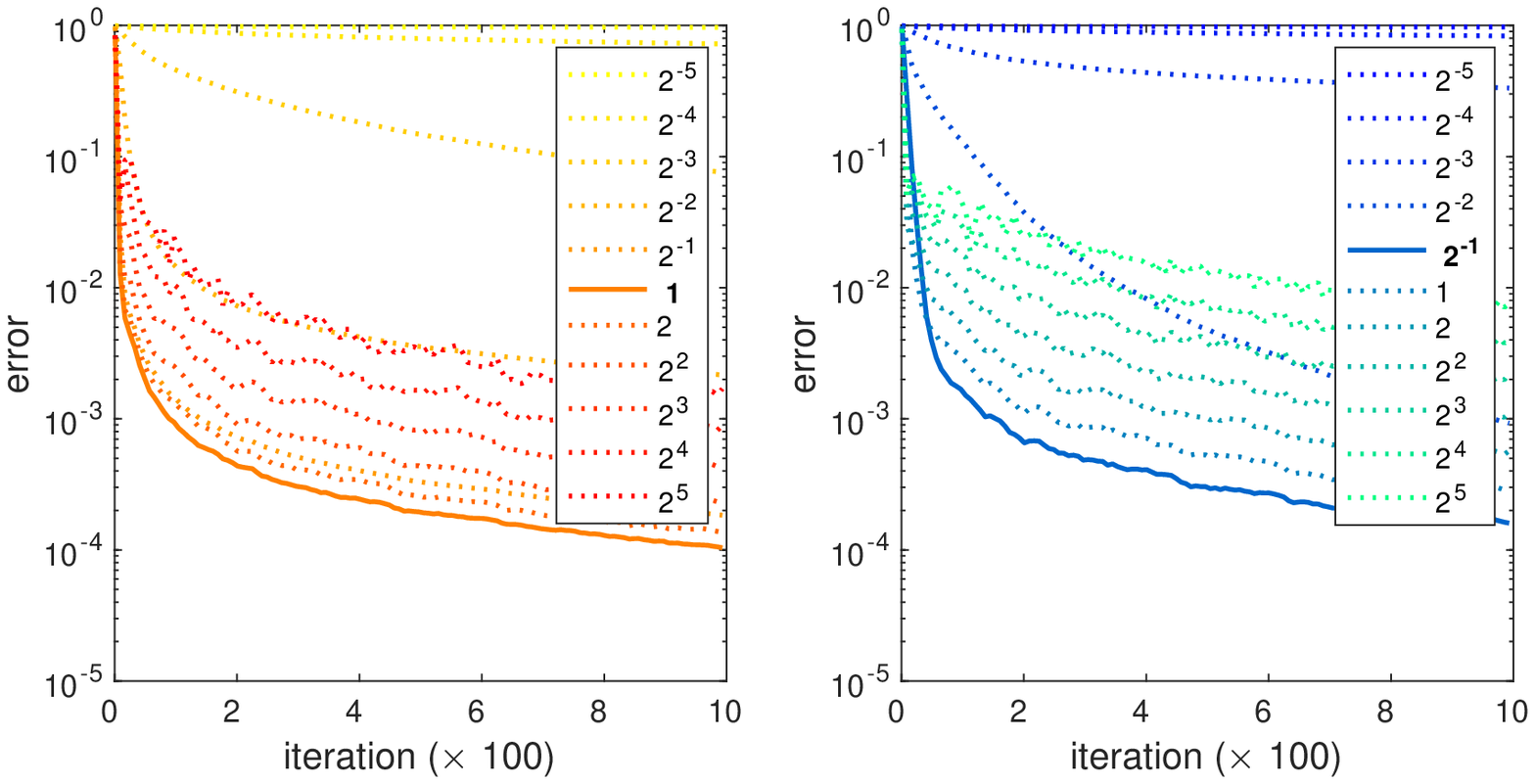}
		\end{minipage}
	}
	\newline
	\subfigure[$h=1$,~$p=30$]{
		\begin{minipage}[t]{0.47\linewidth}
			\includegraphics[width=6cm,trim=0.75cm 3cm 1.8cm 2.5cm,clip]{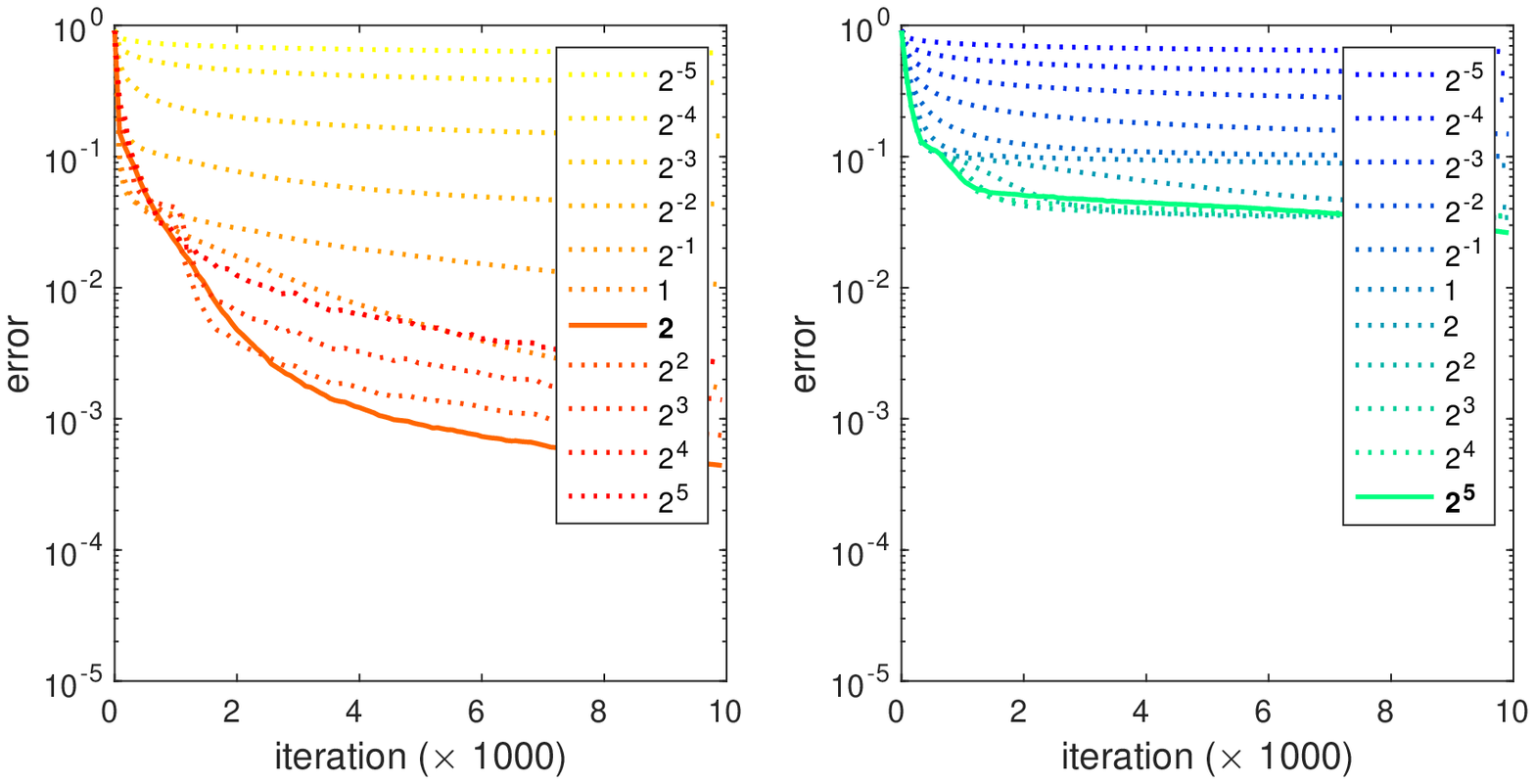}
		\end{minipage}
	}
	\subfigure[$h=10$,~$p=30$]{
		\begin{minipage}[t]{0.47\linewidth}
			\includegraphics[width=6cm,trim=0.75cm 3cm 1.8cm 2.5cm,clip]{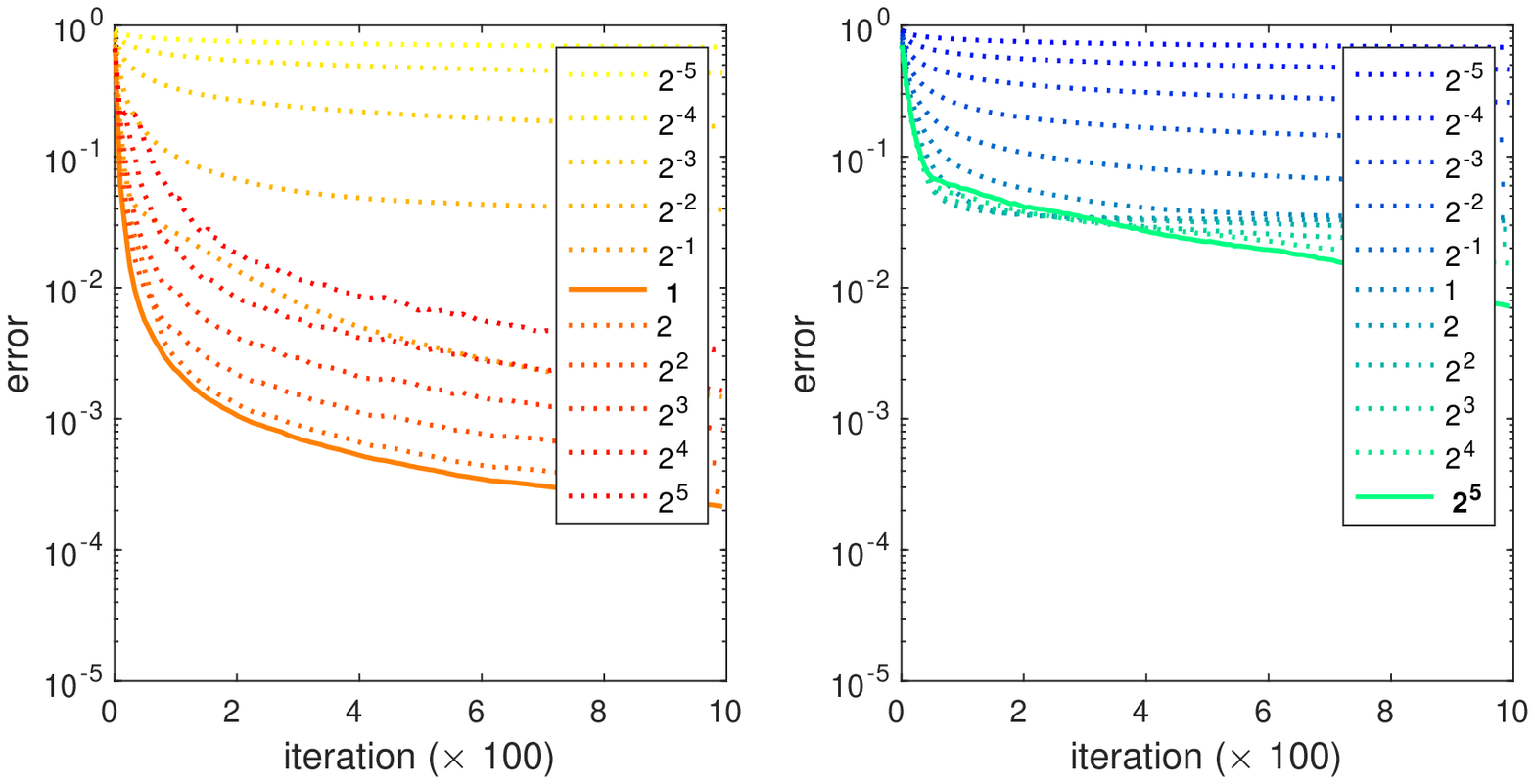}
		\end{minipage}
	}	
	\caption{The estimation errors of Oja's iteration (right in each subfigure) and SGN (left in each subfigure) with diminishing stepsizes when varying $\gamma$ from $2^{-5}$ to $2^5$ on \texttt{Gau-gap-1} with $\bar{\mu} = 10$.}
	\label{fig:T1gau1}
\end{figure}

\begin{figure}[h]
	\centering
	\subfigure[$p_{1} = 0~~~~~$]{
		\begin{minipage}[t]{0.47\linewidth}
			\includegraphics[width=6cm,trim=0.75cm 3cm 1.8cm 2.5cm,clip]{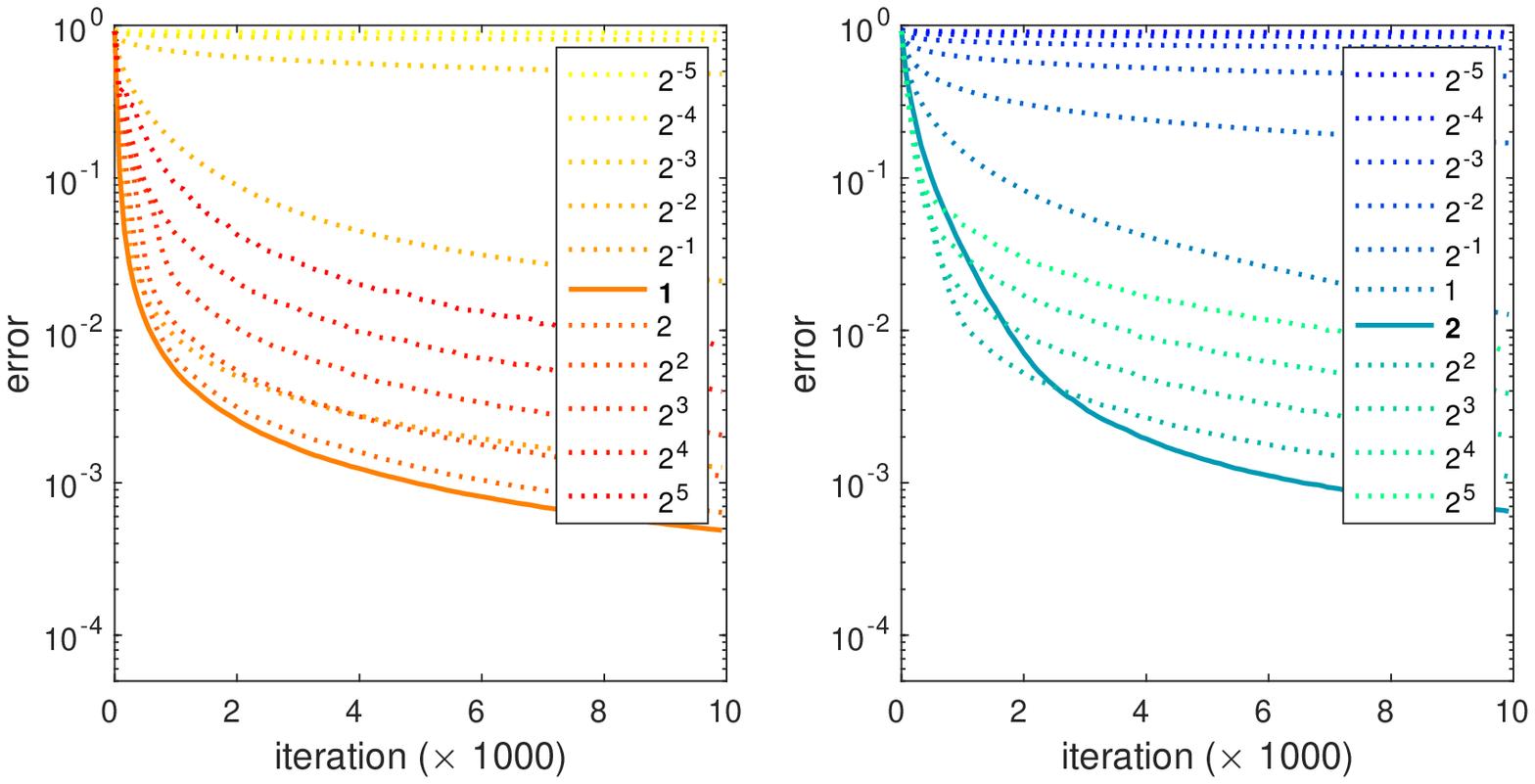}
		\end{minipage}
	}
	\subfigure[$p_{1} = 5~~~~~$]{
		\begin{minipage}[t]{0.47\linewidth}
			\includegraphics[width=6cm,trim=0.75cm 3cm 1.8cm 2.5cm,clip]{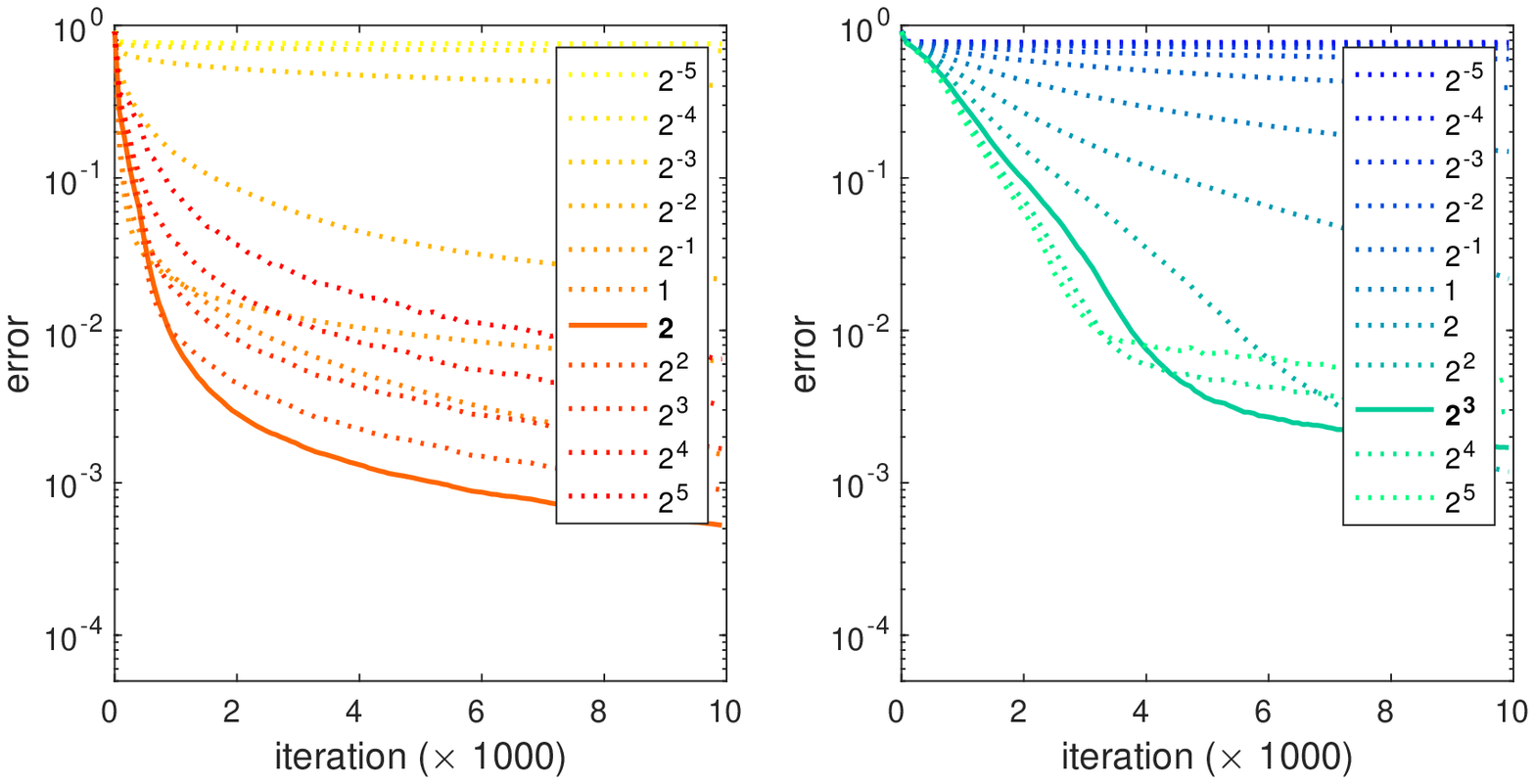}
		\end{minipage}
	}
	\newline
	\subfigure[$p_{1} = 15~~~~~$]{
		\begin{minipage}[t]{0.47\linewidth}
			\includegraphics[width=6cm,trim=0.75cm 3cm 1.8cm 2.5cm,clip]{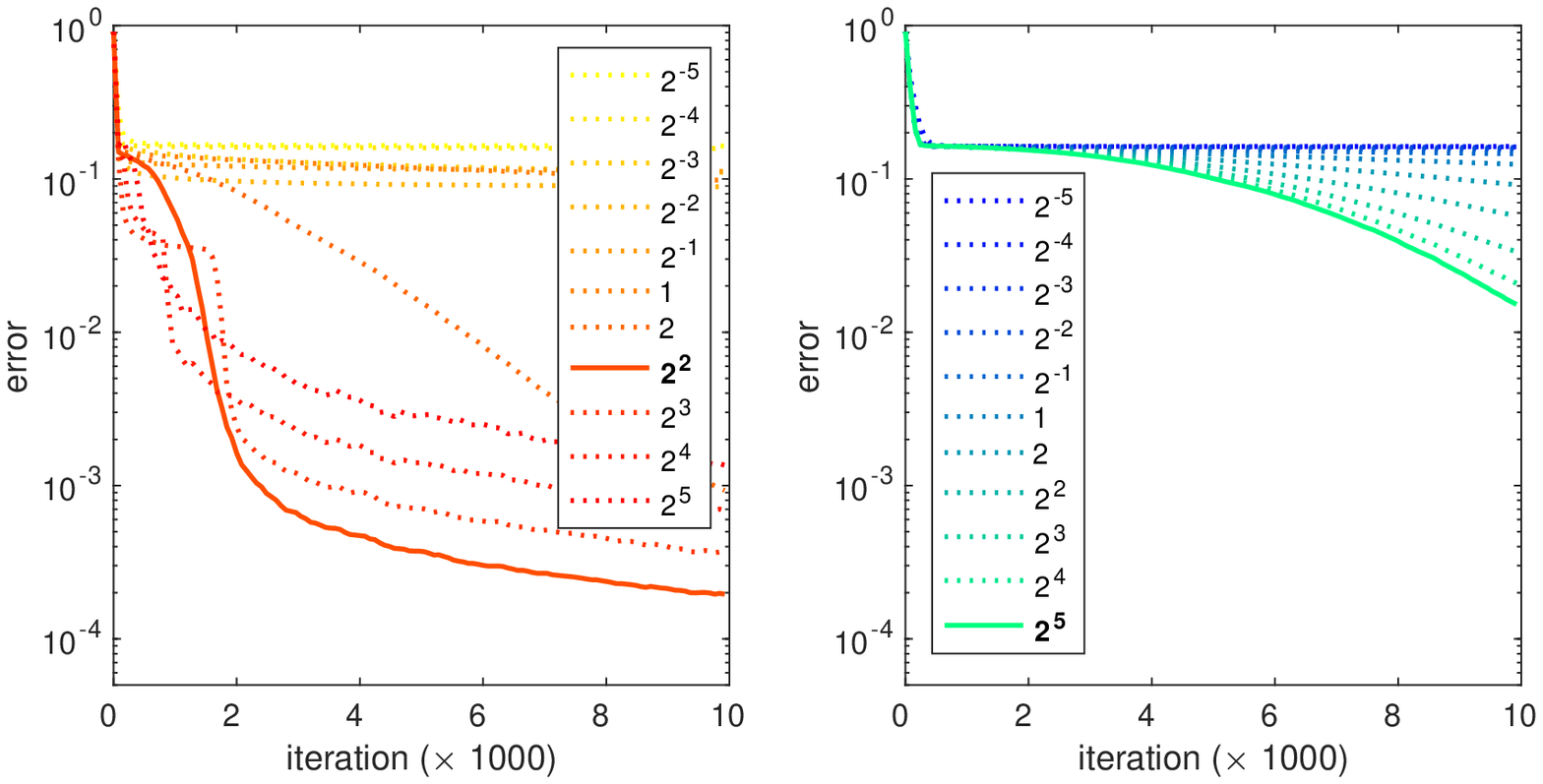}
		\end{minipage}
	}
	\subfigure[$p_{1} = 25~~~~~$]{
		\begin{minipage}[t]{0.47\linewidth}
			\includegraphics[width=6cm,trim=0.75cm 3cm 1.8cm 2.5cm,clip]{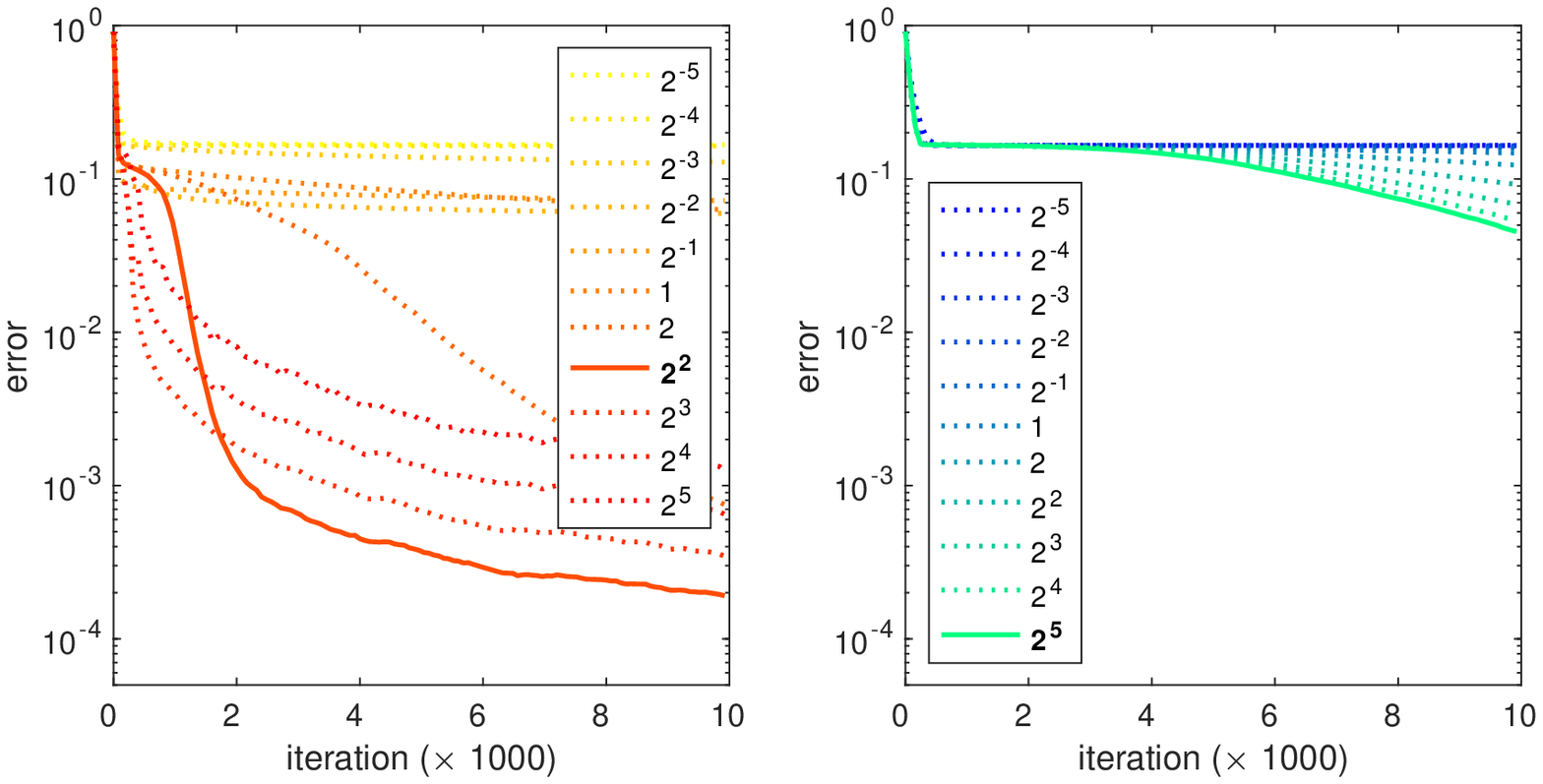}
		\end{minipage}
	}	
	\caption{The estimation errors of Oja's iteration (right in each subfigure) and SGN (left in each subfigure) with diminishing stepsizes when varying $\gamma$ from $2^{-5}$ to $2^5$ on \texttt{Gau-gap-2} with $h=1$.}
	\label{fig:T1gau2}
	\centering
	\subfigure[$h=1$,~$p=1~~~~~$]{
		\begin{minipage}[t]{0.47\linewidth}
			\includegraphics[width=5.9cm,trim=0.75cm 3cm 1.8cm 2.5cm,clip]{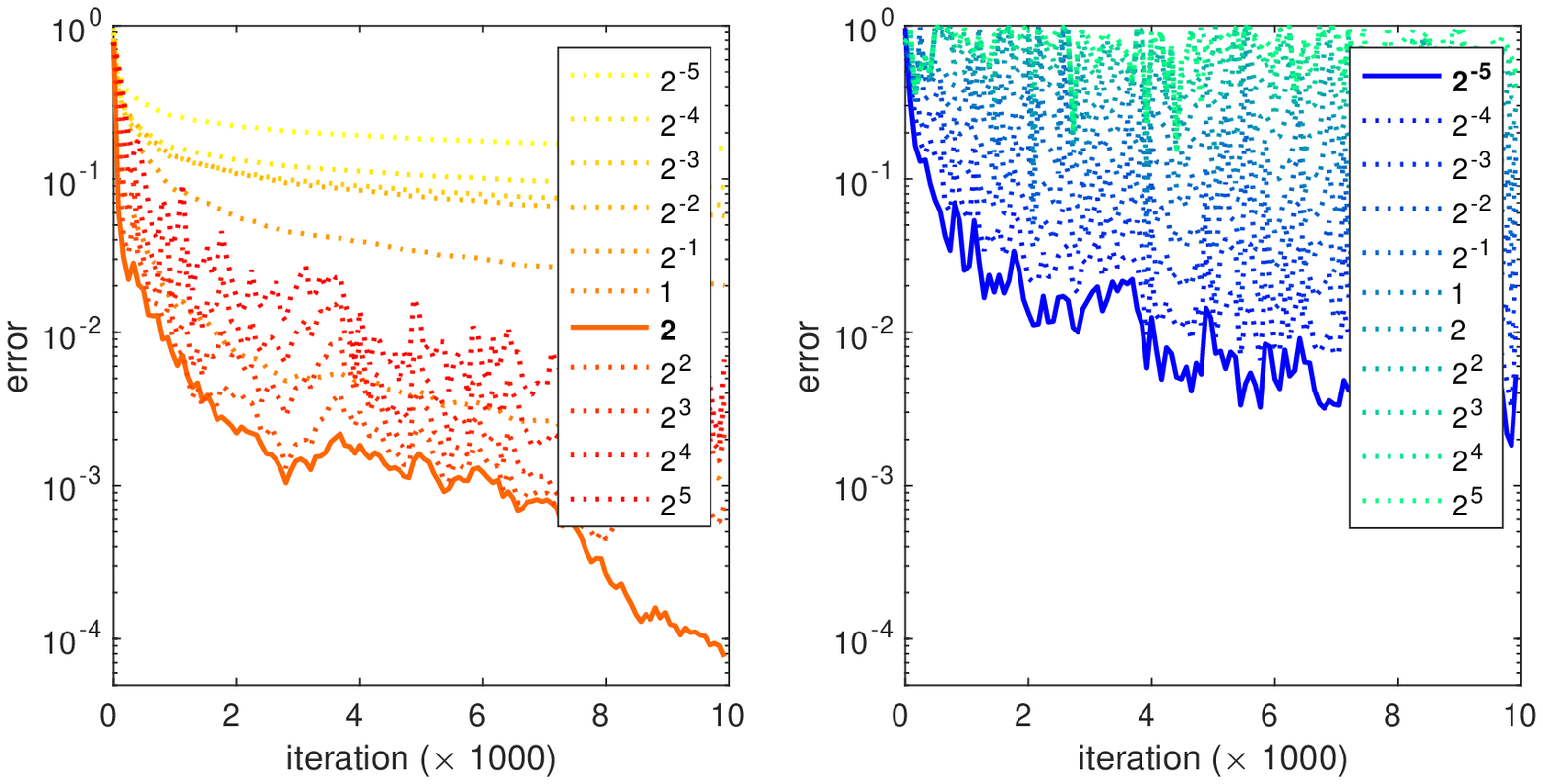}
		\end{minipage}
	}
	\subfigure[$h=10$,~$p=1~~~~~$]{
		\begin{minipage}[t]{0.47\linewidth}
			\includegraphics[width=5.9cm,trim=0.75cm 3cm 1.8cm 2.5cm,clip]{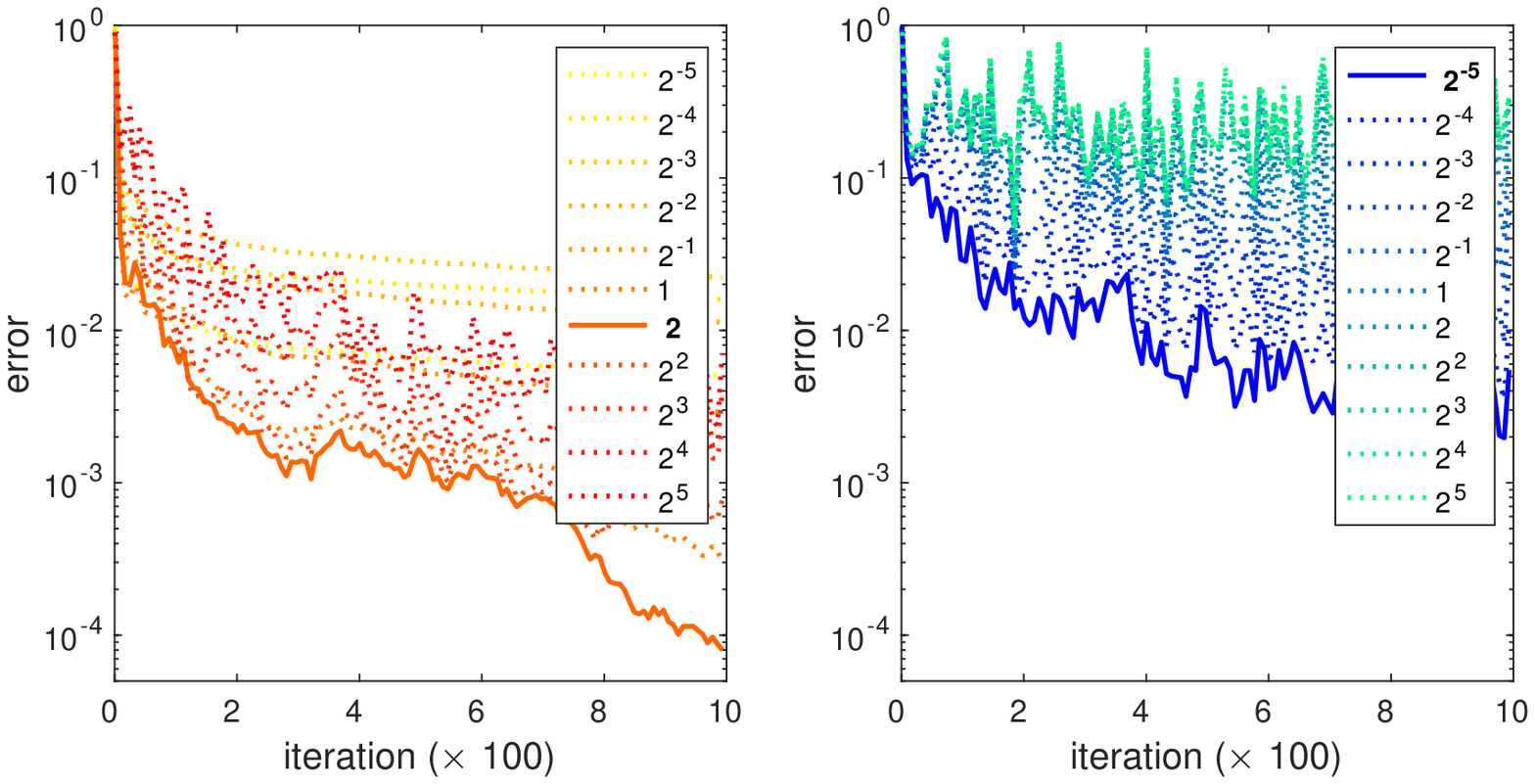}
		\end{minipage}
	}
	\newline
	\subfigure[$h=1$,~$p=30~~~~~$]{
		\begin{minipage}[t]{0.47\linewidth}
			\includegraphics[width=5.9cm,trim=0.75cm 3cm 1.8cm 2.5cm,clip]{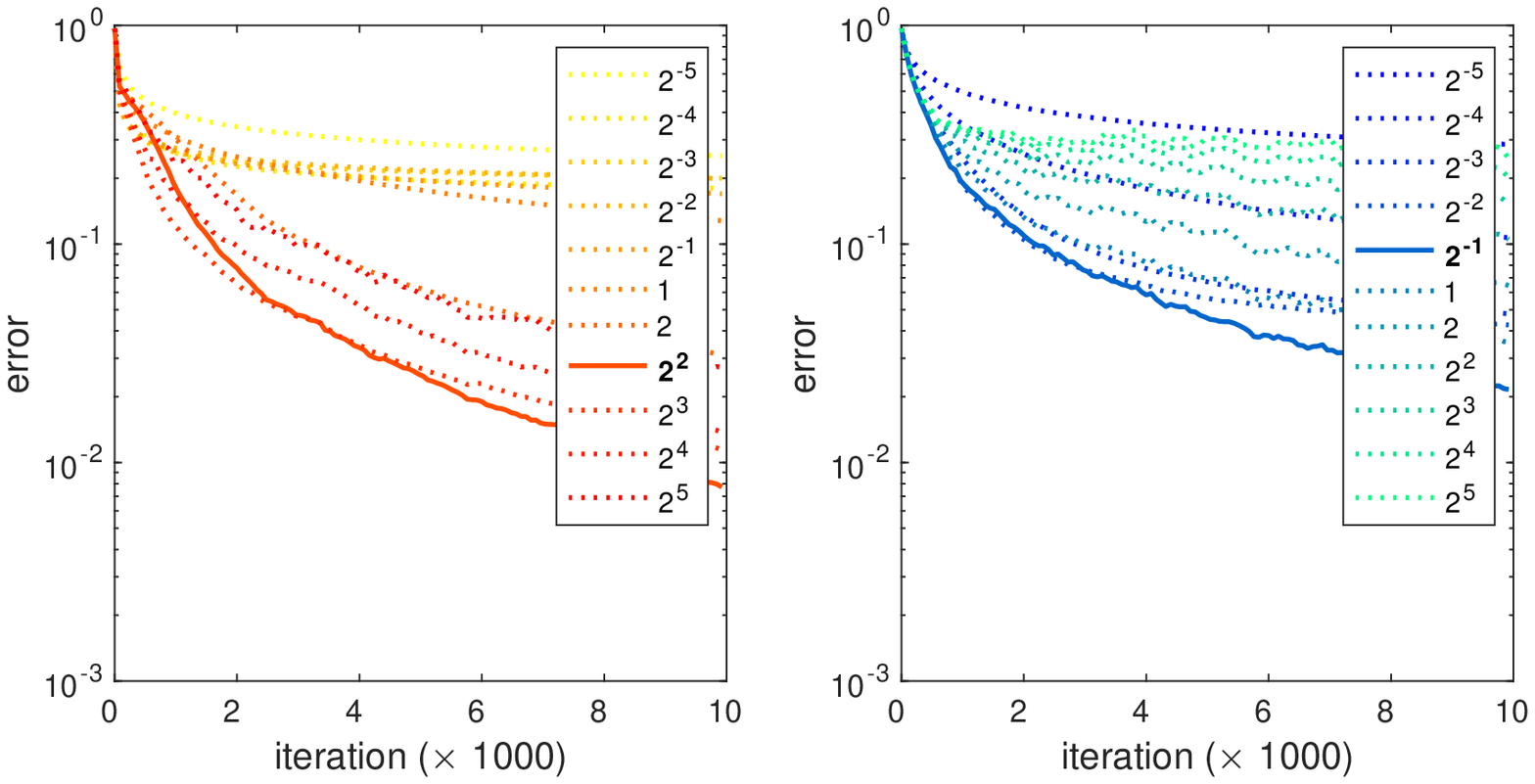}
		\end{minipage}
	}
	\subfigure[$h=10$,~$p=30~~~~~$]{
		\begin{minipage}[t]{0.47\linewidth}
			\includegraphics[width=5.9cm,trim=0.75cm 3cm 1.8cm 2.5cm,clip]{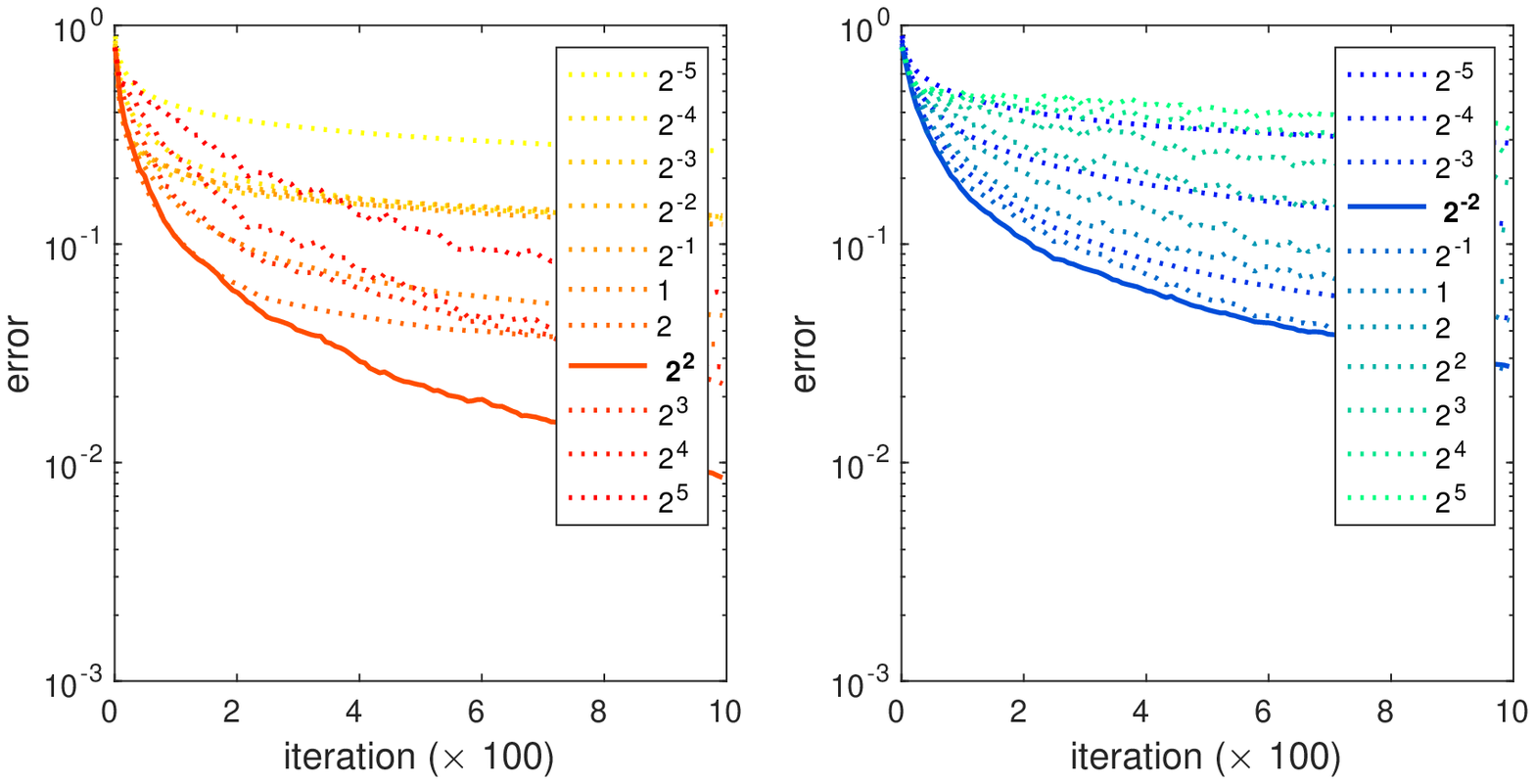}
		\end{minipage}
	}
	\caption{The estimation errors of Oja's iteration (right in each subfigure) and SGN (left in each subfigure) with diminishing stepsizes when varying $\gamma$ from $2^{-5}$ to $2^5$ on \texttt{CIFAR-$10$}.}
	\label{fig:T1cifar}
\end{figure}

\begin{figure}[h]
	\centering	
	\subfigure[$\bar{\mu}=1$\quad~~]{
		\begin{minipage}[t]{0.3\linewidth}
			\includegraphics[width=4.5cm]{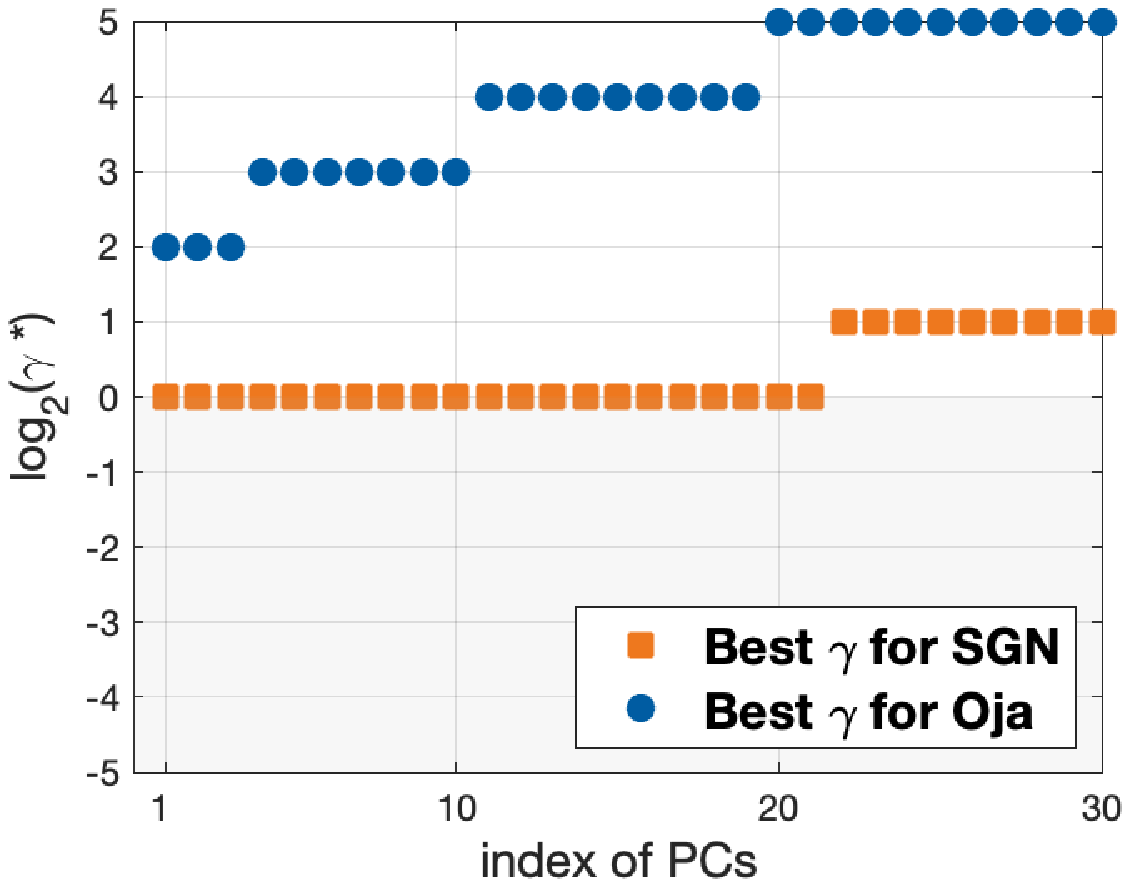}
		\end{minipage}
	}	
	\subfigure[$\bar{\mu}=10$\quad~~]{
		\begin{minipage}[t]{0.3\linewidth}
			\includegraphics[width=4.5cm]{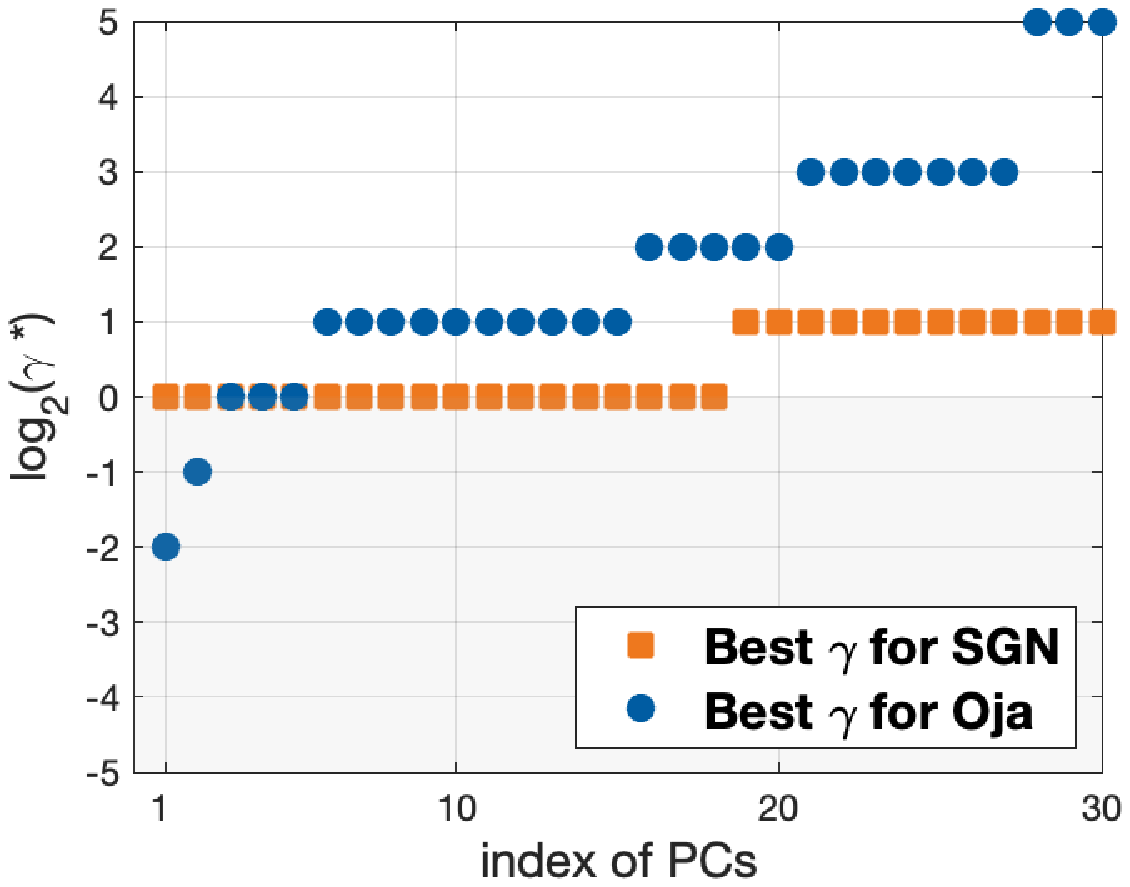}
		\end{minipage}
	}	
	\subfigure[$\bar{\mu}=100$\quad~~]{
		\begin{minipage}[t]{0.3\linewidth}
			\includegraphics[width=4.5cm]{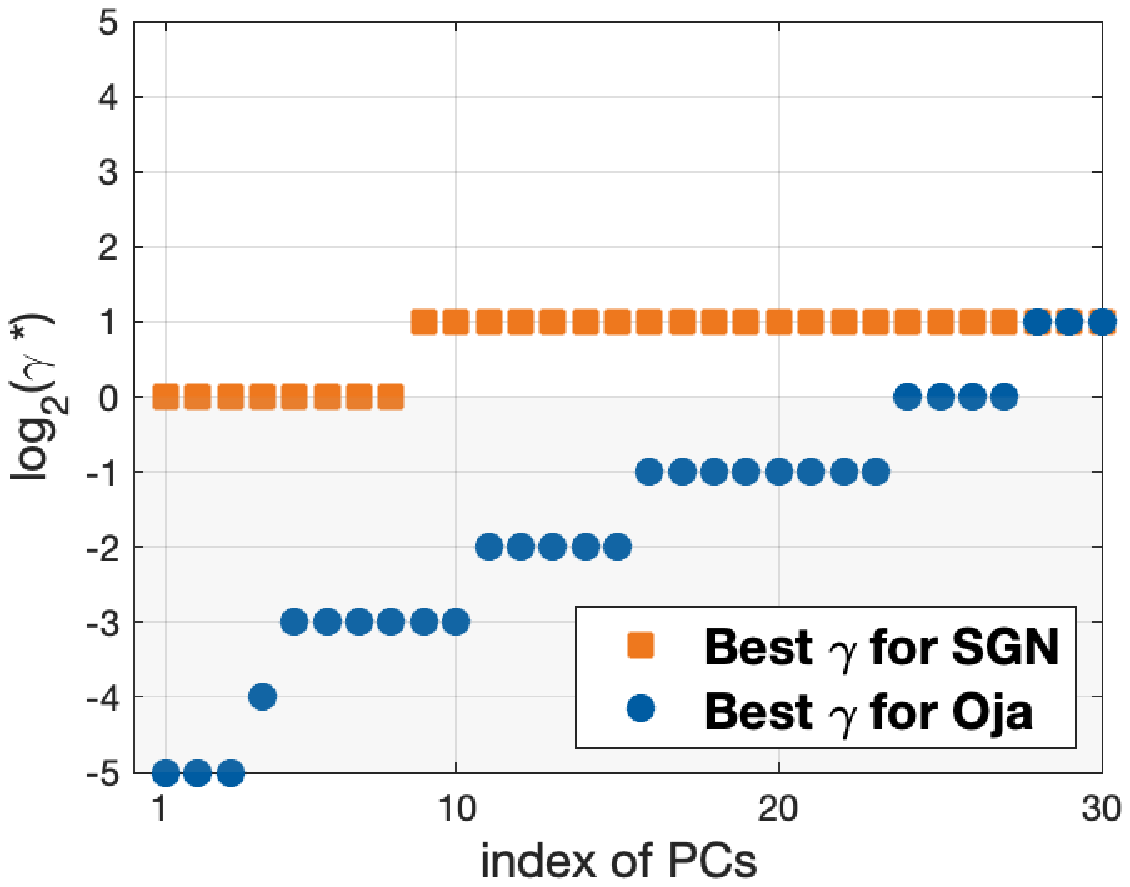}
		\end{minipage}
	}	
	\label{fig:bestsimu}
	\caption{The best $\gamma\in\{2^{-5},\cdots, 2^{5}\}$ for Oja's iteration and SGN on \texttt{Gau-gap-1} with $h=1$.}
\end{figure}

\begin{figure}[h]
	\centering
	\subfigure[\texttt{MNIST}\quad~~]{
		\begin{minipage}[t]{0.3\linewidth}
			\includegraphics[width=4.5cm]{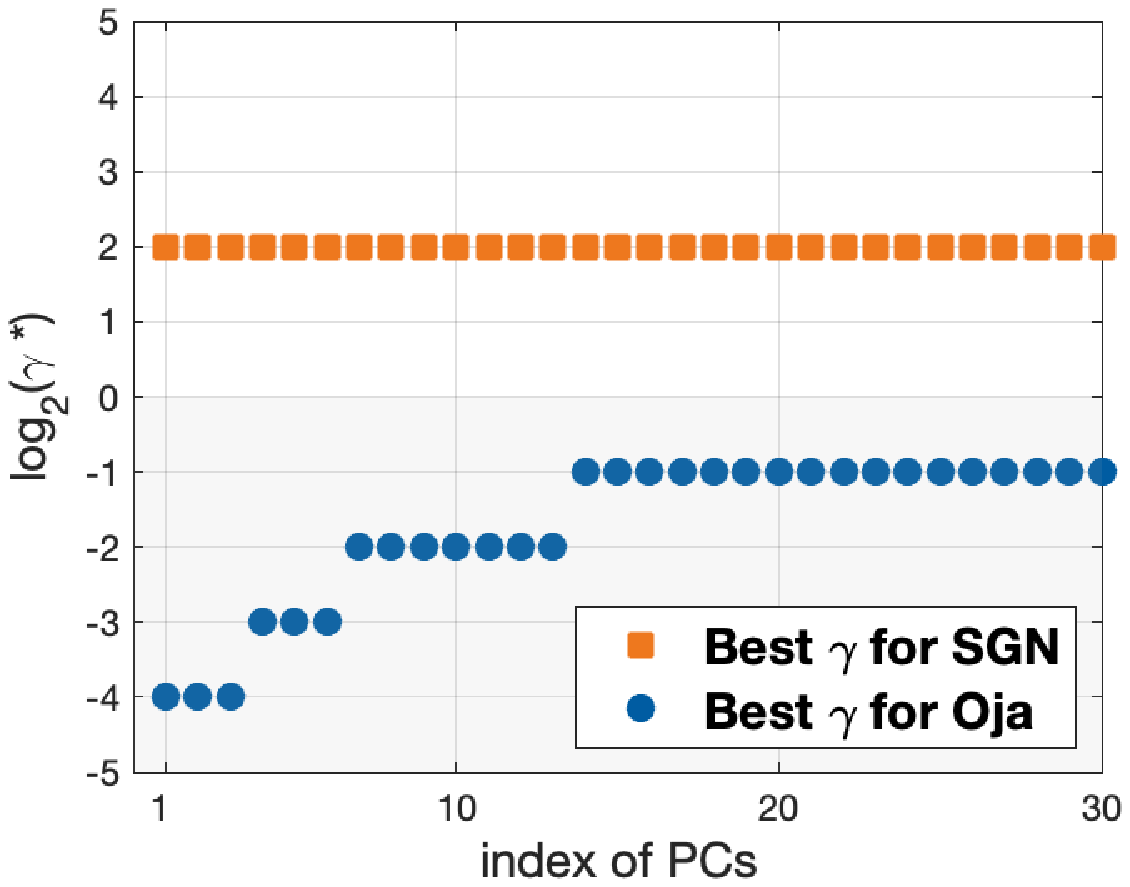}
		\end{minipage}
	}
	\subfigure[\texttt{Fashion-MNIST}\quad~~]{
		\begin{minipage}[t]{0.3\linewidth}
			\includegraphics[width=4.5cm]{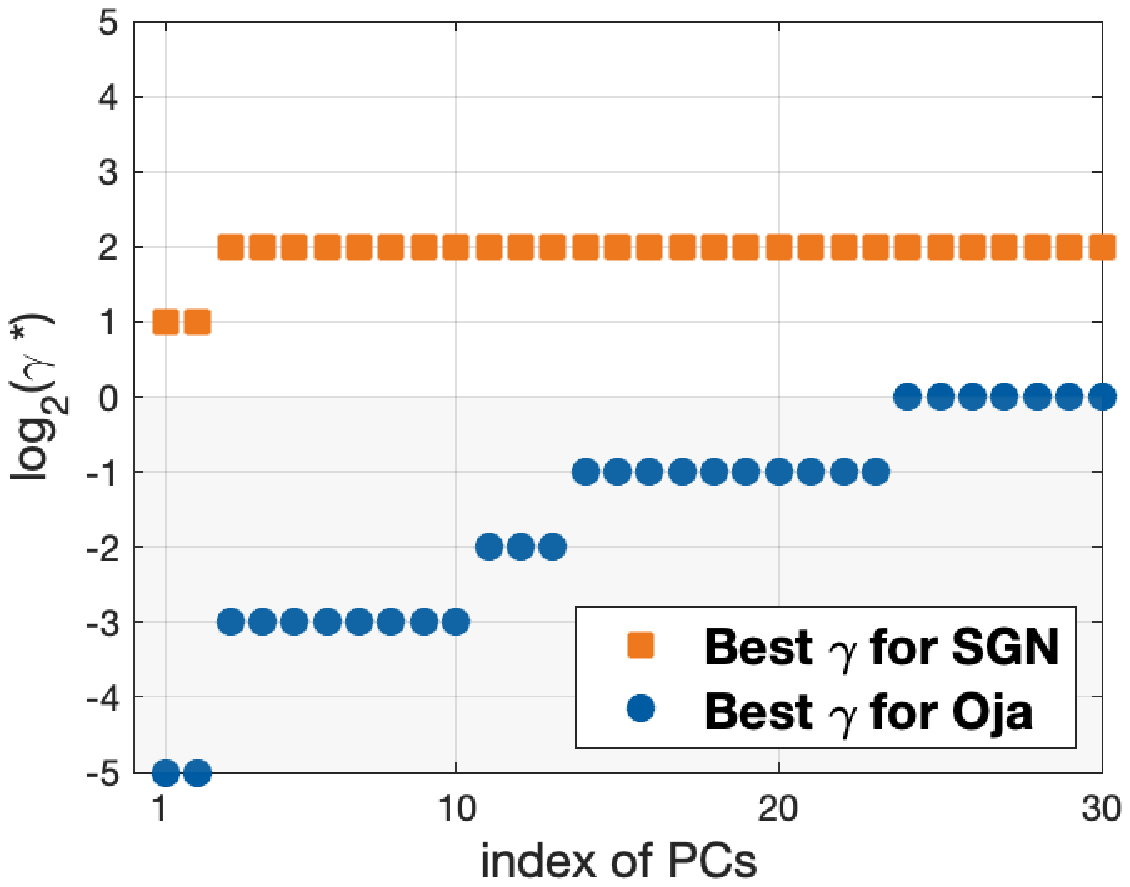}
		\end{minipage}
	}			
	\subfigure[\texttt{CIFAR-$10$}\quad~~]{
		\begin{minipage}[t]{0.3\linewidth}
			\includegraphics[width=4.5cm]{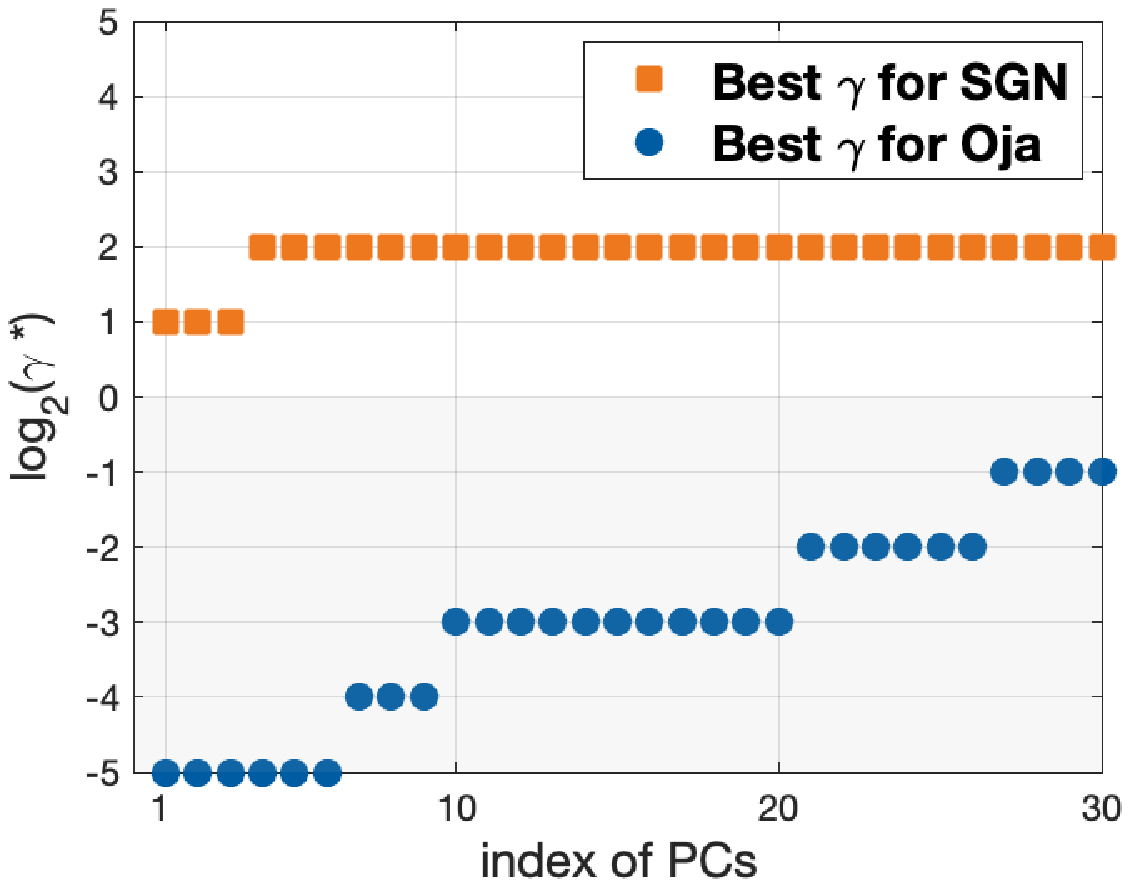}
		\end{minipage}
	}	
	\label{fig:bestreal}
	\caption{The best $\gamma\in\{2^{-5},\cdots, 2^{5}\}$ for Oja's iteration and SGN on real datasets with $h=1$.}
\end{figure}

\subsection{Comparison on adaptive stepsizes}\label{sec:numeri-adasgn}
Finally, we try to demonstrate the effectiveness of AdaSGN compared to
the diminishing-stepsize Oja's iteration and SGN with the optimally tuned $\gamma$ from $\{2^{-5}, 2^{-4},\cdots, 2^{4}, 2^{5}\}$ and AdaOja~\cref{eq:adaoja}.

In \cref{fig:adaSGN_simu} and \cref{fig:adaSGN_real}, AdaSGN shows comparable performance with manually tuned SGN while AdaOja
produces unstable results in different settings (for example, good performance in \cref{fig:adaSGN_simu}(c) but poor performance in \cref{fig:adaSGN_simu}(f)).
In addition, both SGN and AdaSGN exhibit excellent stability against the batchsize change but Oja's iteration and AdaOja suggest otherwise.
And an appropriate batchsize could make a certain improvement (for example, \cref{fig:adaSGN_simu}(d) and \cref{fig:adaSGN_simu}(e)).

\begin{figure}[h]
	\centering
	\subfigure[$h=1,~p=1$]{
		\begin{minipage}[t]{4.5cm}
			\centering
			\includegraphics[width=4.5cm]{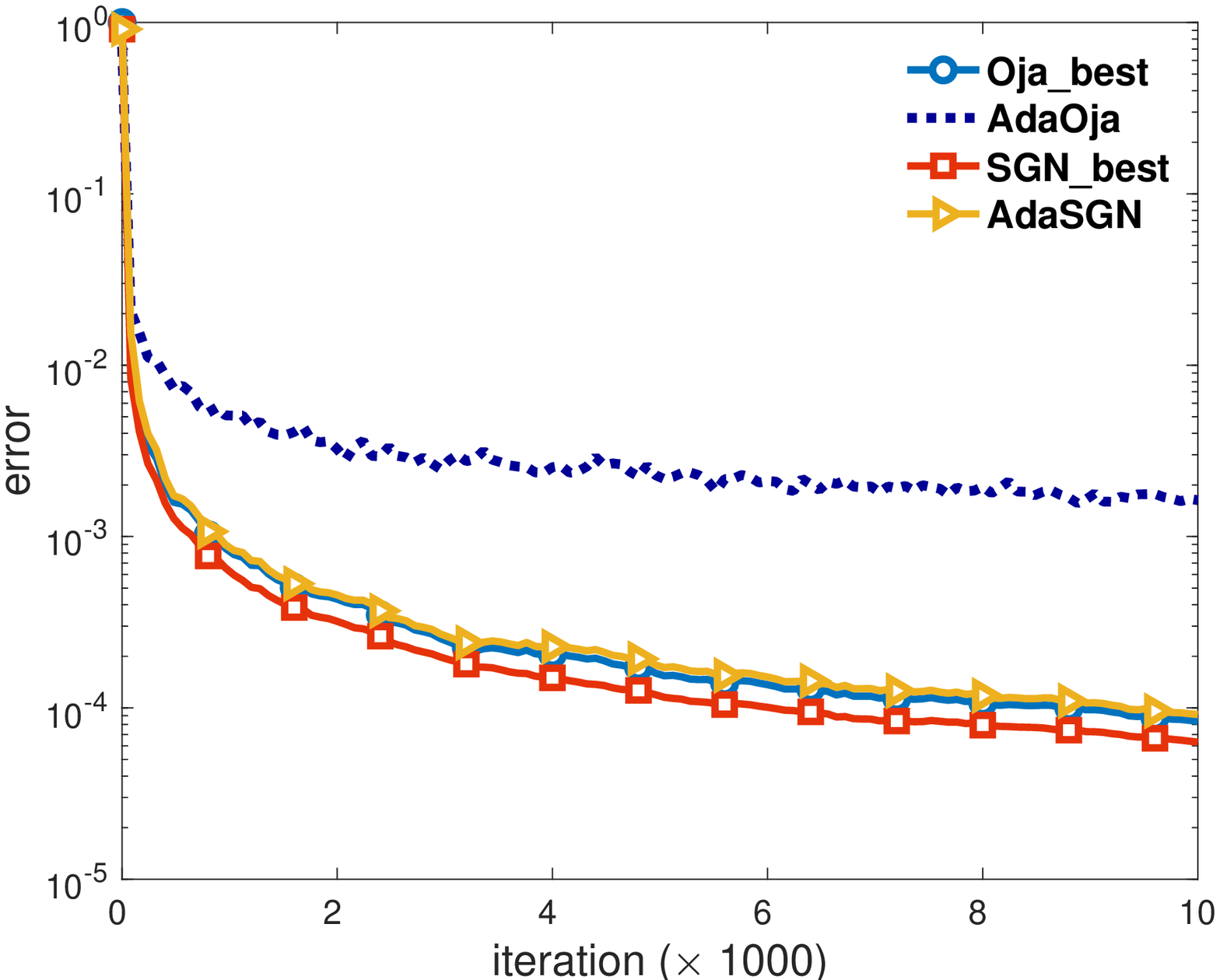}
		\end{minipage}
	}
	\subfigure[$h=10,~p=1$]{
		\begin{minipage}[t]{4.5cm}
			\centering
			\includegraphics[width=4.5cm]{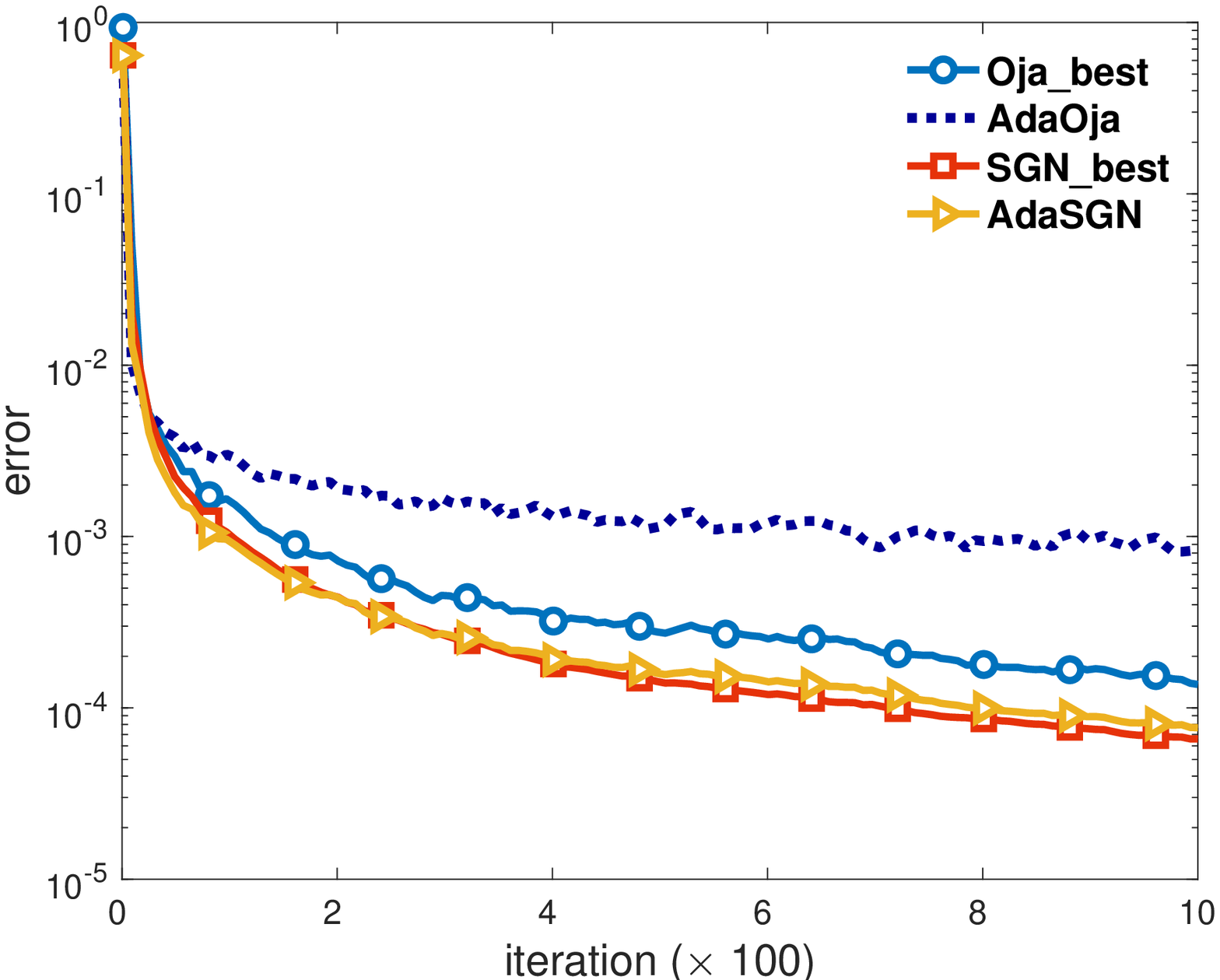}
		\end{minipage}
	}
	\subfigure[$h=100,~p=1$]{
		\begin{minipage}[t]{4.5cm}
			\centering
			\includegraphics[width=4.5cm]{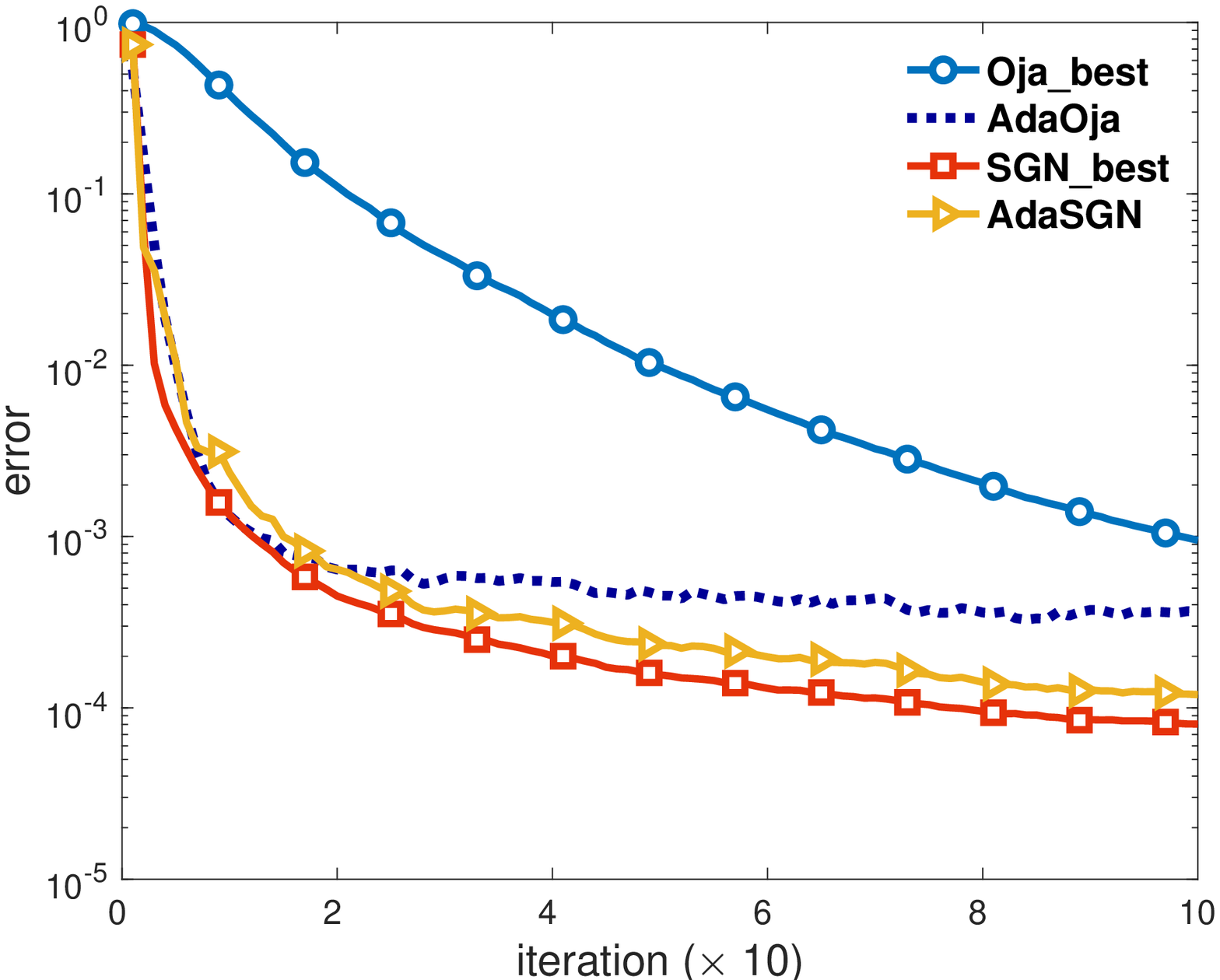}
		\end{minipage}
	}

	\subfigure[$h=1,~p=30$]{
		\begin{minipage}[t]{4.5cm}
			\centering
			\includegraphics[width=4.5cm]{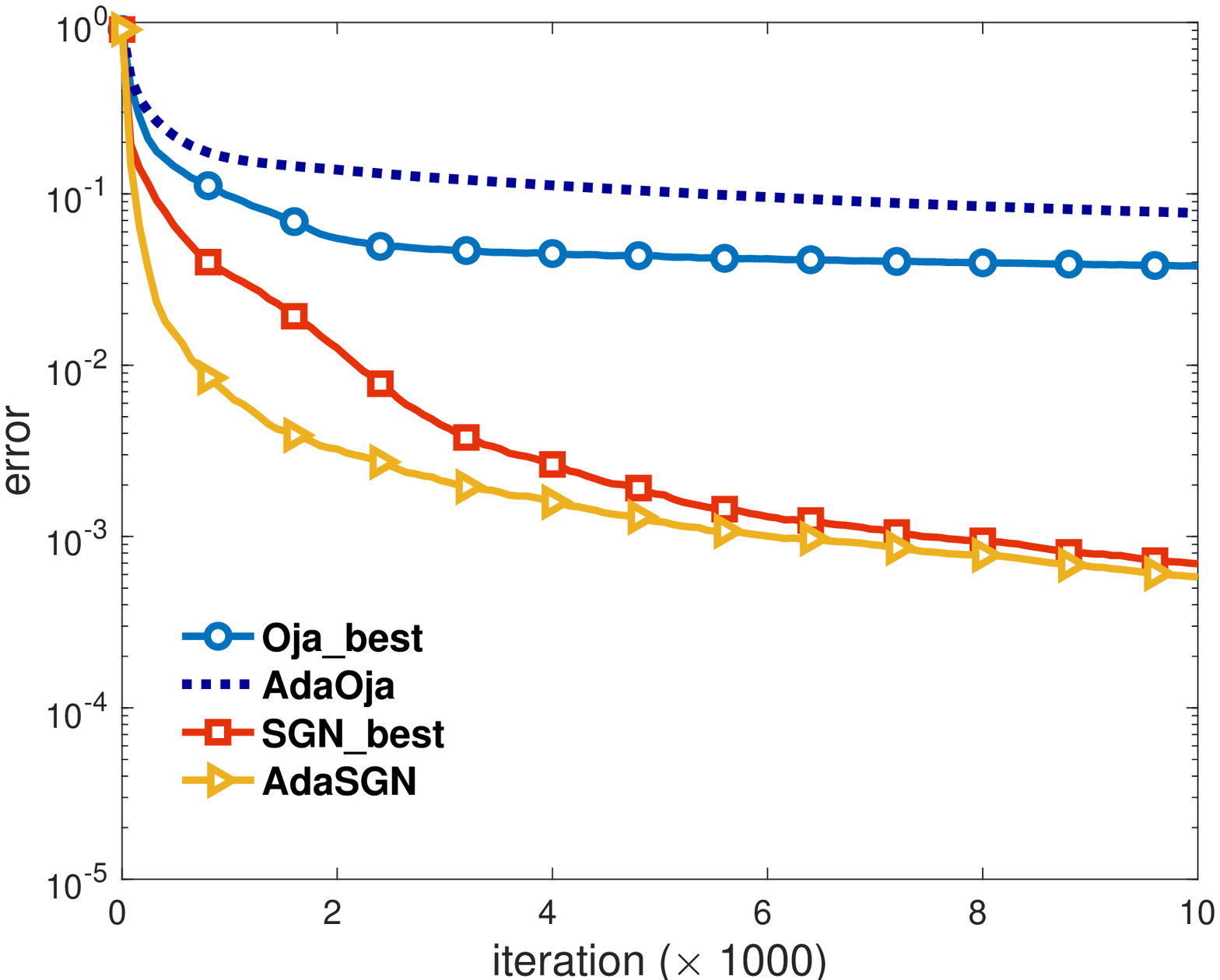}
		\end{minipage}
	}
	\subfigure[$h=10,~p=30$]{
		\begin{minipage}[t]{4.5cm}
			\centering
			\includegraphics[width=4.5cm]{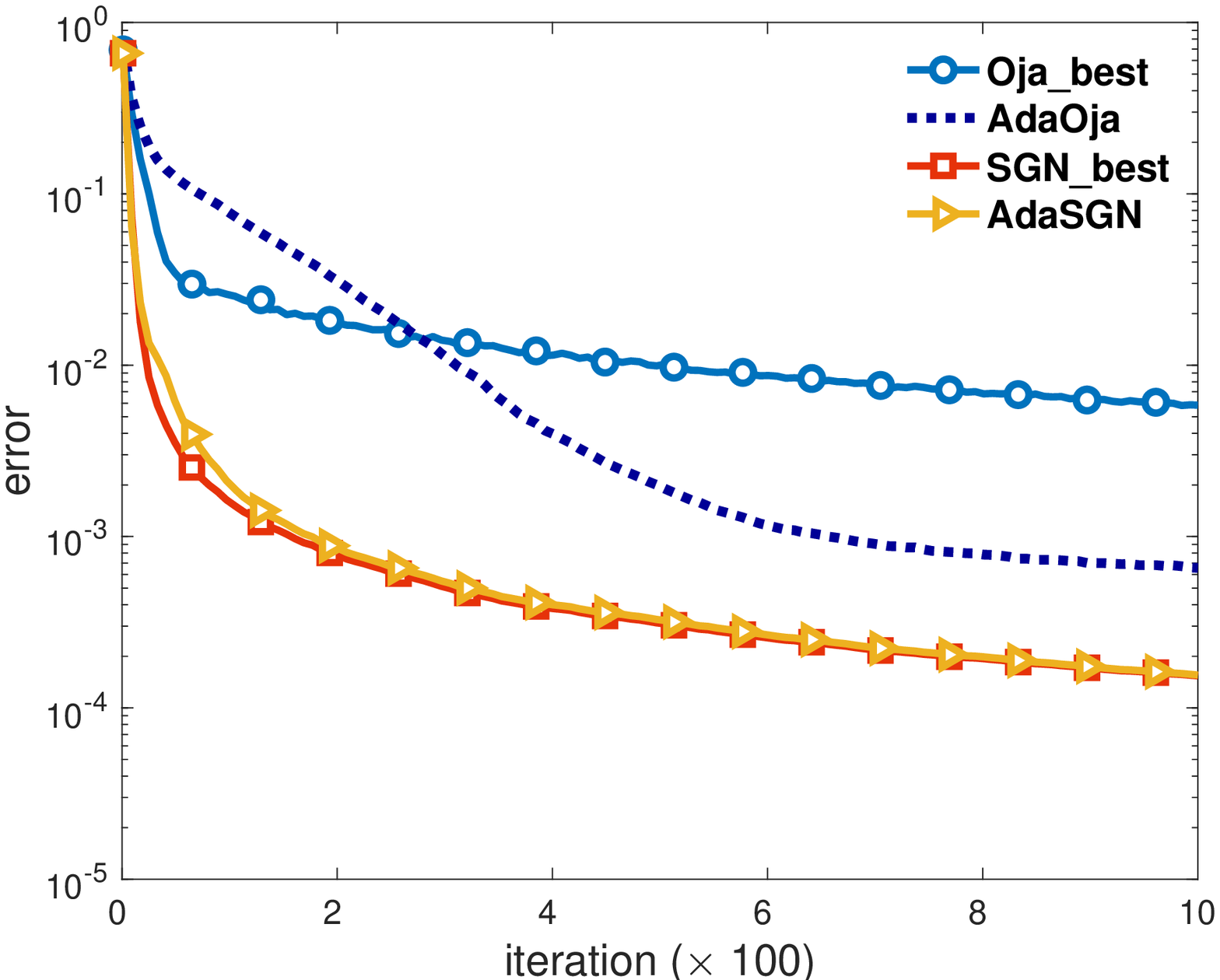}
		\end{minipage}
	}
	\subfigure[$h=100,~p=30$]{
		\begin{minipage}[t]{4.5cm}
			\centering
			\includegraphics[width=4.5cm]{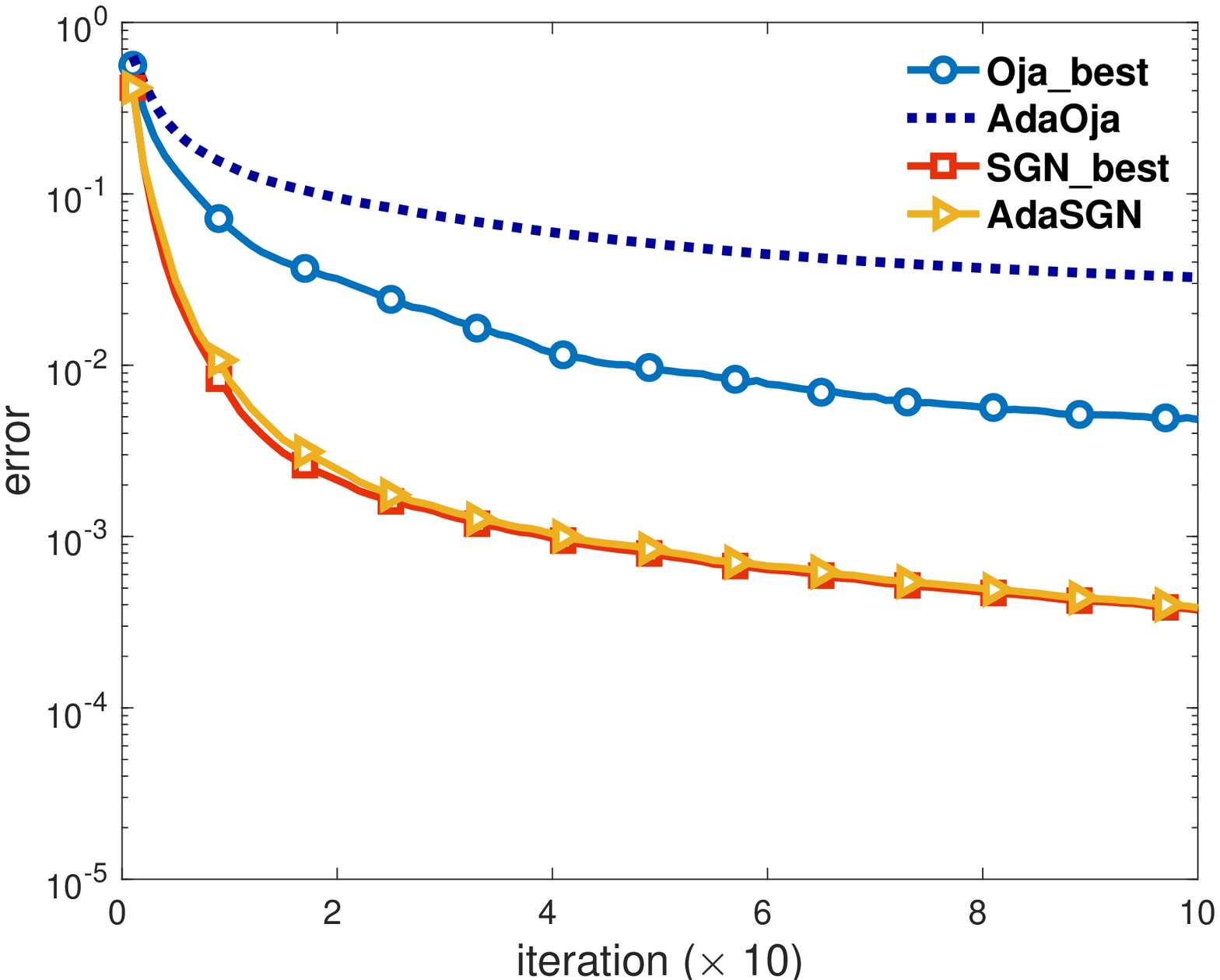}
		\end{minipage}
	}
	\caption{The estimation error of Oja's iteration and
		SGN with best-tuned diminishing stepsizes, AdaOja and AdaSGN on \texttt{Gau-gap-1} with $\bar{\mu}=10$.}
	\label{fig:adaSGN_simu}
\end{figure}

\begin{figure}[h]
	\centering
	\subfigure[$h=1,~p=1$]{
		\begin{minipage}[t]{4.5cm}
			\centering
			\includegraphics[width=4.5cm]{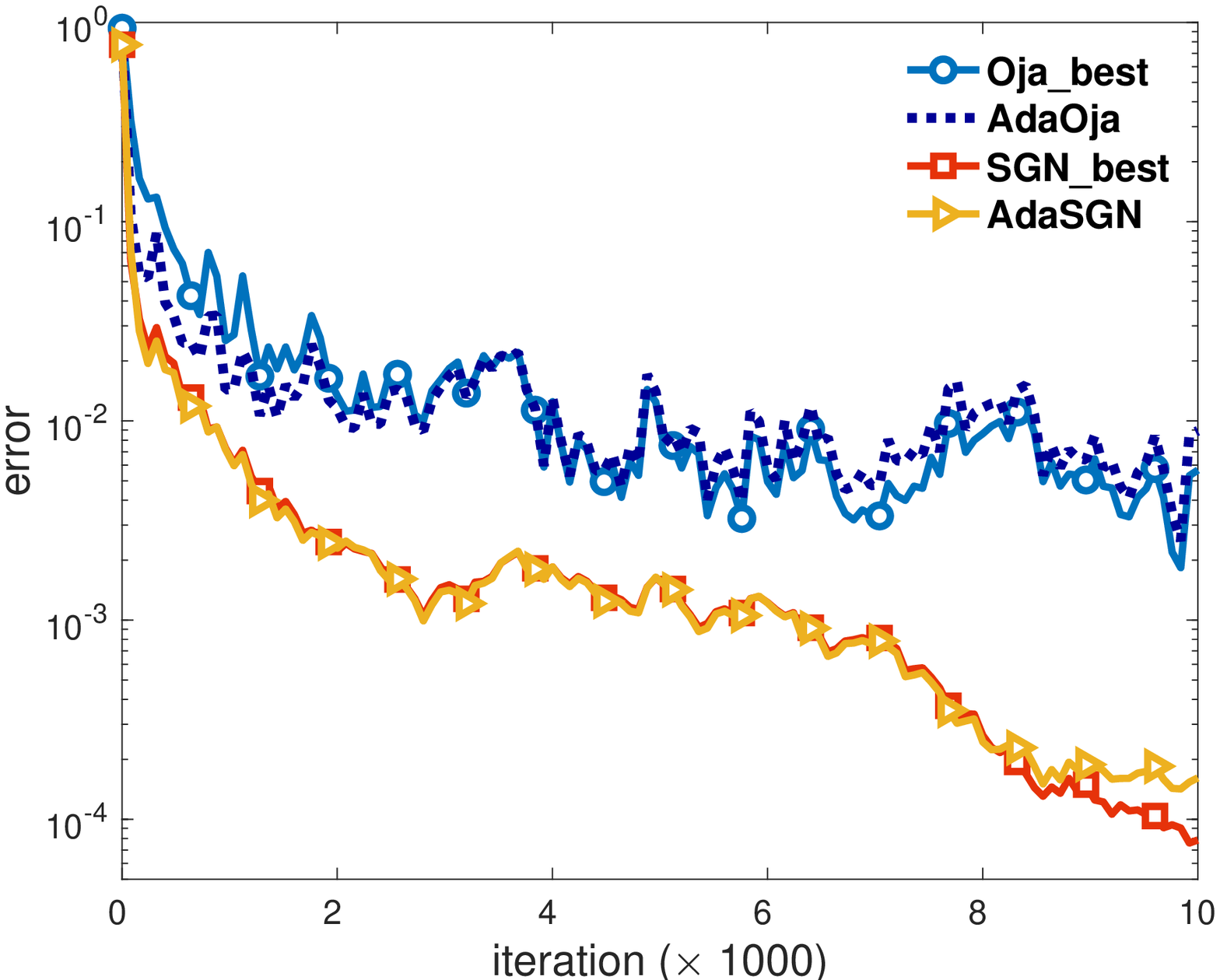}
		\end{minipage}
	}
	\subfigure[$h=10,~p=1$]{
		\begin{minipage}[t]{4.5cm}
			\centering
			\includegraphics[width=4.5cm]{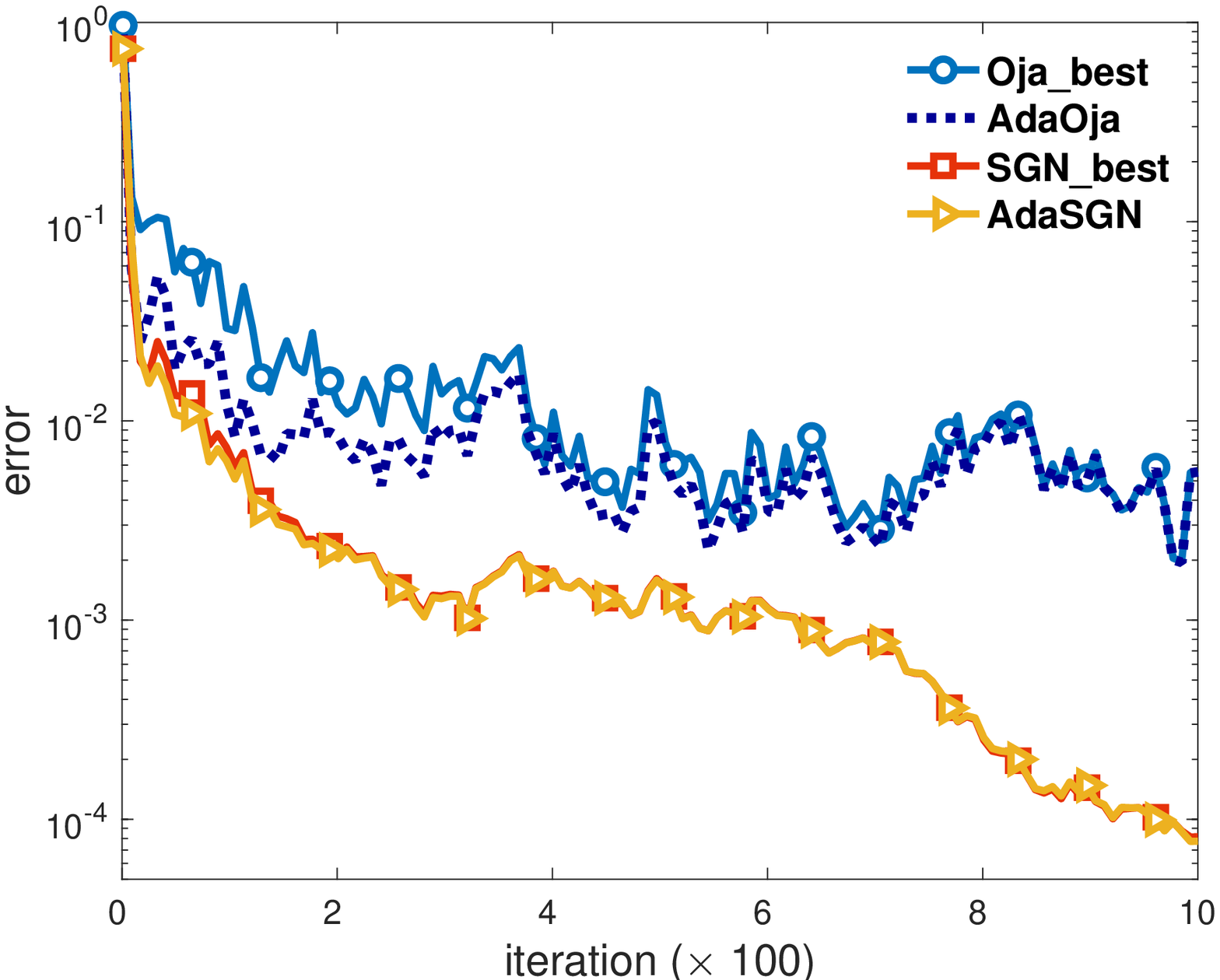}
		\end{minipage}
	}
	\subfigure[$h=100,~p=1$]{
		\begin{minipage}[t]{4.5cm}
			\centering
			\includegraphics[width=4.5cm]{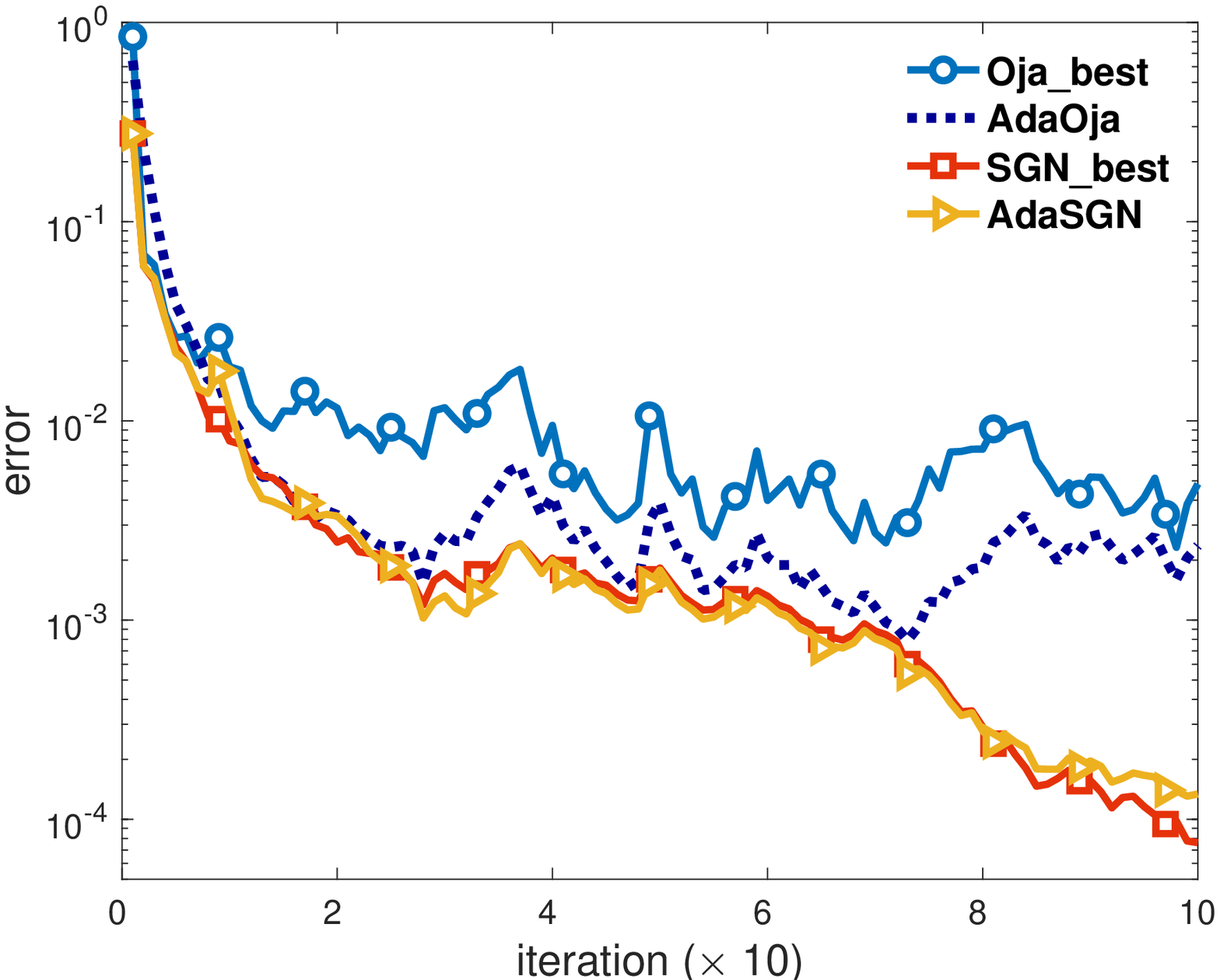}
		\end{minipage}
	}

	\subfigure[$h=1,~p=30$]{
		\begin{minipage}[t]{4.5cm}
			\centering
			\includegraphics[width=4.5cm]{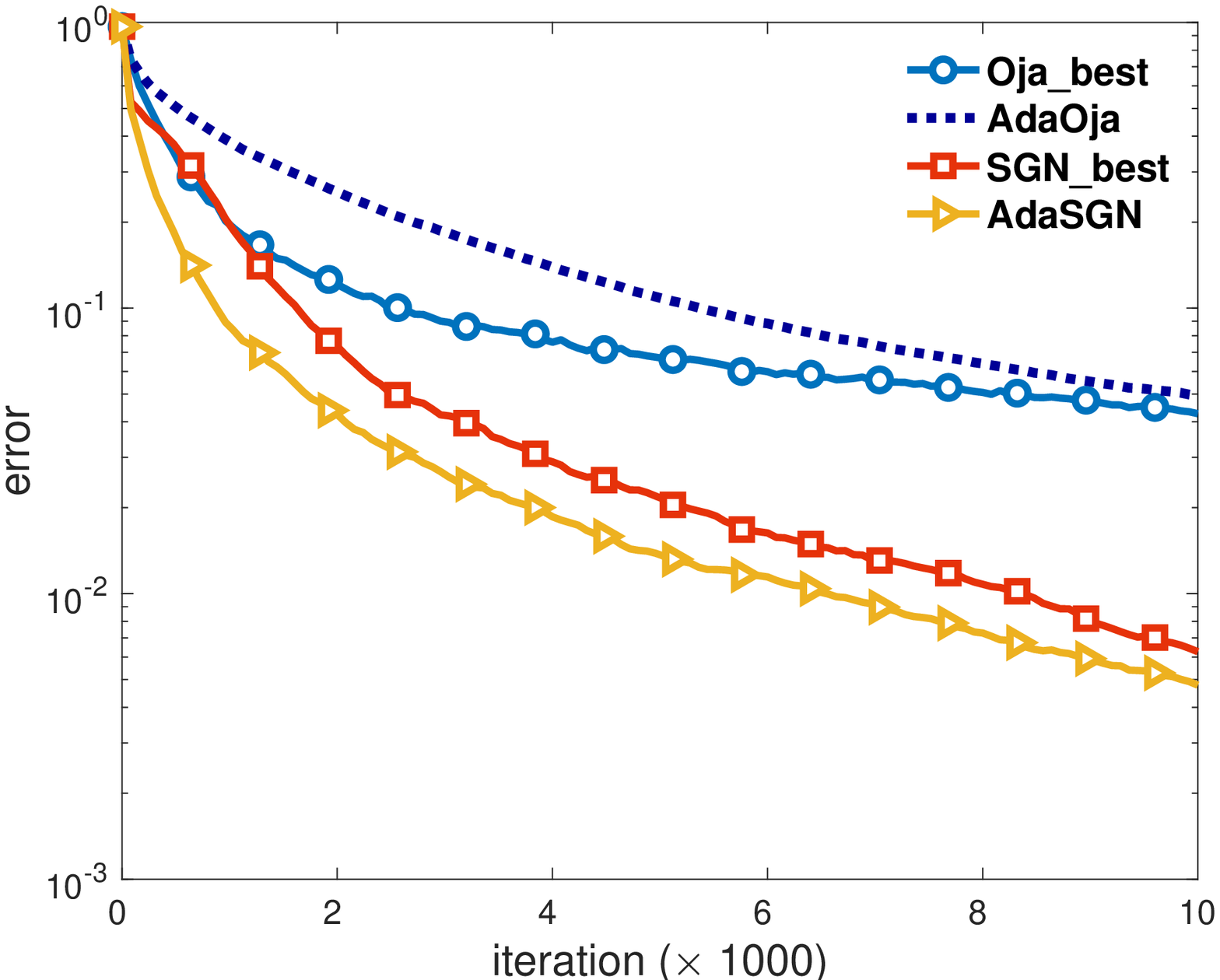}
		\end{minipage}
	}
	\subfigure[$h=10,~p=30$]{
		\begin{minipage}[t]{4.5cm}
			\centering
			\includegraphics[width=4.5cm]{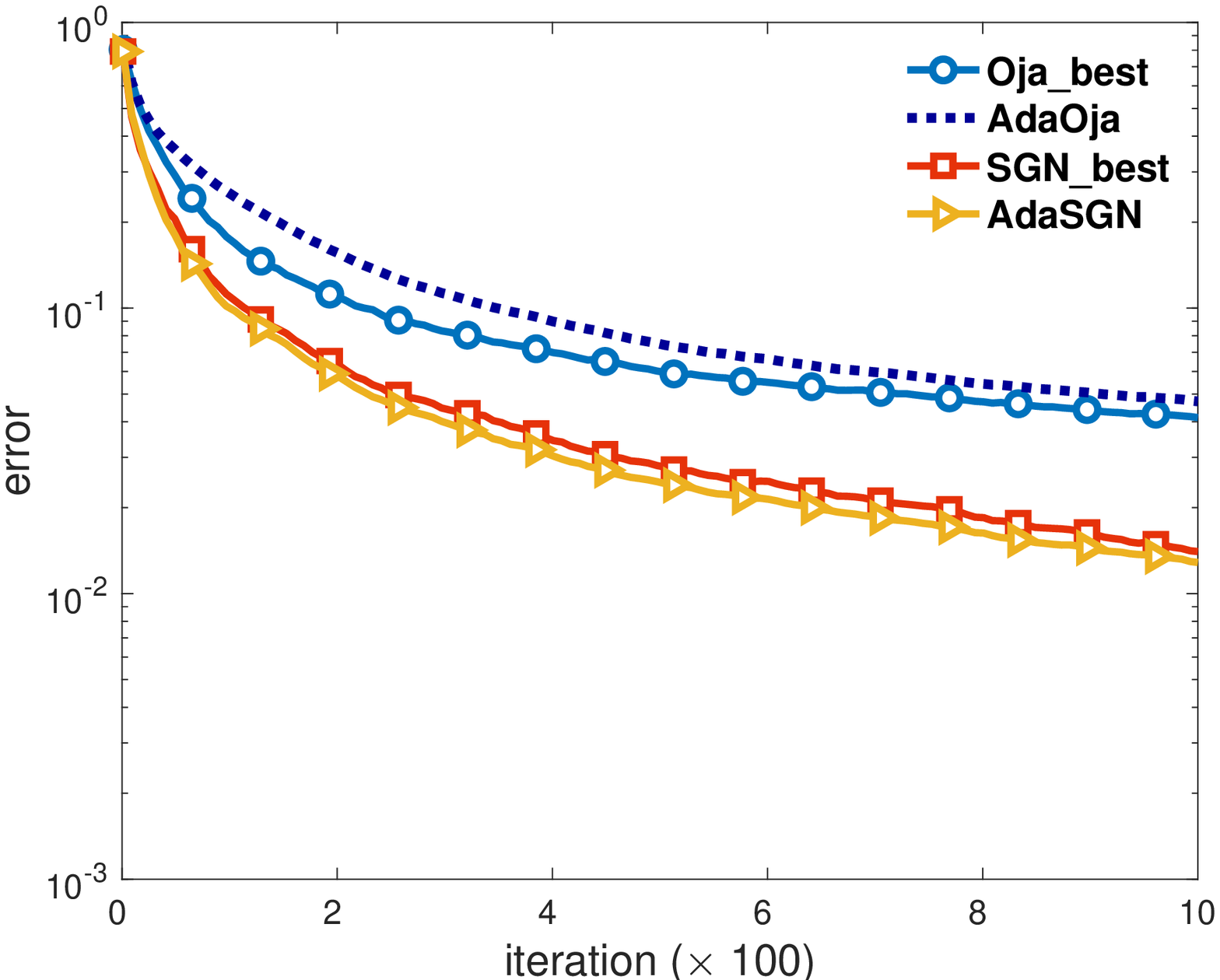}
		\end{minipage}
	}
	\subfigure[$h=100,~p=30$]{
		\begin{minipage}[t]{4.5cm}
			\centering
			\includegraphics[width=4.5cm]{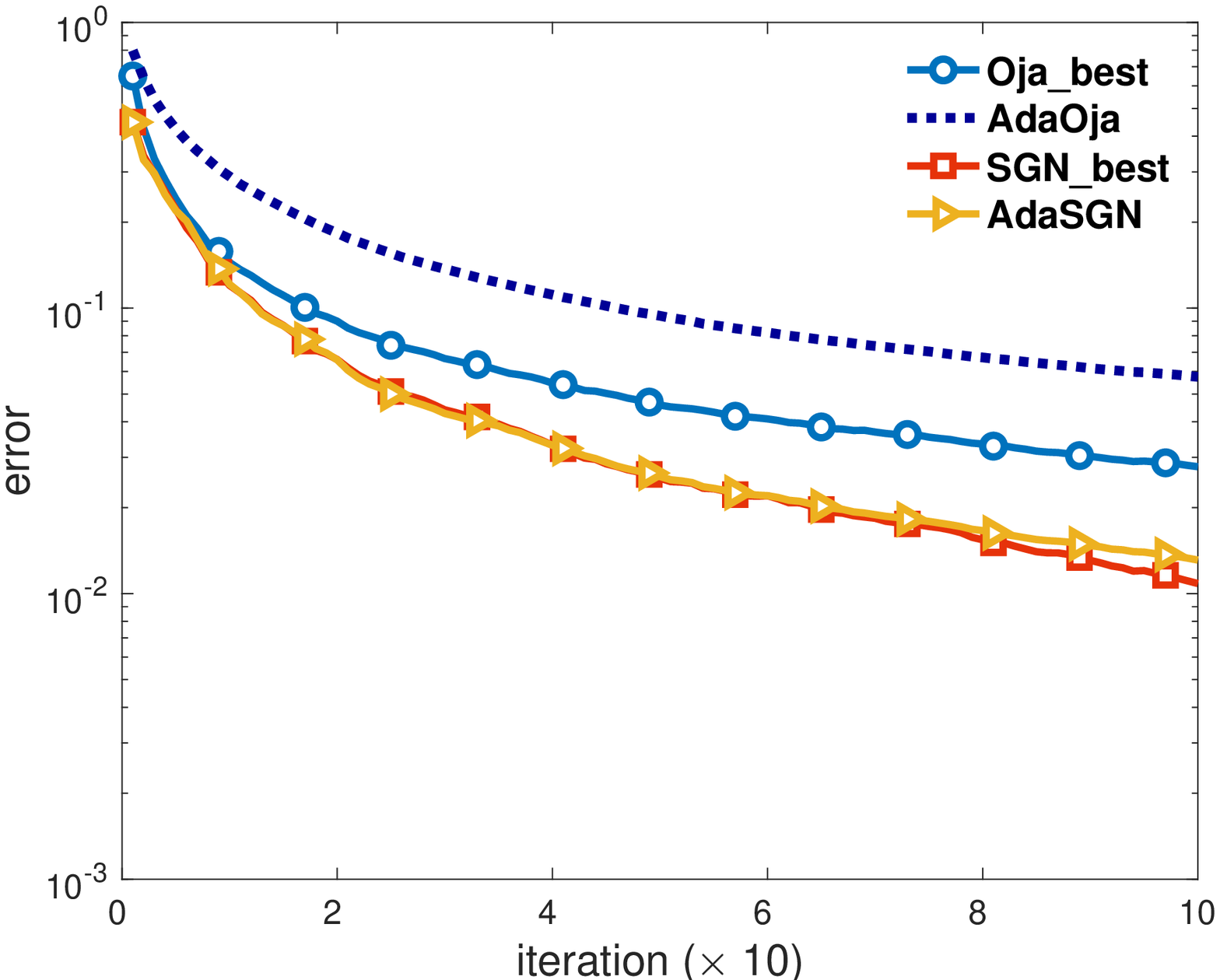}
		\end{minipage}
	}
	\caption{The estimation error of Oja's iteration and
		SGN with best-tuned diminishing stepsizes, AdaOja and AdaSGN on \texttt{CIFAR-$10$}.}
	\label{fig:adaSGN_real}
\end{figure}

\section{Concluding remarks}\label{sec:conclu}

In this paper, 
we develop an SGN method for solving the OPCA problem, which exhibits the empirical robustness with respect to the varying of the stepsize parameter and input data.
By adopting the diffusion approximation, we provide the global convergence of SGN under sub-Gaussian samples without
traditional gap assumption. 
Then, to avoid stepsize tuning which often utilizes prior information on the input data, we focus on the sample consistency and design an adaptive scheme for SGN called AdaSGN, with a basic idea of assigning smaller stepsizes for samples of lower consistency.
Numerical results indicate that AdaSGN shows a better performance than the adaptive version
of Oja method and is comparable with SGN using diminishing stepsizes of manual selection.

Several aspects of this work deserve further studies.
First, our theoretical results are established on the stepsize scheme \cref{eq:dimi} with infinitesimal $\gamma$ and $\beta\in[0, 1)$. 
In spite of the desirable empirical performance of SGN with $\beta\in[0, 1]$, it is worthwhile to extend the analysis techniques to the case of $\beta=1$ so that the gap between the theory and practice could be filled.
Then, how our proposed adaptive strategy is analyzed theoretically and whether it will work with other OPCA algorithms remain to be answered. 
Additionally, developing variants of SGN for different settings (e.g., sparsity on PCs, capability of missing data) is of particular interest as well.


\bibliographystyle{siam}
	\bibliography{ref-sgn}

\begin{thebibliography}{10}

\bibitem{Allen2017First}
{\sc Z.~Allen-Zhu and Y.~Li}, {\em First efficient convergence for streaming
  k-pca: a global, gap-free, and near-optimal rate}, in 2017 IEEE 58th Annual
  Symposium on Foundations of Computer Science (FOCS), IEEE, 2017,
  pp.~487--492.

\bibitem{balcan2016improved}
{\sc M.-F. Balcan, S.~S. Du, Y.~Wang, and A.~W. Yu}, {\em An improved
  gap-dependency analysis of the noisy power method}, in Conference on Learning
  Theory, PMLR, 2016, pp.~284--309.

\bibitem{Balsubramani2013The}
{\sc A.~Balsubramani, S.~Dasgupta, and Y.~Freund}, {\em The fast convergence of
  incremental pca}, in Advances in Neural Information Processing Systems,
  vol.~26, 2013, pp.~3174--3182.

\bibitem{balzano2018streaming}
{\sc L.~Balzano, Y.~Chi, and Y.~M. Lu}, {\em Streaming pca and subspace
  tracking: The missing data case}, Proceedings of the IEEE, 106 (2018),
  pp.~1293--1310.

\bibitem{cardot2018online}
{\sc H.~Cardot and D.~Degras}, {\em Online principal component analysis in high
  dimension: Which algorithm to choose?}, International Statistical Review, 86
  (2018), pp.~29--50.

\bibitem{chen2018dimensionality}
{\sc M.~Chen, L.~F. Yang, M.~Wang, and T.~Zhao}, {\em Dimensionality reduction
  for stationary time series via stochastic nonconvex optimization}, in
  Proceedings of the 32nd International Conference on Neural Information
  Processing Systems, 2018, pp.~3500--3510.

\bibitem{de2015global}
{\sc C.~De~Sa, C.~Re, and K.~Olukotun}, {\em Global convergence of stochastic
  gradient descent for some non-convex matrix problems}, in International
  Conference on Machine Learning, PMLR, 2015, pp.~2332--2341.

\bibitem{ethier1986markov}
{\sc S.~N. Ethier and T.~G. Kurtz}, {\em Markov processes: characterization and
  convergence}, Wiley, 1986.

\bibitem{feng2018semigroups}
{\sc Y.~Feng, L.~Li, and J.-G. Liu}, {\em Semigroups of stochastic gradient
  descent and online principal component analysis: properties and diffusion
  approximations}, Communications in Mathematical Sciences, 16 (2018),
  pp.~777--789.

\bibitem{gao2019parallelizable}
{\sc B.~Gao, X.~Liu, and Y.-x. Yuan}, {\em Parallelizable algorithms for
  optimization problems with orthogonality constraints}, SIAM Journal on
  Scientific Computing, 41 (2019), pp.~A1949--A1983.

\bibitem{Golub1996Matrix}
{\sc G.~H. Golub and C.~F. Van~Loan}, {\em Matrix computations}, vol.~3, JHU
  press, 2013.

\bibitem{hardt2014noisy}
{\sc M.~Hardt and E.~Price}, {\em The noisy power method: a meta algorithm with
  applications}, in Proceedings of the 27th International Conference on Neural
  Information Processing Systems-Volume 2, 2014, pp.~2861--2869.

\bibitem{henriksen2019adaoja}
{\sc A.~Henriksen and R.~Ward}, {\em Adaoja: Adaptive learning rates for
  streaming pca}, arXiv preprint arXiv:1905.12115,  (2019).

\bibitem{huang2021streaming}
{\sc D.~Huang, J.~Niles-Weed, and R.~Ward}, {\em Streaming k-pca: Efficient
  guarantees for oja's algorithm, beyond rank-one updates}, in Conference on
  Learning Theory, PMLR, 2021, pp.~2463--2498.

\bibitem{jain2016streaming}
{\sc P.~Jain, C.~Jin, S.~M. Kakade, P.~Netrapalli, and A.~Sidford}, {\em
  Streaming pca: Matching matrix bernstein and near-optimal finite sample
  guarantees for oja's algorithm}, in Conference on Learning Theory, PMLR,
  2016, pp.~1147--1164.

\bibitem{Johnson2013Accelerating}
{\sc R.~Johnson and T.~Zhang}, {\em Accelerating stochastic gradient descent
  using predictive variance reduction}, in Proceedings of the 26th
  International Conference on Neural Information Processing Systems-Volume 1,
  2013, pp.~315--323.

\bibitem{Jolliffe1986Principal}
{\sc I.~T. Jolliffe}, {\em Principal component analysis}, Springer-Verlag,
  1986.

\bibitem{jolliffe2016principal}
{\sc I.~T. Jolliffe and J.~Cadima}, {\em Principal component analysis: A review
  and recent developments}, Philosophical Transactions of the Royal Society A:
  Mathematical, Physical and Engineering Sciences, 374 (2016), p.~20150202.

\bibitem{knyazev2001toward}
{\sc A.~V. Knyazev}, {\em Toward the optimal preconditioned eigensolver:
  Locally optimal block preconditioned conjugate gradient method}, SIAM Journal
  on Scientific Computing, 23 (2001), pp.~517--541.

\bibitem{krasulina1969method}
{\sc T.~Krasulina}, {\em The method of stochastic approximation for the
  determination of the least eigenvalue of a symmetrical matrix}, USSR
  Computational Mathematics and Mathematical Physics, 9 (1969), pp.~189--195.

\bibitem{krizhevsky2009learning}
{\sc A.~Krizhevsky and G.~Hinton}, {\em Learning multiple layers of features
  from tiny images}, Technical Report, Department of Computer Science,
  University of Toronto,  (2009).

\bibitem{lecun1998gradient}
{\sc Y.~LeCun, L.~Bottou, Y.~Bengio, and P.~Haffner}, {\em Gradient-based
  learning applied to document recognition}, Proceedings of the IEEE, 86
  (1998), pp.~2278--2324.

\bibitem{Li2016Near}
{\sc C.~J. Li, M.~Wang, L.~Han, and Z.~Tong}, {\em Near-optimal stochastic
  approximation for online principal component estimation}, Mathematical
  Programming, 167 (2016), pp.~75--97.

\bibitem{li2017diffusion}
{\sc C.~J. Li, M.~Wang, H.~Liu, and T.~Zhang}, {\em Diffusion approximations
  for online principal component estimation and global convergence}, in
  Advances in Neural Information Processing Systems, vol.~30, 2017.

\bibitem{Li2015Rivalry}
{\sc C.-L. Li, H.-T. Lin, and C.-J. Lu}, {\em Rivalry of two families of
  algorithms for memory-restricted streaming pca}, in Artificial Intelligence
  and Statistics, PMLR, 2016, pp.~473--481.

\bibitem{liu2013limited}
{\sc X.~Liu, Z.~Wen, and Y.~Zhang}, {\em Limited memory block krylov subspace
  optimization for computing dominant singular value decompositions}, SIAM
  Journal on Scientific Computing, 35 (2013), pp.~A1641--A1668.

\bibitem{liu2015efficient}
\leavevmode\vrule height 2pt depth -1.6pt width 23pt, {\em An efficient
  gauss--newton algorithm for symmetric low-rank product matrix
  approximations}, SIAM Journal on Optimization, 25 (2015), pp.~1571--1608.

\bibitem{lv2006global}
{\sc J.~C. Lv, Z.~Yi, and K.~K. Tan}, {\em Global convergence of oja's pca
  learning algorithm with a non-zero-approaching adaptive learning rate},
  Theoretical Computer Science, 367 (2006), pp.~286--307.

\bibitem{Oja1982Simplified}
{\sc E.~Oja}, {\em Simplified neuron model as a principal component analyzer},
  Journal of Mathematical Biology, 15 (1982), pp.~267--273.

\bibitem{pearson1901liii}
{\sc K.~Pearson}, {\em Liii. on lines and planes of closest fit to systems of
  points in space}, The London, Edinburgh, and Dublin Philosophical Magazine
  and Journal of Science, 2 (1901), pp.~559--572.

\bibitem{raychaudhuri1999principal}
{\sc S.~Raychaudhuri, J.~M. Stuart, and R.~B. Altman}, {\em Principal
  components analysis to summarize microarray experiments: Application to
  sporulation time series}, in Biocomputing 2000, World Scientific, 1999,
  pp.~455--466.

\bibitem{Robbins1951A}
{\sc H.~Robbins and S.~Monro}, {\em A stochastic approximation method}, Annals
  of Mathematical Statistics, 22 (1951), pp.~400--407.

\bibitem{rutishauser1970simultaneous}
{\sc H.~Rutishauser}, {\em Simultaneous iteration method for symmetric
  matrices}, Numerische Mathematik, 16 (1970), pp.~205--223.

\bibitem{saad1980rates}
{\sc Y.~Saad}, {\em On the rates of convergence of the lanczos and the
  block-lanczos methods}, SIAM Journal on Numerical Analysis, 17 (1980),
  pp.~687--706.

\bibitem{shamir2015stochastic}
{\sc O.~Shamir}, {\em A stochastic pca and svd algorithm with an exponential
  convergence rate}, in International Conference on Machine Learning, PMLR,
  2015, pp.~144--152.

\bibitem{shamir2016convergence}
\leavevmode\vrule height 2pt depth -1.6pt width 23pt, {\em Convergence of
  stochastic gradient descent for pca}, in International Conference on Machine
  Learning, PMLR, 2016, pp.~257--265.

\bibitem{shamir2016fast}
\leavevmode\vrule height 2pt depth -1.6pt width 23pt, {\em Fast stochastic
  algorithms for svd and pca: Convergence properties and convexity}, in
  International Conference on Machine Learning, PMLR, 2016, pp.~248--256.

\bibitem{sleijpen2000jacobi}
{\sc G.~L. Sleijpen and H.~A. Van~der Vorst}, {\em A jacobi--davidson iteration
  method for linear eigenvalue problems}, SIAM Review, 42 (2000), pp.~267--293.

\bibitem{sorensen1997implicitly}
{\sc D.~C. Sorensen}, {\em Implicitly restarted arnoldi/lanczos methods for
  large scale eigenvalue calculations}, in Parallel Numerical Algorithms,
  Springer, 1997, pp.~119--165.

\bibitem{stathopoulos1994davidson}
{\sc A.~Stathopoulos and C.~F. Fischer}, {\em A davidson program for finding a
  few selected extreme eigenpairs of a large, sparse, real, symmetric matrix},
  Computer Physics Communications, 79 (1994), pp.~268--290.

\bibitem{stewart1990matrix}
{\sc G.~Stewart and J.~Sun}, {\em Matrix perturbation theory}, Academic Press,
  1990.

\bibitem{subasi2010eeg}
{\sc A.~Subasi and M.~I. Gursoy}, {\em Eeg signal classification using pca,
  ica, lda and support vector machines}, Expert Systems with Applications, 37
  (2010), pp.~8659--8666.

\bibitem{tang2019exponentially}
{\sc C.~Tang}, {\em Exponentially convergent stochastic k-pca without variance
  reduction}, in Advances in Neural Information Processing Systems, 2019,
  pp.~12393--12404.

\bibitem{tropp2015introduction}
{\sc J.~A. Tropp}, {\em An introduction to matrix concentration inequalities},
  Foundations and Trends in Machine Learning, 8 (2015), pp.~1--230.

\bibitem{turk1991eigenfaces}
{\sc M.~Turk and A.~Pentland}, {\em Eigenfaces for recognition}, Journal of
  Cognitive Neuroscience, 3 (1991), pp.~71--86.

\bibitem{uhlenbeck1930theory}
{\sc G.~E. Uhlenbeck and L.~S. Ornstein}, {\em On the theory of the brownian
  motion}, Physical Review, 36 (1930), p.~823.

\bibitem{vu2013minimax}
{\sc V.~Q. Vu and J.~Lei}, {\em Minimax sparse principal subspace estimation in
  high dimensions}, Annals of Statistics, 41 (2013), pp.~2905--2947.

\bibitem{wang2016online}
{\sc C.~Wang and Y.~M. Lu}, {\em Online learning for sparse pca in high
  dimensions: Exact dynamics and phase transitions}, in 2016 IEEE Information
  Theory Workshop (ITW), IEEE, 2016, pp.~186--190.

\bibitem{wang2021multipliers}
{\sc L.~Wang, B.~Gao, and X.~Liu}, {\em Multipliers correction methods for
  optimization problems over the stiefel manifold}, CSIAM Transactions on
  Applied Mathematics, 2 (2021), pp.~508--531.

\bibitem{ward2019adagrad}
{\sc R.~Ward, X.~Wu, and L.~Bottou}, {\em Adagrad stepsizes: Sharp convergence
  over nonconvex landscapes}, in International Conference on Machine Learning,
  PMLR, 2019, pp.~6677--6686.

\bibitem{wen2016trace}
{\sc Z.~Wen, C.~Yang, X.~Liu, and Y.~Zhang}, {\em Trace-penalty minimization
  for large-scale eigenspace computation}, Journal of Scientific Computing, 66
  (2016), pp.~1175--1203.

\bibitem{wen2017accelerating}
{\sc Z.~Wen and Y.~Zhang}, {\em Accelerating convergence by augmented
  rayleigh--ritz projections for large-scale eigenpair computation}, SIAM
  Journal on Matrix Analysis and Applications, 38 (2017), pp.~273--296.

\bibitem{xiao2017fashion}
{\sc H.~Xiao, K.~Rasul, and R.~Vollgraf}, {\em Fashion-mnist: a novel image
  dataset for benchmarking machine learning algorithms}, arXiv preprint
  arXiv:1708.07747,  (2017).

\bibitem{xiao2020class}
{\sc N.~Xiao, X.~Liu, and Y.-x. Yuan}, {\em A class of smooth exact penalty
  function methods for optimization problems with orthogonality constraints},
  Optimization Methods and Software,  (2020), pp.~1--37.

\bibitem{zhou2020convergence}
{\sc S.~Zhou and Y.~Bai}, {\em Convergence analysis of oja’s iteration for
  solving online pca with nonzero-mean samples}, Science China Mathematics, 64
  (2021), pp.~849--868.

\end{thebibliography}

\addcontentsline{toc}{section}{References}

	\appendix
\section{Boundedness of solution to ODE \cref{eq:ode}}\label{sec:bound}
\begin{lemma}\label{lemma:bound}
	The solution $X(t)$ to ODE \cref{eq:ode} is uniformly bounded.
\end{lemma}

\begin{proof}[Proof]
	Denote $L=\tr(X^\top X)$ and $M=\tr(X^\top\Sigma X(X^\top X)^{-1})$.
	\begin{align}
	\notag\frac{dL}{dt}&=2\tr\left(X^\top\frac{dX}{dt}\right)\\
	\label{eq:dL}&=\tr\left(2X^\top \Sigma X(X^\top X)^{-1}-X^\top X-X^\top \Sigma X(X^\top X)^{-1}\right)=M-L.
	\end{align}
	The explicit solution of \cref{eq:dL} is given by
	\begin{align*}
	L(t)=\exp\{-t\}\left(\int M\exp\{t\}ds+C(L^0)\right),
	\end{align*}
	where $C(L^0)$ is some constant depending only on the initial value $L(0)=L^0=\tr(X^{0\top}X^{0})$.
	Since
	\begin{align*}
	\lambda_{n}I_{p}\preccurlyeq (X^\top X)^{-\frac{1}{2}}X^\top\Sigma X(X^\top X)^{-\frac{1}{2}} \preccurlyeq\lambda_{1}I_{p},
	\end{align*}
	we have
	\begin{align*}
	p\lambda_{n}\leq M\leq p\lambda_{1},
	\end{align*}
	which then yields the bound
	\begin{align*}
	p\lambda_{n}+C(L^0)\exp\{-t\}\leq L\leq p\lambda_{1}+C(L^0)\exp\{-t\}.
	\end{align*}
\end{proof}

\section{Proof of \cref{lemma:kT1}}\label{sec:proofkT1}
We  begin with the full-rankness of SGN.
\begin{lemma}\label{lemma:fullrank}
	Suppose that the assumptions [A1]-[A3] hold. Then, the sequence $\{X^{(k)}\}$ generated by \cref{alg:sgn} with the stepsize $\alpha^{(k)}\leq 1$ satisfies
	\begin{align*}
	\sigma_{min}(X^{(k+1)})\geq\frac{1}{2}\sigma_{min}(X^{(k)}),
	\end{align*}
	which indicates that $\sigma_{min}(X^{(k+1)})>0$ whenever $\sigma_{min}(X^{(k)})>0$.
\end{lemma}
\begin{proof}
	To simplify the notation, we denote 
	\begin{align}
	\label{eq:nota}	X^{+} = X^{(k+1)},\quad X=X^{(k)},\quad  \alpha=\alpha^{(k)},\quad  \Sigma_{h}=\Sigma_{h}^{(k)},\quad S = S^{(k)}.
	\end{align}
	Recall that $P=X(X^\top X)^{-1}$. It is obvious that
	$$
	X^{+\top}\left(\lambda_{\max}(PP^\top)I_n-PP^\top\right)X^{+}\succeq 0.$$
	Thus, we have
	$$
	\sigma^{2}_{\min}(X^{+})\sigma_{\max}^{2}(P)\geq \sigma_{\min}^{2}(P^\top X^{+}),
	$$
	where
	\begin{equation*}
	P^\top X^{+}=I_p+\frac{\alpha}{2}\left((X^\top X)^{-1}X^\top\Sigma_h X(X^\top X)^{-1}-I_p\right).
	\end{equation*}
	Then, we obtain
	\begin{align*}
	\sigma_{\min}(P^\top X^{+})&=\lambda_{\min}(P^\top X^{+})\\
	&=1-\frac{\alpha}{2}\left(1-\lambda_{\min}\left((X^\top X)^{-1}X^\top \Sigma_{h}X(X^\top X)^{-1}\right)\right)\\
	&\geq 1-\frac{\alpha}{2}.
	\end{align*}	
	Since $\sigma_{\max}(P)=\sigma_{\min}^{-1}(X)$, we have
	\begin{align*}
	\sigma_{\min}(X^{+})\geq \sigma_{\min}(X)\sigma_{\min}(P^\top X^{+})\geq\frac{1}{2}\sigma_{\min}(X)>0.
	\end{align*}
\end{proof}

Though the full-rankness of SGN has been proved in \cref{lemma:fullrank}, the possibility of $\lim\limits_{k\to\infty}\sigma_{\min}(X^{(k)})=0$ can not be ruled out.
As a theoretical supplement, we set a threshold $\underline{\sigma}>0$, and take an additional correction step $X^{(k)}+S^{(k)}_c$ when $\sigma_{\min}(X^{(k)})\leq\underline{\sigma}$.
The correction direction $S^{(k)}_c$ is chosen to satisfy
\begin{align}
\label{eq:G1} &\sigma_{\min}(X^{(k)}+S^{(k)}_c)>\underline{\sigma},\\
\label{eq:G2}
&f_{\text{E}}(X^{(k)}+S^{(k)}_c)\leq f_{\text{E}}(X^{(k)}),\\
\label{eq:G3}
&\Vert\sin\Theta(X^{(k)}+S^{(k)}_c, U_{p'})\Vert_{\text{F}}^{2}\leq\Vert\sin\Theta(X^{(k)}, U_{p'})\Vert_{\text{F}}^{2}.
\end{align}
This will be detailed in \cref{lemma:correction}.

We perform an SVD of the iterate $X^{(k)}\in\mathbb{R}^{n\times p}$ as
\begin{equation}\label{eq:xsvd}
X^{(k)} = U^{(k)}\Sigma^{(k)} V^{(k)\top} = U_{1}^{(k)}\Sigma_1^{(k)}V^{(k)\top},
\end{equation}
where $U^{(k)}\in\mathbb{R}^{n\times n}$ is an orthogonal matrix, $U_1^{(k)}=U^{(k)}_{(:, 1:p)}\in\mathbb{R}^{n\times p}$, $\Sigma^{(k)}\in\mathbb{R}^{n\times p}$ and $\Sigma^{(k)}_1\in\mathbb{R}^{p\times p}$ have the singular values $\sigma_{1}^{(k)}\geq\cdots\geq\sigma_p^{(k)}>0$ of $X^{(k)}$ on their diagonals and zeros elsewhere, and 
$V^{(k)}\in\mathbb{R}^{p\times p}$ is also an orthogonal matrix. 
Then, we have the following result.

\begin{lemma}\label{lemma:correction} 
	For some iterate $X^{(k)}\in\mathbb{R}^{n\times p}$ with the SVD form \cref{eq:xsvd}.
	Given a threshold 
	\begin{equation}
	\label{eq:threshold}
	\underline{\sigma}\leq \frac{4}{35}\frac{\lambda_n}{\lambda_1}\sqrt{\lambda_n}.
	\end{equation}
	Suppose that there are $p_{\underline{\sigma}}$ singular values no greater than $\underline{\sigma}$. Define a correction direction as
	\begin{equation}\label{eq:correction}
	S^{(k)}_c=\theta Q^{(k)} V_{p_{\underline{\sigma}}}^{(k)\top},
	\end{equation}
	where $\theta=\sqrt{\lambda_n}$, $Q^{(k)}\in\mathbb{R}^{n\times p}$ has orthonormal columns, and $V_{p_{\underline{\sigma}}}^{(k)}\in\mathbb{R}^{p\times p_{\underline{\sigma}}}$ is formed by the last $p_{\underline{\sigma}}$ columns of $V^{(k)}$. 
	Then, the inequalities \cref{eq:G1} and \cref{eq:G2} hold if
	$Q^{(k)}\in\textbf{span}\{U_{(:, (p+1):n)}^{(k)}\}$; and \cref{eq:G1}, \cref{eq:G2} and \cref{eq:G3} all hold if $Q^{(k)}$ is taken as $Q^{(k)}=U_{(:, (p-p_{\underline{\sigma}}+1):p)}^{(k)}$.
\end{lemma}

\begin{proof}
	Here and below in this proof, we drop the superscript ``$(k)$" to simplify the notations, that is 
	\begin{equation*}
	X=X^{(k)},\ S_{c}=S_{c}^{(k)},\  U_1=U_1^{(k)},\  \Sigma_1=\Sigma_{1}^{(k)},\quad V=V^{(k)},\  V_{p_{\underline{\sigma}}}=V_{p_{\underline{\sigma}}}^{(k)},\ Q=Q^{(k)},
	\end{equation*}
	and $\sigma_i = \sigma_i^{(k)}\ (i=1,\cdots, k)$.

	By \cref{eq:xsvd} and the definition \cref{eq:correction} of $S_c$, we have
	\begin{align}
	\notag&X^\top X+X^\top S_{c}+S_{c}^\top X+ S_{c}^\top S_{c}\\
	\label{eq:singular_corr}=&V\Sigma_1^2 V^\top+\theta V\Sigma_1U_1^\top QV_{p_{\underline{\sigma}}}^\top+\theta V_{p_{\underline{\sigma}}}Q^\top U_1\Sigma_1V^\top+\theta^2V_{p_{\underline{\sigma}}}Q^\top QV_{p_{\underline{\sigma}}}^\top.
	\end{align}
	If		
	$Q\in\textbf{span}\{U_{(:, (p+1):n)}\}$, we have
	\begin{equation*}
	\cref{eq:singular_corr}=V\textbf{Diag}\left\{\sigma_1^2,\cdots,\sigma_{p-p_{\underline{\sigma}}}^2, (\sigma_{p-p_{\underline{\sigma}}+1}^2\!+\!\theta\sigma_{p-p_{\underline{\sigma}}+1}+\theta^2), \cdots, (\sigma_{p}^2+\theta\sigma_{p}\!+\!\theta^2)\right\}V^\top,
	\end{equation*}
	and if $Q=U_{(:, (p-p_{\underline{\sigma}}+1):p)}$, we have
	\begin{equation*}
	\cref{eq:singular_corr}=V\textbf{Diag}\left\{\sigma_1^2,\cdots, \sigma_{p-p_{\underline{\sigma}}}^2, (\sigma_{p-p_{\underline{\sigma}}+1}^2+\theta^2), \cdots, (\sigma_{p}^2+\theta^2)\right\}V^\top.
	\end{equation*}
	The first inequality \cref{eq:G1} is then obviously true from
	\begin{align*}
	\sigma_{\min}^2(X+S_{c})=&\lambda_{\min}((X+S_{c})^\top(X+S_{c}))\\
	=&\begin{cases}
	\min\left\{\sigma_{p-p_{\underline{\sigma}}}^2,  (\sigma_{p}^2+\theta^2+\theta\sigma_{p})\right\}, & Q\in\textbf{span}\{U_{(:, (p+1):n)}\}, \\ 
	\min\left\{\sigma_{p-p_{\underline{\sigma}}}^2, (\sigma_{p}^2+\theta^2)\right\}, & Q=U_{(:, (p-p_{\underline{\sigma}}+1):p)}.\end{cases}
	\end{align*}
	
	To check the second inequality \cref{eq:G2}, we compute
	\begin{align*}
	&f_{\text{E}}(X+S_c)-f_{\text{E}}(X)\\
	=\ &2\theta^2(\sigma_{p-p_{\underline{\sigma}}+1}^2+\cdots+\sigma_p^2)+p_{\underline{\sigma}}\theta^4-4\tr\left(\Sigma XS_c^\top\right)-2\theta^2\tr\left(Q^\top \Sigma Q\right)\\
	&+
	\begin{cases}
	\sum\limits_{i=p-p_{\underline{\sigma}}+1}^{p}4\theta\sigma_i^3+\sum\limits_{i=p-p_{\underline{\sigma}}+1}^{p}4\theta^2\sigma_i^2+\sum\limits_{i=p-p_{\underline{\sigma}}+1}^{p}4\theta^3\sigma_i, & Q\in\textbf{span}\{U_{(:, (p+1):n)}\}, \\ 
	0, & Q=U_{(:, (p-p_{\underline{\sigma}}+1):p)},\end{cases}\\
	\leq& \ \begin{cases}
	p_{\underline{\sigma}}\theta\left(4\underline{\sigma}^3+6\theta^2\underline{\sigma}^2+(4\theta^2+4\lambda_1)\underline{\sigma}+\theta^3-2\theta \lambda_n\right), & Q\in\textbf{span}\{U_{(:, (p+1):n)}\}, \\ 
	p_{\underline{\sigma}}\theta(
	2\theta\underline{\sigma}^2+4\lambda_1\underline{\sigma}+\theta^3-2\theta\lambda_n), & Q=U_{(:, (p-p_{\underline{\sigma}}+1):p)}.\end{cases}
	\end{align*}
	Then, setting $\theta=\sqrt{\lambda_n}$ and using \cref{eq:threshold} immediately gives the desired inequality \cref{eq:G2}.

	For the third inequality (\ref{eq:G3}), we have
	\begin{align*}
	&\Vert \sin^2 (X+S_c, U_{p'})\Vert_{\text{F}}^2-\Vert \sin^2 (X, U_{p'})\Vert_{\text{F}}^2\\
	=&\ \theta^2\left(\frac{\bar{u}_{p-p_{\underline{\sigma}}+1}^\top\bar{u}_{p-p_{\underline{\sigma}}+1}}{\sigma^2_{p-p_{\underline{\sigma}}+1}+\theta^2}+\cdots+\frac{\bar{u}_{p}^\top\bar{u}_{p}}{\sigma^2_{p}+\theta^2}\right)-\theta^2\left(\frac{\bar{q}_{1}^\top\bar{q}_{1}}{\sigma^2_{p-p_{\underline{\sigma}}+1}+\theta^2}+\cdots+\frac{\bar{q}_{p_{\underline{\sigma}}}^\top\bar{q}_{p_{\underline{\sigma}}}}{\sigma^2_{p}+\theta^2}\right)\\
	&-2\theta\left(\frac{\bar{u}_{p-p_{\underline{\sigma}}+1}^\top\bar{q}_{1}\sigma_{p-p_{\underline{\sigma}}+1}}{\sigma^2_{p-p_{\underline{\sigma}}+1}+\theta^2}+\cdots+\frac{\bar{u}_{p}^\top\bar{q}_{p_{\underline{\sigma}}}\sigma_p}{\sigma^2_{p}+\theta^2}\right),
	\end{align*}
	where $q_{i}\in\mathbb{R}^{n}$ is the $i$-th column of $Q$.
	When $Q=U_{(:, (p-p_{\underline{\sigma}}+1):p)}$, we have $\bar{q}_1=\bar{u}_{p-p_{\underline{\sigma}}+1},\cdots, \bar{q}_{p_{\underline{\sigma}}}=\bar{u}_{p}$, and thus  $\Vert \sin^2 (X+S_c, U_{p'})\Vert_{\text{F}}^2-\Vert \sin^2 (X, U_{p'})\Vert_{\text{F}}^2<0$.
	But when $Q\in\textbf{span}\{U_{(:, (p+1):n)}\}$, it is easy to find a counterexample for \cref{eq:G3}, e.g., the case of $\bar{U}_{1}=I_p$.	
\end{proof}

\begin{remark}\label{remark:correction}
	(1) For the requirements on the correction step, \cref{eq:G1} is natural. And \cref{eq:G2} and \cref{eq:G3} are just different measures for the convergence analysis, where the former is the one adopted in \cite{liu2015efficient}, and the latter is the one that our work focuses on;
	(2) The choice of $Q\in\textbf{span}\{U_{(:, (p+1):n)}\}$ coincides with the correction step taken in \cite{liu2015efficient}, which only suffices to obtain \cref{eq:G1} and \cref{eq:G2}. And however, this is also acceptable in our work as the estimation error of SGN iterates does not decrease monotonically.
\end{remark}

Next, we give the conditional expected boundedness of the increment of SGN iterates using sufficiently small stepsize.
\begin{lemma}\label{lemma:bound2}
	Suppose the assumptions [A1]-[A3] hold and stepsize $\alpha^{(k)}$ is chosen by
	\begin{equation}
	\label{eq:stepsize}
	\alpha^{(k)}\leq\min\left\{ \frac{\underline{\sigma}^2\varphi_{2}}{2\varphi_{4}},~1\right\},
	\end{equation}
	where $\underline{\sigma}$ is the threshold given in \cref{eq:threshold},
	$\varphi_2 = \mathbb{E}\Vert \Sigma_{h}^{(k)}\Vert_2$ and $\varphi_4 = \mathbb{E}\Vert \Sigma_{h}^{(k)}\Vert^2_2$.
	Then $\{X^{(k)}\}$ generated by \cref{alg:sgn} satisfies
	\begin{align*}
	\mathbb{E}\left[\left\Vert S^{(k)}(X^{(k)})\right\Vert_{\text{\rm F}}^2\Big\vert X^{(k)}=X\right]<p\underline{\sigma}^{-2}\varphi_4+\frac{p}{4}\max\{1,\ 2\varphi_2\}.
	\end{align*}
\end{lemma}
\begin{proof}
	For simplicity, we use the same notations as \cref{eq:nota}.
	From the update rule of SGN, it is straightforward to get
	\begin{align}
	\notag	\tr\left(X^{+\top}X^{+}\right)&=\left(\frac{1}{4}\alpha^2-\alpha+1\right)\tr(X^\top X)+\left(-\frac{1}{2}\alpha^2+\alpha\right)\tr\left(X^\top \Sigma_{h}X(X^\top X)^{-1}\right)\\		\label{eq:fourterms}&\quad\quad\quad\quad+\alpha^2\tr\left(X^\top \Sigma_{h}^{2}X(X^\top X)^{-2}\right)-\frac{3}{4}\alpha^2\tr\left((X^\top \Sigma_{h}X)^2(X^\top X)^{-3}\right).
	\end{align}
	Hereafter, we assume that $X^{(k)} = X$ is known.
	Suppose that 
	\begin{align}
	\label{eq:hypo}
	\tr(X^\top X)\leq 2p\mathbb{E}\Vert \Sigma_{h}\Vert_2,
	\end{align}
	we then have
	\begin{align*}
	\notag	\mathbb{E}\tr(X^{+\top}X^{+})
	&\leq p\left(\frac{1}{2}\alpha^2-2\alpha+2-\frac{1}{2}\alpha^2+\alpha\right)\mathbb{E}\Vert \Sigma_{h}\Vert_2+\alpha^2\mathbb{E}\Vert \Sigma_{h}\Vert^2_2\tr((X^\top X)^{-1})\\
	&\leq p\left(-\alpha+2\right)\mathbb{E}\Vert \Sigma_{h}\Vert_2+\alpha^2p\mathbb{E}\Vert \Sigma_{h}\Vert^2_2\sigma_{\min}^{-2}(X)\\
	&\leq^{(i)} p\left(-\alpha+2\right)\mathbb{E}\Vert \Sigma_{h}\Vert_2+\frac{p}{2}\alpha\mathbb{E}\Vert \Sigma_{h}\Vert_2=p\left(2-\frac{1}{2}\alpha\right)\varphi_{2}<\infty,
	\end{align*}
	where the inequality (i) follows from \cref{eq:stepsize}.

	Given that $\tr(X^{(0)\top}X^{(0)})=p$, the hypothesis \cref{eq:hypo} may not be available. 
	Hence, we turn to the case of $\tr(X^\top X)\geq 2p\mathbb{E}\Vert \Sigma_{h}\Vert_2$, where we have
	\begin{align*}
	\mathbb{E}\tr\left(X^{+\top}X^{+}\right)&\leq\left(\frac{1}{4}\alpha^2-\alpha+1-\frac{1}{4}\alpha^2+\frac{1}{2}\alpha\right)\tr(X^\top X)+\alpha^2\mathbb{E}\Vert \Sigma_{h}\Vert^2_2\tr((X^\top X)^{-1})\\
	&\leq\left(\frac{1}{4}\alpha^2-\alpha+1-\frac{1}{4}\alpha^2+\frac{1}{2}\alpha+\frac{1}{4}\alpha\right)\tr(X^\top X)\leq\tr(X^\top X).
	\end{align*}
	The inequality above shows that $\mathbb{E}\tr\left(X^{+\top}X^{+}\right)$ will not increase, until \cref{eq:hypo} is satisfied. Even though \cref{eq:hypo} might never hold, the SGN iterates could always be bounded in expectation by the initial value due to the non-increasing property. Or more specifically,
	$$\mathbb{E}\tr\left(X^{+\top}X^{+}\right)\leq\max\{p,\ 2p\varphi_2\}.$$

	We then turn to the expectational increment
	\begin{align*}
	\mathbb{E}\tr\left(S^\top S\right) 
	&\leq\mathbb{E}\tr\left((X^\top X)^{-1}X^\top \Sigma_{h}^2X(X^\top X)^{-1} \right)+\frac{1}{4}\tr\left(X^\top X\right)\\
	&\leq\mathbb{E}\Vert\Sigma_h\Vert^2_2\tr((X^\top X)^{-1})+\frac{1}{4}\tr\left(X^\top X\right)\\
	&\leq p\underline{\sigma}^{-2}\varphi_4+\frac{1}{4}\max\{p,\ 2p\varphi_2\}
	<\infty,
	\end{align*}
	which thus completes the proof.
\end{proof}

Then, we are in a position to prove	\cref{lemma:kT1}.
\begin{proof}[Proof of \cref{lemma:kT1}] 
	The infinitesimal mean of $\kvec\left(X(t)\right)$ could be directly obtained from \cref{eq:infinitemean} as 
	\begin{align*}
	&\lim_{\alpha\to 0}\frac{1}{\alpha}\mathbb{E}[\kvec\left(\Delta X_\alpha\right)\vert X_\alpha =X]\\
	=&~\kvec\left(\Sigma X(X^\top X)^{-1}-\frac{1}{2}X-\frac{1}{2}X(X^\top X)^{-1}X^\top\Sigma X(X^\top X)^{-1}\right),
	\end{align*}
	and the infinitesimal variance from \cref{eq:infinitevariance} as 
	\begin{align*}
	\lim_{\alpha\to 0}\frac{1}{\alpha}\mathbb{E}\left[\kvec\left(\Delta X_\alpha\right)\kvec\left(\Delta X_\alpha\right)^\top\vert X_\alpha=X\right]\leq 0=0(=\mathcal{O}(\alpha)),
	\end{align*}		
	where the inequality is due to the result in \cref{lemma:bound2} that $\alpha^{-2}\mathbb{E}\Vert \Delta X_\alpha\Vert_{\text{F}}^2$ is finite, and thus $\mathbb{E}\Vert\kvec\left(\Delta X_\alpha\right)\Vert_{\text{F}}^2=\mathcal{O}(\alpha^2)$. 
	
	Applying Corollary 4.2 in Section 7.4 of \cite{ethier1986markov} then completes the proof.
\end{proof}

\section{Proof of \cref{lemma:kT2}}\label{sec:proofkT2}
\begin{proof}
	As $X_{\alpha}(t)$ lies around some stationary point $\tilde{X}$ associated with the eigen-index set $\mathcal{I}_{p}=\{i_{1},\cdots, i_{p}\}$ and the extended one $\hat{\mathcal{I}}_p$, we impose two additional assumptions on $\tilde{X}$ for simplicity:		
	\begin{itemize}
		\item[(i)] $i_{1}> i_{2}>\cdots > i_{p-1}>i_{p}$;
		\item[(ii)] $\lambda_{i_{1}}, \lambda_{i_{2}}, \cdots, \lambda_{i_{p-1}}$ are single eigenvalues and $\lambda_{i_{p}}$ has multiplicity $(q-p+1)\geq 1$,
	\end{itemize}
	where $q=\vert\hat{\mathcal{I}}_p\vert$. 
	Then, $X_{\alpha}(t)$ could be expressed in the form
	\begin{align*}
	X_{\alpha}(t)=E_{q}\Lambda_{q}^{\frac{1}{2}}WV^\top+\mathcal{O}(\alpha^{\frac{1}{2}}),
	\end{align*}
	where $E_q=[e_{i_{1}}, e_{i_{2}}, \cdots, e_{i_{q}}]\in\mathbb{R}^{n\times q},~\Lambda_{q}=\Diag(\lambda_{i_{1}}, \lambda_{i_{2}}, \cdots, \lambda_{i_{q}})\in\mathbb{R}^{q\times q}$, $V$ is a $p$-dimensional orthogonal matrix, and the weight matrix $W\in\mathbb{R}^{q\times p}$ is defined by
	\begin{equation*} 
	W=
	\left[         
	\begin{array}{cc}   
	I_{p-1}&0_{(p-1)\times 1}\\
	0_{(q-p+1)\times(p-1)}&w
	\end{array}
	\right],\quad w^\top w=1,\quad w\in\mathbb{R}^{q-p+1}.
	\end{equation*}
	Let $
	\Delta Y_\alpha(t) = Y_\alpha (t+\alpha)-Y_\alpha(t)=\alpha^{-\frac{1}{2}}\left(X_\alpha(t+\alpha)-X_\alpha(t)\right)V.
	$
	In the same vein as the proof of \cref{lemma:kT1}, we compute the $(np)$-dimensional infinitesimal mean as
	\begin{align}
	\label{eq:mean2}\lim_{\alpha\to 0}\frac{1}{\alpha}\mathbb{E}\left[\kvec(\Delta Y_\alpha(t))\vert Y_\alpha(t)=Y\right]=~\kvec\left(\Sigma Y\Lambda_{p}^{-1}-Y\right),
	\end{align}
	
	\noindent and the $(np)$-dimensional infinitesimal variance matrix as
	\begin{equation*} 
	\lim_{\alpha\to 0}\frac{1}{\alpha}\mathbb{E}\left[\kvec(\Delta Y_\alpha)\kvec(\Delta Y_\alpha)^\top\vert Y_\alpha(t)=Y\right]=
	\mathbb{E}
	\left[         
	\begin{array}{cccc}   
	S_{(:, 1)}S_{(:, 1)}^\top&\cdots& S_{(:, 1)}S_{(:, p)}^\top\\  
	\cdots &\cdots&\cdots\\  
	S_{(:, p)}S_{(:, 1)}^\top&\cdots& S_{(:, p)}S_{(:, p)}^\top
	\end{array}
	\right],
	\end{equation*}
	where
	\begin{align*}
	S_{(:, j)}=\Sigma_{h}E_{q}\Lambda_{q}^{\frac{1}{2}}&W\Lambda_{p(:, j)}^{-\frac{1}{2}}-\frac{1}{2}E_{q}\Lambda_{q}^{\frac{1}{2}}W_{(:, j)}\\
	&~~~~~~~~-\frac{1}{2}E_{q}\Lambda_{q}^{\frac{1}{2}}W\Lambda_{p}^{-1}W^\top \Lambda_{q}^{\frac{1}{2}}E_{q}^\top \Sigma_{h}E_{q}\Lambda_{q}^{\frac{1}{2}}W\Lambda_{p(:, j)}^{-1},~\quad 1\leq j\leq p,
	\end{align*}
	where $\Sigma_h=\sum_{i=1}^h a_{i}a_{i}^\top/h$ and $a_{i}\in\mathbb{R}^n~(i = 1, \cdots, h)$ are i.i.d. copies of the random vector $a$.
	For any $1\leq l,~r\leq p$, we have
	\begin{align}
	\notag\mathbb{E}\left[S_{(:, l)}S_{(:, r)}^\top\right]&=
	\left(-\frac{1}{2}-\frac{1}{2}+\frac{1}{4}+\frac{1}{4}+\frac{1}{4}\right)T_{1}^{lr}+\left(1-\frac{1}{2}-\frac{1}{2}+\frac{1}{4}\right) \left(1-\frac{1}{h}\right)T_{1}^{lr}\\
	\label{eq:variance2}&\quad~+\frac{1}{h}T_{2}^{lr}-\frac{1}{2h}\left(\tilde{W}T_{2}^{lr}+T_{2}^{lr}\tilde{W}\right)+\frac{1}{4h}\tilde{W}T_{2}^{lr}\tilde{W},
	\end{align}
	where
	\begin{equation*}
	\begin{aligned}
	T^{lr}_{1}=&\left(\lambda_{i_{l}}\lambda_{i_{r}}\right)^{\frac{1}{2}}(E_{q}W)_{(:, l)}\left((E_{q}W)_{(:, r)}\right)^\top,\\
	T^{lr}_{2}=&\left\{
	\begin{array}{rcl}
	&\Sigma+2T^{lr}_{1}+\sum_{j=1}^{\tau}w_{j}^2(\mathbb{E}[b_{i_{p}}^{4}]/\lambda_{i_{p}}-3\lambda_{i_{p}})E_{i_{p}+j-1, i_{p}+j-1}, & {l=r=p,}\\
	&\Sigma+(\mathbb{E}[b_{i_{l}}^{4}]/\lambda_{i_{l}}^2-1)T^{lr}_{1},& {l=r\neq p,}\\
	&(\lambda_{i_{l}}\lambda_{i_{r}})^{\frac{1}{2}}\left(T^{lr}_{1}+T^{rl}_{1}\right),& {otherwise,}
	\end{array} \right.\\
	\tilde{W}=&
	\left[         
	\begin{array}{cc}   
	WW^\top&0_{q\times (n-q)}\\
	0_{(n-q)\times q}&0_{(n-q)\times (n-q)}
	\end{array}
	\right]\in\mathbb{R}^{n\times n},
	\end{aligned}
	\end{equation*}
	and $E_{l, r}$ denotes the $n$-dimensional matrix with its $(l, r)$-th element being one and other entries being zero.
	Using \cref{eq:mean2}, \cref{eq:variance2} and applying Corollary 4.2 in Section 7.4 of \cite{ethier1986markov} finally yield the SDE approximation \cref{eq:ksde}.
	It could be easily verified that even if the additional assumptions (i), (ii) are eliminated, \cref{eq:ksde} is still valid.
\end{proof}
\end{document}